\documentclass[11pt]{article}%{amsart}
\usepackage{dsfont, amssymb,amsmath,amscd,latexsym, amsthm, extarrows, amsxtra,amsfonts}
\usepackage[margin=1in]{geometry}

\usepackage[all]{xy}
\usepackage[active]{srcltx}
\usepackage[pdfstartview=FitH, bookmarksnumbered=true,bookmarksopen=true, colorlinks=true, pdfborder=001, citecolor=blue, linkcolor=blue,urlcolor=blue]{hyperref}
\usepackage[round,authoryear]{natbib}
\usepackage{graphicx}  
\usepackage{float}  
\usepackage{subfigure}
\usepackage{graphicx}
\usepackage[title]{appendix}
\usepackage{bm}
\usepackage{tikz}
\usepackage{footnote}
\makesavenoteenv{equation}   % 让 equation 环境支持脚注
\newtheorem{theorem}{Theorem}[section]

\newtheorem{definition}[theorem]{Definition}

\newtheorem{lemma}[theorem]{Lemma}

\newtheorem{proposition}[theorem]{Proposition}
\newtheorem{remark}[theorem]{Remark}
\numberwithin{equation}{section}

\begin{document}
\makeatletter
\def\@setauthors{%
\begingroup
\def\thanks{\protect\thanks@warning}%
\trivlist \centering\footnotesize \@topsep30\p@\relax
\advance\@topsep by -\baselineskip
\item\relax
\author@andify\authors
\def\\{\protect\linebreak}%
{\authors}%
\ifx\@empty\contribs \else ,\penalty-3 \space \@setcontribs
\@closetoccontribs \fi
\endtrivlist
\endgroup } \makeatother
 %\baselineskip 12pt
%\linenumbers

\title{Equilibria for Time-inconsistent Regular-singular Control Problems with Non-exponential Discount}
\date{}
\author{
Yuting Jia\thanks{
\scriptsize Department of Mathematical Sciences, Tsinghua University, Beijing 100084, China. \texttt{jyt22@mails.tsinghua.edu.cn}}
\and
Zongxia Liang\thanks{\scriptsize 
    Department of Mathematical Sciences, and Center for Insurance and Risk Management, Tsinghua University, Beijing, China. \texttt{liangzongxia@tsinghua.edu.cn}}
 \and
 Xiaodong Luo\thanks{
\scriptsize Department of Applied Mathematics, The Hong Kong Polytechnic University, Kowloon, Hong Kong. \texttt{xiaodong.luo@polyu.edu.hk}}}
\maketitle
\begin{abstract}
This paper studies a time-inconsistent mixed regular-singular control problem in continuous time with non-exponential discount. We seek time-consistent equilibria in an intrapersonal game framework and propose two novel definitions of regular-singular equilibrium where both regular and singular controls are unified via control laws. The standard equilibrium definition allows simultaneous perturbations of both controls, while a mild definition restricts perturbations to one control at a time. Verification theorems providing sufficient conditions are established for both definitions. We illustrate the applicability of the theory through an effort and dividend problem. First, a mild equilibrium is constructed, yet not a standard equilibrium, highlighting the distinction between the two definitions. Moreover, we derive semi-explicit threshold-type equilibria under both a mixture of exponential discount functions and a pseudo-exponential discount function. In both cases, the convexity of the value function is rigorously established, the equilibrium conditions are verified, and the existence of the effort and dividend thresholds is proved. Numerical results reveal the impacts of exogenous parameters on the effort and dividend thresholds and the equilibrium value function, along with their economic implications.
\vskip 15 pt \noindent
\textbf{Keywords:} Time-inconsistency; Regular-singular control law; Non-exponential discount; Equilibrium; Mild equilibrium; Effort and dividend problem.
%\vskip 5pt  \noindent
%\textbf{2010 Mathematics Subject Classification:} 91G10,  91G80, 49L20.
 %\vskip 5pt  \noindent
%\textbf{JEL Classifications:} G11, C61.
\end{abstract}
\vskip15pt

\section{Introduction}\label{Introduction}
Time-inconsistency is a pervasive feature in financial and economic decision-making. It commonly originates from non-exponential discount, mean-variance objectives, and distorted probabilities, etc. In time-inconsistent problems, Bellman's principle of optimality no longer holds, and consequently, a strategy that is deemed optimal today may no longer appear optimal when reconsidered from a future perspective, and conventional dynamic programming techniques are not directly applicable. \cite{strotz1955myopia} proposed both pre-commitment and equilibrium strategies to address the inconsistency. Under the former, the agent formulates a plan that is optimal from today's perspective and commits to it, disregarding any future incentives to deviate ( see, e.g., \cite{BajeuxBesnainou1998}, \cite{li2000optimal}, \cite{richardson1989amin}, \cite{zhou2000continuous}). In the latter approach, the agent's strategies are characterized as the equilibrium outcome of an intrapersonal game, in which the players are the successive temporal selves of the same individual. However, given that insisting on a pre-committed strategy that a future self would find suboptimal is behaviorally untenable, we pursue time-consistent equilibria in this paper.

The intrapersonal game treatment of time inconsistency dates back to \cite{strotz1955myopia}, which treated the successive incarnations of the same decision maker as noncooperative players within a dynamic game and sought Nash subgame-perfect equilibrium points. \cite{Marin2010consum} derived a modified Hamilton--Jacobi--Bellman (HJB) equation to solve an optimal consumption and investment problem with the non-constant discount rate. The precise definition of the equilibrium concept in continuous time was first provided by \cite{ekeland2006being,ekeland2010the} and \cite{ekeland2008invest}. \cite{ekeland2012time} applied the equilibrium framework to the portfolio management problem under hyperbolic discount. Following their definition, \cite{bjork2010general} further advanced the game-theoretic framework by deriving an extended HJB system, along with a corresponding verification theorem, for a broad class of Markovian time-inconsistent stochastic control problems. \cite{bjork2014mean} extended the equilibrium framework to the mean-variance portfolio optimization problem with state-dependent risk aversion. Around the same time, \cite{hu2012time,hu2017time} developed a rigorous equilibrium theory for time-inconsistent stochastic linear-quadratic control problems. \cite{christensen2018on,christensen2020on}, \cite{ebert2020weighted} extend the framework from stochastic control to time-inconsistent optimal stopping problems. Beyond that, the equilibrium framework has been widely applied across a broad range of time-inconsistent problems, including portfolio selection (see, e.g., \cite{dong2014time}, \cite{kryger2010some}), reinsurance-investment problems (see, e.g., \cite{li2015time}, \cite{liang2015time} \cite{lin2016time}), dividend rate problems (see, e.g., \cite{cao2025equilibrium}, \cite{hu2025equilibrium}, \cite{li2016equilibrium}, \cite{ZHAO20141}).

It should be noted that the aforementioned literature is confined to regular control problems,  whereas research on time‑inconsistent singular control has only begun to emerge in recent years. In contrast to regular control, the singular control is characterized as a nondecreasing c\`adl\`ag process. The theoretical foundation of time‑inconsistent singular control was first examined by \cite{liang2024equilibria}, which established an equilibrium framework for singular control via the notion of singular control laws. This framework was subsequently extended by \cite{bodnariu2025time}, which introduced the notion of mild thresholds and generalized the framework for a broader class of candidate equilibrium singular controls. In a parallel strand, \cite{dai2024dynamic} proposed an alternative equilibrium definition for a time‑inconsistent mean-variance portfolio selection problem with transaction costs and characterized the associated trading strategies. Hereafter, the theory of time‑inconsistent singular control has been further developed and brought to bear on various financial problems, such as the singular reinsurance Stackelberg game studied by \cite{liang2025stackelberg} and the singular dividend problem examined by \cite{cao2026equilibrium} under the mean-variance criterion. In very recent work, \cite{caoyue2026equilibrium} investigated an equilibrium singular dividend control problem arising from ambiguity aggregation of heterogeneous discount rates, and \cite{dai2026time} studied equilibrium singular control problems with a running minimum process under non-exponential discount. 

Although the equilibrium theory of time‑inconsistent singular control has advanced considerably in recent years, singular interventions alone cannot adequately address many practical decision problems. In fact, in many real-world settings--including finance, insurance, resource management, and inventory control--decision makers routinely employ both continuous adjustments and singular interventions, which naturally gives rise to mixed regular-singular control problems (see, e.g., \cite{chen2021free}, \cite{guan2023dynamic}, \cite{guo2004constrained}, \cite{jgaard1999controlling}, \cite{qiu2023optimal}, \cite{taksar1998optimal}, among many others). However, the existing literature remains largely restricted to time‑consistent settings, and a general equilibrium framework for mixed regular-singular control problems under non-exponential discount has yet to be established. Most recently, \cite{liang2026timeconsistent} and \cite{liang2026equilibrium} investigate a class of regular-singular control problems under the mean-variance criterion. This paper works within the same equilibrium framework of \cite{liang2026equilibrium}, but considers a different objective: instead of the mean-variance criterion, we analyze mixed regular-singular control under non-exponential discount. Non-exponential discount, as one of the major causes of time inconsistency in the literature (see, e.g.,  \cite{ekeland2008invest}, \cite{harris2001the}, among many others), is also worthy of investigation. This paper aims to fill this gap.

 In this paper, we study a time-inconsistent mixed regular-singular control problem in continuous time with non-exponential discount and seek time-consistent equilibria in an intrapersonal game framework. The main contributions of this paper are summarized as follows. 

 First, we establish a closed-loop equilibrium framework for the general class of time-inconsistent regular-singular stochastic control problems in continuous time with non-exponential discount, based on a unified control law formulation. The drift and volatility coefficients of the controlled state process are jointly influenced by both regular and singular controls. Building upon the frameworks of \cite{bjork2017timeinconsistent} and \cite{liang2024equilibria}, both the regular and singular controls are simultaneously and jointly generated by their corresponding control laws through the Skorokhod reflection problem. Within our framework, control laws are equivalently defined as mappings on the augmented state space: the regular control law maps into the regular control space, whereas the singular control law is equivalently a mapping into the binary-valued set of ``waiting'' and ``intervening''. Because the regular and singular controls jointly affect both the drift and volatility coefficients of the controlled state process, they are intrinsically coupled; the equilibrium framework therefore perturbs them simultaneously. This approach stands in contrast to \cite{liang2023weak}, in which the regular control and the stopping time are perturbed separately in a mixed control problem of regular control and stopping. Similar to \cite{liang2024equilibria}, the perturbation we introduce is of an open-loop form. The corresponding equilibrium definition is presented in Definition \ref{equil}.

 Second, we establish a verification theorem \ref{verificationthm} that provides a sufficient condition for an admissible regular-singular control law to constitute an equilibrium. Specifically, under suitable regularity and integrability conditions, any admissible control law derived from a solution to the extended HJB system qualifies as an equilibrium. Theorem \ref{verificationthm} thus bridges the gap between the abstract equilibrium definition and its practical implementation.
 
 Furthermore, parallel to \cite{liang2023weak}, we propose another novel definition of mild equilibrium (see Definition \ref{mildequil}), in which perturbations are only allowed to be applied separately to either the regular or the singular control. This alternative equilibrium notion accommodates scenarios where the decision maker may deviate in only one type of control at a time, rather than both simultaneously. We also establish the corresponding mild verification theorem (Theorem \ref{mildverificationthm}), which provides a sufficient condition for an admissible control law to constitute a mild equilibrium. This additional equilibrium concept enriches our theoretical framework and accommodates a broader spectrum of time-inconsistent decision-making scenarios. Together with the standard equilibrium framework, this mild equilibrium concept offers greater flexibility in characterizing time-consistent strategies across different perturbation structures.

Finally, we illustrate the applicability of our theoretical framework through the effort and dividend problem. The firm controls the effort and dividend strategies to minimize the expected total net cost of running the business until failure. A mild equilibrium is constructed, yet not a standard equilibrium, highlighting the distinction between the two definitions. Solving the extended HJB system, we obtain semi-explicit equilibrium characterized by the effort and dividend thresholds under two representative discount structures: a mixture of exponential discount functions and a pseudo-exponential discount function. Note that for the time-consistent counterpart, the convexity of the value function follows directly from its optimality; see, e.g.,  \cite[Proposition 1.1]{jgaard1999controlling}. The same conclusion, however, does not apply to the equilibrium value function. For the two discount functions considered in this paper, the convexity of the value function is rigorously established, the equilibrium conditions are verified, and the existence of the effort and dividend thresholds is proved. The equilibrium policy prescribes exerting the maximal level of effort when the surplus falls below the effort threshold and paying dividends once the surplus exceeds the dividend threshold when the cost for dividend is relatively small; otherwise, when the cost for dividend is relatively large, the equilibrium policy is to pay dividends once the surplus exceeds the dividend threshold and not to apply effort even if the firm is at the ruin time. The closed-form solutions validate the feasibility of our theoretical framework and provide economically intuitive insights for dynamic corporate finance decisions involving effort allocation and shareholder payouts. Numerical analysis further quantifies the effects of key exogenous parameters on the thresholds and the equilibrium value function under the pseudo-exponential discount function, thereby shedding light on their economic implications.

The remainder of the paper is organized as follows. Section \ref{Model formulation} formulates the model of time-inconsistent regular-singular control problems with non-exponential discount and introduces the equilibrium framework. Section \ref{Verification theorem} establishes a verification theorem \ref{verificationthm}. Section \ref{mild} provides the concept of mild equilibrium and gives the mild verification theorem \ref{mildverificationthm}. Section \ref{Application} applies the theory to the effort and dividend problem. Section \ref{Numerical} presents and discusses some numerical results and sensitivity analysis. The last section concludes this paper.

\section{Model formulation}
\label{Model formulation}
Let $\left(\Omega, \mathcal{F},\left\{\mathcal{F}_{t}\right\}_{t\ge 0}, \mathbb{P}\right)$ be a filtered complete probability space satisfying the usual conditions and $B=\left\{B_t\right\}_{t\ge 0}$ be a standard one-dimensional Brownian motion on this space. 

For any fixed $(x,t,y)\in \mathcal{Q}:=[0,+\infty)\times[0,+\infty)\times[0,+\infty)$, the controlled state process $X^{\pi,\xi}$ under the regular control $\pi$ and the singular control $\xi$ follows the general stochastic differential equation (abbr. SDE):
\begin{equation}
\left\{\begin{array}{l}
\mathrm{d}X^{\pi,\xi}_{r}=\mu\big(X^{\pi,\xi}_{r},r,\pi_{r},\xi_{r}\big)\mathrm{d}r+\sigma\big(X^{\pi,\xi}_{r},r,\pi_{r},\xi_{r}\big)\mathrm{d}B_{r}-\mathrm{d}\xi_{r},\ r\in[t,\tau^{\pi,\xi}],\footnotemark[1]\\
X^{\pi,\xi}_{t-}=x,
\end{array}\right.\label{sde-regular}
\end{equation}
\footnotetext[1]{Because we only consider non-negative state process $X^{\pi,\xi}$, we restrict the time horizon up to the bankruptcy time $\tau^{\pi,\xi}$. In this paper, for any stopping time $\tau$, when $\tau=+\infty$, $[a,\tau]$ represents $[a,+\infty)$ for any $a\in\mathbb{R}$.}
where $\pi=\left\{\pi_r\right\}_{r\ge t}$ is the regular control, $\xi=\left\{\xi_r\right\}_{r\ge t}$ is the singular control which is nondecreasing and c\`adl\`ag with initial $\xi_{t-}=y$, and $\Delta \xi_r:=\xi_r-\xi_{r-}\le X_{r-}^{\pi,\xi}$. Denote by $\tau^{\pi,\xi}:=\tau^{x,t,y;\pi,\xi}_{t}:=\inf\big\{r\ge t\bigm|X^{\pi,\xi}_{r}= 0\big\}$ the time of bankruptcy, where the initial state $(x,t,y)$ is omitted from the notation $\tau^{x,t,y;\pi,\xi}_{t}$ for simplicity (and similarly for other bankruptcy times throughout this paper). The drift function $\mu$ and the volatility function $\sigma$ are exogenously given and deterministic continuous functions on $[0,+\infty)\times[0,+\infty)\times U\times[0,+\infty)$, where compact set $U\subset\mathbb{R}$ is the regular control space.%U Borel-measurable, U compact 为了验证定理过得去，例子也是满足，不过也可以加上扰动过程u右连续（没必要右连左极），但是我们的例子的均衡策略不是右连续的，而且不想加扰动过程u右连续的条件。也可以对于H关于\pi变量的性质加条件。

The optimization problem is to find a regular-singular control $(\pi,\xi)$ that minimizes the objective
{\small
\begin{equation}
\begin{aligned}
J(x, t, y ; \pi,\xi):=&\mathbb{E}_{x, t, y}\left[\int_t^{\tau^{\pi,\xi}} \beta(r-t) H\big(X_r^{\pi,\xi}, r, \pi_r, \xi_r \big) \mathrm{d}r +\int_t^{\tau^{\pi,\xi}} \beta(r-t) c\big(X_{r-}^{\pi,\xi}, r, \xi_{r-}\big) \mathrm{d} \xi_r\right],
\end{aligned}\nonumber
\end{equation}}
where $\mathbb{E}_{x, t, y}$ denotes the expectation conditioning on $X^{\pi,\xi}_{t-}=x$ and $\xi_{t-}=y$, and 
\begin{eqnarray*}
 \int_t^{\tau^{\pi,\xi}} \beta(r-t) c\big(X_{r-}^{\pi,\xi}, r, \xi_{r-}\big) \mathrm{d} \xi_r&=&\int_t^{\tau^{\pi,\xi}} \beta(r-t) c\big(X_{r}^{\pi,\xi}, r, \xi_{r}\big) \mathrm{d} \xi_r^{c}\\
&&+\sum\limits_{r\in[t,\tau^{\pi,\xi}]} \beta(r-t) c\big(X_{r-}^{\pi,\xi}, r, \xi_{r-}\big) \Delta \xi_r,
\end{eqnarray*}
$\xi_r^{c}$ and $\Delta \xi_r$ are  the continuous and discrete parts of $\xi_r$, respectively. $\beta$ is a general discount function. Furthermore, we assume that the deterministic functions $H$, $c$ and $\beta$ are exogenously given continuous functions  bounded from below. In addition, we assume $c\in C^{1,0,1}(\mathcal{Q})$ and that $\beta\in C([0,+\infty))$ is non-negative and nonincreasing with $\beta(0)=1$. In addition, the cost rate $c$ should be assumed to stay the same during the jump, i.e., 
\begin{equation}c(x-a,t,y+a)=c(x,t,y),\ \forall (x,t,y)\in \mathcal{Q}, 0\leq a\leq x,\label{cxy}\end{equation} or, equivalently, $c_x(x,t,y)=c_y(x,t,y), \forall (x,t,y)\in \mathcal{Q}$.

As the discount function might not be exponential, time-inconsistency is involved. In the sense of closed-loop control, we introduce the definition of regular-singular control law. The regular control process $\pi$ is generated by some regular control law $\Pi=\Pi(\cdot,\cdot,\cdot)$ as a feedback function of the state-time-(singular)control triple as \cite{bjork2017timeinconsistent} and \cite{bjork2021time}, while the singular control process $\xi$ is generated by some singular control law $\Xi$ as a division of the space $\mathcal{Q}$ of state-time-(singular)control triple as \cite{liang2024equilibria} and \cite{liang2025stackelberg}. $\Xi$ is equivalently a mapping from $\mathcal{Q}$ into the binary-set of ``waiting'' and ``intervening''. Hence, the admissible regular-singular control law is defined as follows:
\begin{definition}[Admissible regular-singular control law]\label{def1}
Let $\Pi=\Pi(\cdot,\cdot,\cdot)$ be a $U$-valued Borel-measurable function on $\mathcal{Q}$, and $\Xi=\left(W^{\Xi},P^{\Xi}\right)$ be a Borel-measurable division of $\mathcal{Q}$ with $W^{\Xi}$ denoting the waiting region and $P^{\Xi}=\left(W^{\Xi}\right)^{c}$ denoting the action region (the complement of the waiting region). The regular-singular control law $\left(\Pi,\Xi\right)$ is admissible if the following conditions hold:\\
(a). Given any initial $(x,t,y)\in \mathcal{Q}$, the Skorokhod reflection type SDE
{\small 
\begin{equation}
\!\!\!\left\{\begin{array}{l}
\!\mathrm{d}X^{\Pi,\Xi}_{r}\!\!=\!\mu\!\big(\!X^{\Pi,\Xi}_{r}\!,r,\!\Pi\!\big(\!X^{\Pi,\Xi}_{r}\!,r,\xi^{\Pi,\Xi}_{r}\big)\!,\xi^{\Pi,\Xi}_{r}\big)\!\mathrm{d}r\!+\!\sigma\!\big(\!X^{\Pi,\Xi}_{r}\!,r,\!\Pi\!\big(\!X^{\Pi,\Xi}_{r}\!,r,\xi^{\Pi,\Xi}_{r}\big)\!,\xi^{\Pi,\Xi}_{r}\big)\!\mathrm{d}B_{r}\!\!-\!\mathrm{d}\xi^{\Pi,\Xi}_{r}\!\!,\ r\!\in\![t,\!\tau^{\Pi,\Xi}],\\
\!\big(\!X^{\Pi,\Xi}_{r}\!,r,\xi^{\Pi,\Xi}_{r}\big)\!\in\! \overline{W^{\Xi}},\ r\!\in\![t,\!\tau^{\Pi,\Xi}),\\
\!\xi^{\Pi,\Xi}_{r}\!\!=\!y\!+\!\int_{t}^{r}\!1_{\big\{\big(\!X^{\Pi,\Xi}_{l}\!,l,\xi^{\Pi,\Xi}_{l}\big)\in P^{\Xi}\big\}}\!\mathrm{d}\xi^{\Pi,\Xi}_{l}\!\!,\ r\!\in\![t,\!\tau^{\Pi,\Xi}],\\
\!X^{\Pi,\Xi}_{t-}=x
\end{array}\right.\nonumber
\end{equation}}
has a unique strong solution $\big(X^{x,t,y;\Pi,\Xi},\xi^{x,t,y;\Pi,\Xi}\big):=\big(\big\{X^{\Pi,\Xi}_{r}\big\}_{r\in [t,\tau^{\Pi,\Xi}]},\big\{\xi^{\Pi,\Xi}_{r}\big\}_{r\in [t,\tau^{\Pi,\Xi}]}\big)$ which satisfies $(\hat{X}_{r-}-a,r,\hat{\xi}_{r-}+a)\in P^{\hat{\Xi}}$  for any $r\in[t,\tau^{\hat{\Pi},\hat{\Xi}}]$ such that $\Delta\hat{\xi}_{r}>0$ and any $a\in[0,\Delta\hat{\xi}_{r}]$, where $\tau^{\Pi,\Xi}:=\inf\big\{r\ge t\bigm|X^{x,t,y;\Pi,\Xi}_{r}= 0\big\}$. The process $\pi^{x,t,y;\Pi,\Xi}:=\big\{\Pi\big(X^{\Pi,\Xi}_{r},r,\xi^{\Pi,\Xi}\big)\big\}_{r\in [t,\tau^{\Pi,\Xi}]}$ and $\xi^{x,t,y;\Pi,\Xi}$ are respectively the regular control process and singular control process simultaneously generated by the pair of control law $(\Pi,\Xi)$ at initial $(x,t,y)\in\mathcal{Q}$.\\
%这个条件也不能保证扰动后存在唯一强解，主要是扰动的区间内强解存在唯一不能保证，所以去掉这个条件(b). $X^{x,t,y;\Pi,\Xi}$ is the unique strong solution for the SDE \eqref{sde-regular} with regular-singular control pair $\left(\pi^{x,t,y;\Pi,\Xi},\xi^{x,t,y;\Pi,\Xi}\right)$.\\
(b). The strong solution $\big(X^{x,t,y;\Pi,\Xi},\xi^{x,t,y;\Pi,\Xi}\big)$ given in (a) satisfies 
\begin{equation}
\begin{aligned}
&\mathbb{E}_{x, t, y}\left[\int_t^{\tau^{\Pi,\Xi}} \beta(r-t) \big|H\big(X^{x,t,y;\Pi,\Xi}_r, r, \pi^{x,t,y;\Pi,\Xi}_r, \xi^{x,t,y;\Pi,\Xi}_r \big)\big| \mathrm{d} r\right]<+\infty, \\
&\mathbb{E}_{x, t, y}\left[\int_t^{\tau^{\Pi,\Xi}} \beta(r-t) \big|c\big(X^{x,t,y;\Pi,\Xi}_{r-}, r, \xi^{x,t,y;\Pi,\Xi}_{r-}\big)\big| \mathrm{d} \xi^{x,t,y;\Pi,\Xi}_r\right]<+\infty.
\end{aligned}\nonumber
\end{equation}
\end{definition}
The objective of an admissible regular-singular control law $(\Pi,\Xi)$ is then 
{\small
\begin{equation}
\begin{aligned}
J(x,t,y;\Pi,\Xi):=&\mathbb{E}_{x,t,y}\left[\int_t^{\tau^{\Pi,\Xi}} \beta(r-t) H\big(X^{x,t,y;\Pi,\Xi}_r, r, \pi^{x,t,y;\Pi,\Xi}_r, \xi^{x,t,y;\Pi,\Xi}_r \big) \mathrm{d} r\right.\\
& \left.+\int_t^{\tau^{\Pi,\Xi}} \beta(r-t) c\big(X^{x,t,y;\Pi,\Xi}_{r-}, r, \xi^{x,t,y;\Pi,\Xi}_{r-}\big) \mathrm{d} \xi^{x,t,y;\Pi,\Xi}_r \right],\ \forall (x,t,y)\in\mathcal{Q}.
\end{aligned}\label{objective}
\end{equation} }

Before introducing the equilibrium, we introduce the admissible set of open-loop controls for perturbation:
\begin{definition}[Admissible regular-singular control pair]\label{admissible2}
Fix initial $(x,t,y)\in\mathcal{Q}$. Let $\pi=\{\pi_{r}\}_{r\ge t}$ be a $U$-valued $\{\mathcal{F}_{r}\}_{r\ge t}$-progressively measurable process and $\xi=\{\xi_{r}\}_{r\ge t}$ be a nondecreasing c\`adl\`ag $\{\mathcal{F}_{r}\}_{r\ge t}$-adapted process. Then $(\pi,\xi)$ is called an admissible regular-singular control pair at $(x,t,y)$ if the followings hold:\\
(a). $\xi_{t-}=y$, $\Delta \xi_r=\xi_r-\xi_{r-}\le X_{r-}^{\pi,\xi}$, and the SDE \eqref{sde-regular} has a unique strong solution $X^{\pi,\xi}=\{X^{\pi,\xi}_{r}\}_{r\in [t,\tau^{\pi,\xi}]}$.\\
(b). The strong solution $X^{\pi,\xi}$ in (a) and $(\pi,\xi)$ satisfies
\begin{equation}
\begin{aligned}
&\mathbb{E}_{x, t, y}\left[\int_t^{\tau^{\pi,\xi}} \beta(r-t) \big|H\big(X_r^{\pi,\xi}, r, \pi_r, \xi_r \big)\big| \mathrm{d} r\right]<+\infty,\\
&\mathbb{E}_{x, t, y}\left[\int_t^{\tau^{\pi,\xi}} \beta(r-t) \big|c\big(X_{r-}^{\pi,\xi}, r, \xi_{r-}\big)\big| \mathrm{d} \xi_r\right]<+\infty.
\end{aligned}\nonumber
\end{equation}
Denote the set of all admissible regular-singular control pairs at $(x,t,y)$ by $\mathcal{D}^{x,t,y}$. 

For any fixed admissible regular-singular control law $(\hat\Pi, \hat\Xi)$, given $(u,\eta)\in\mathcal{D}^{x,t,y}$, if:\\
(i). There exist constants $\tilde{h}>0$ and $\tilde{M}>0$ such that $\eta$ additionally satisfies
\begin{equation}
\label{growthcond}
\eta_{(t+h)-}-\eta_{t}\le \tilde{M}h,\ \text{a.s.},\ \forall h\in(0,\tilde{h}).
\end{equation}
(ii). Under the perturbed regular-singular control pair $(\pi^{h},\xi^{h})$, $\forall h\in(0,\tilde{h})$ defined by
\footnotetext[2]{In this paper, if $\tau^{u,\eta}<t+h$, $[t,\tau^{u,\eta}\wedge((t+h)-)]$ refers to $[t,\tau^{u,\eta}]$ and $X_{\tau^{u,\eta}\wedge((t+h)-)}$ refers to $X_{\tau^{u,\eta}}$; and if $\tau^{u,\eta}\ge t+h$, $[t,\tau^{u,\eta}\wedge((t+h)-)]$ refers to $[t,t+h)$ and $X_{\tau^{u,\eta}\wedge((t+h)-)}$ refers to $X_{(t+h)-}$.}
\footnotetext[3]{Similarly, when $\tau^{u,\eta}<t+h$, $(\tau^{u,\eta}\wedge((t+h)-),\tau^{\pi^h,\xi^h}]$ refers to $(\tau^{u,\eta},\tau^{\pi^h,\xi^h}]=(\tau^{u,\eta},\tau^{u,\eta}]=\emptyset$, and when $\tau^{u,\eta}\ge t+h$, $(\tau^{u,\eta}\wedge((t+h)-),\tau^{\pi^h,\xi^h}]$ refers to $[t+h,\tau^{\pi^h,\xi^h}]$.}
\begin{equation}
\begin{aligned}
&\pi^{h}_{r}:=\begin{cases}
u_{r}, & r\in[t,\tau^{u,\eta}\wedge((t+h)-)],\footnotemark[2]\\
\pi^{X^{u,\eta}_{\tau^{u,\eta}\wedge((t+h)-)},\tau^{u,\eta}\wedge(t+h),\eta_{\tau^{u,\eta}\wedge((t+h)-)};\hat{\Pi},\hat{\Xi}}_{r}, & r\in(\tau^{u,\eta}\wedge((t+h)-),\tau^{\pi^h,\xi^h}],\footnotemark[3]
\end{cases}\\
&\xi^{h}_{r}:=\begin{cases}
\eta_{r}, & r\in[t,\tau^{u,\eta}\wedge((t+h)-)],\\
\xi^{X^{u,\eta}_{\tau^{u,\eta}\wedge((t+h)-)},\tau^{u,\eta}\wedge(t+h),\eta_{\tau^{u,\eta}\wedge((t+h)-)};\hat{\Pi},\hat{\Xi}}_{r}, &  r\in(\tau^{u,\eta}\wedge((t+h)-),\tau^{\pi^h,\xi^h}],
\end{cases}
\end{aligned}\label{pihxih}
\end{equation}
the corresponding state process
\begin{equation}
\begin{aligned}
X^{\pi^{h},\xi^{h}}_{r}=\begin{cases}
X^{u,\eta}_{r}, & r\in[t,\tau^{u,\eta}\wedge((t+h)-)],\\
X^{X^{u,\eta}_{\tau^{u,\eta}\wedge((t+h)-)},\tau^{u,\eta}\wedge(t+h),\eta_{\tau^{u,\eta}\wedge((t+h)-)};\hat{\Pi},\hat{\Xi}}_{r}, &  r\in(\tau^{u,\eta}\wedge((t+h)-),\tau^{\pi^h,\xi^h}],
\end{cases}
\end{aligned}\label{xpihxih}
\end{equation}
satisfies integrability condition (b) where $\tau^{\pi^h,\xi^h}:=\inf\big\{r\ge t\bigm|X^{\pi^h,\xi^h}_{r}= 0\big\}$.\\
Then we call $(u,\eta)$ an admissible pair of perturbations at $(x,t,y)$  w.r.t. $(\hat\Pi, \hat\Xi)$. We denote the set of all admissible pairs of perturbations at $(x,t,y)$  w.r.t. $(\hat\Pi, \hat\Xi)$ by $\mathcal{\tilde D}^{x,t,y;\hat\Pi,\hat\Xi}$.
\end{definition}

\begin{remark}
The set $\mathcal{D}^{x,t,y}$ can be regarded as the set of admissible open-loop controls. Growth condition similar to \eqref{growthcond} is standard in the literature such as \cite{liang2025stackelberg} on time-inconsistent singular control.
\end{remark}

\begin{remark}\label{wellposed}
$(\pi^{h},\xi^{h}), \forall h>0$ in \eqref{pihxih} is not necessarily an admissible regular-singular control pair in $\mathcal{D}^{x,t,y}$ as  the SDE \eqref{sde-regular} has a strong solution \eqref{xpihxih} with $(\pi,\xi)$ replaced by $(\pi^{h},\xi^{h})$, but the strong solution \eqref{xpihxih} is not necessarily unique. However, \eqref{xpihxih} is well defined by Condition (a) in Definition \ref{def1} and Condition (a) in Definition \ref{admissible2}. Besides, the integrability condition (b) in Definition \ref{admissible2} may fail to hold for \eqref{pihxih} and \eqref{xpihxih}. Therefore, we add the Condition (b) in the definition of $\mathcal{\tilde D}^{x,t,y;\hat\Pi,\hat\Xi}$ to make $J(x,t,y;\pi^{h},\xi^{h})$ well defined.
\end{remark}

\begin{remark}
We hope that  the perturbed regular-singular control pair satisfies $\tau^{\pi^h,\eta^h}=\tau^{u,\eta}$ when $\tau^{u,\eta}<t+h$. Therefore, we define the perturbed regular-singular control pair by \eqref{pihxih} rather than
\begin{equation}
\begin{aligned}
\tilde\xi^{h}_{r}:=\begin{cases}
\eta_{r}, & r\in[t,\tau^{u,\eta}\wedge(t+h)),\\
\xi^{X^{u,\eta}_{(\tau^{u,\eta}\wedge(t+h))-},\tau^{u,\eta}\wedge(t+h),\eta_{(\tau^{u,\eta}\wedge(t+h))-};\hat{\Pi},\hat{\Xi}}_{r}, & r\in[\tau^{u,\eta}\wedge(t+h),\tau^{\tilde\pi^h,\tilde\xi^h}],
\end{cases}
\end{aligned}\label{tildepih}
\end{equation}
and $\tilde\pi^{h}$ in a similar way, although \eqref{tildepih} is relatively concise. The reasonable perturbed singular control such that $\xi^h_t=y+x$ may not be included in \eqref{tildepih}. Besides, unlike \cite{cao2026equilibrium}, in the first stage we used $\tau^{u,\eta}$ instead of the implicit $\tau^{\pi^h,\eta^h}$. This definition is more explicit and natural. 
\end{remark}

We introduce the equilibrium for the time-inconsistent regular-singular control problem \eqref{objective} in the weak equilibrium sense: 
\begin{definition}[Equilibrium regular-singular control law]
\label{equil}
An admissible regular-singular control law $(\hat{\Pi},\hat{\Xi})$ is called an equilibrium for Problem \eqref{objective} if for any initial $(x,t,y)\in\mathcal{Q}$, and for any admissible pair of perturbations $(u,\eta)\in\mathcal{\tilde D}^{x,t,y;\hat\Pi,\hat\Xi}$, it holds 
\begin{equation}
\label{equicond}
\liminf\limits_{h\downarrow 0}\frac{J\big(x,t,y;\pi^{h},\xi^{h}\big)-J\big(x,t,y;\hat{\pi},\hat{\xi}\big)}{h}\ge 0,
\end{equation}
where the regular-singular control pair $(\hat{\pi},\hat{\xi})$ is given by 
\begin{equation}\nonumber\hat{\pi}:=\pi^{x,t,y;\hat{\Pi},\hat{\Xi}}, \ \hat{\xi}:=\xi^{x,t,y;\hat{\Pi},\hat{\Xi}},\end{equation}
and the perturbed regular-singular control pair $(\pi^{h},\xi^{h})$ is defined in \eqref{pihxih}.

 Corresponding to the equilibrium regular-singular control law $(\hat{\Pi},\hat{\Xi})$, we define the equilibrium value function $V$ by $V(x,t,y):=J\big(x,t,y;\hat{\Pi},\hat{\Xi}\big)=J\big(x,t,y;\hat{\pi},\hat{\xi}\big)$.
\end{definition}

\begin{remark}
The generation of the control process $\pi^{x,t,y;\Pi,\Xi}$ and $\xi^{x,t,y;\Pi,\Xi}$ is simultaneous and the impact from $(\Pi,\Xi)$ is coupled. Therefore, the definition here takes a different form from that in \cite{liang2023weak}, where the equilibrium conditions are listed separately for stopping and regular control. We also introduce a parallel definition \ref{mildequil} in Section \ref{mild}.
\end{remark}

\section{Verification Theorem}\label{Verification theorem}
In this section, we establish a verification theorem (Theorem \ref{verificationthm}) to provide a sufficient condition of the equilibrium regular-singular control law.
\begin{theorem}[Verification Theorem]\label{verificationthm}
Given a function $V(x,t,y)$ and a family of functions $\left\{f^{s}(x,t,y)\right\}_{s\in[0,+\infty)}$. Define $f(x,t,y,s):=f^{s}(x,t,y)$. Define the infinitesimal operator
\begin{equation}
\begin{aligned}
&\big(\mathcal{A}^{u}\varphi\big)(x,t,y):=\varphi_{t}(x,t,y)+\mu(x,t,u,y)\varphi_{x}(x,t,y)+\frac{1}{2}\sigma^{2}(x,t,u,y)\varphi_{xx}(x,t,y),\ \forall\varphi:\mathcal{Q}\rightarrow\mathbb{R},\\
&\big(\mathcal{A}^{\Pi}\varphi\big)(x,t,y):=\big(\mathcal{A}^{\Pi(x,t,y)}\varphi\big)(x,t,y),\ \forall\varphi:\mathcal{Q}\rightarrow\mathbb{R}.
\end{aligned}\nonumber
\end{equation} 
Assume  that the following conditions hold:\\
(1).
\begin{equation}
\begin{aligned}
 V \in C^{2,1,1}\left(\mathcal{Q}\right), \
 f \in C^{2,1,1,1}\left(\big\{(x,t,y,s)\bigm| (x,t,y)\in\mathcal{Q}, s\in [0,t]\big\}\right).
\end{aligned}\nonumber
\end{equation}
(2). There exists an admissible regular-singular control law $(\hat{\Pi},\hat{\Xi})$ such that the singular control law $\hat{\Xi}$ is given by
\begin{equation}
\begin{aligned}
&W^{\hat{\Xi}}:=\big\{(x,t,y)\in\mathcal{Q}\bigm|c(x,t,y)-V_{x}(x,t,y)+V_{y}(x,t,y)>0\big\},\\
&P^{\hat{\Xi}}:=\big\{(x,t,y)\in\mathcal{Q}\bigm|c(x,t,y)-V_{x}(x,t,y)+V_{y}(x,t,y)=0\big\},
\end{aligned}\nonumber
\end{equation}
and the regular control law $\hat{\Pi}$ satisfies
\begin{equation}
\hat{\Pi}(x,t,y)\in\mathop{\mathrm{argmin}}\limits_{u\in U}\left\{\big(\mathcal{A}^{u}V\big)(x,t,y)+H(x,t,u,y)-f_s(x,t,y,t)\right\},\ \forall (x,t,y)\in\mathcal{Q},
\nonumber
\end{equation}
$\hat{\Pi}$ is continuous on $\overline{W^{\hat{\Xi}}}-W^{\hat{\Xi}}$ when $\hat{\Pi}$ is regarded as a function on $\overline{W^{\hat{\Xi}}}$. \\
(3). $V(x,t,y)$ and $f^{s}(x,t,y)$ satisfy

\begin{subequations}
\begin{align}
&\begin{aligned}
\min\Big\{&\big(\mathcal{A}^{\hat{\Pi}}V\big)(x,t,y)+H\big(x,t,\hat{\Pi}(x,t,y),y\big)-f_s(x,t,y,t),\\
&c(x,t,y)-V_{x}(x,t,y)+V_{y}(x,t,y)\Big\}=0,\ \forall (x,t,y)\in\mathcal{Q},
\end{aligned}\label{v}\\
&V(0,t,y)=0,\ \forall(0,t,y)\in\mathcal{Q},\label{vT}\\
&\big(\mathcal{A}^{\hat{\Pi}}f^{s}\big)(x,t,y)+\beta(t-s)H\big(x,t,\hat\Pi(x,t,y),y\big)=0,\ \forall (x,t,y)\in \overline{W^{\hat{\Xi}}}, s\in[0,t],\label{fw}\\
& \beta(t-s)c(x,t,y)-f^{s}_{x}(x,t,y)+f^{s}_{y}(x,t,y)=0,\ \forall (x,t,y)\in P^{\hat{\Xi}}, s\in[0,t],\label{fp}\\
&f^{s}(0,t,y)=0,\ \forall (0,t,y)\in\mathcal{Q}, s\in[0,t].\label{fT}
\end{align}
\end{subequations}
(4). For functions $\varphi=f^{s},V,(x,t,y)\mapsto f(x,t,y,t)$, for any $(x,t,y)\in\mathcal{Q}$ with $s\in[0,t]$, and any $n\geq t$, it holds that 
\begin{equation}
\mathbb{E}_{x,t,y}\int_{t}^{\tau_n}\big[\varphi_{x}(\hat{X}_{r},r,\hat\xi_{r})\sigma(\hat{X}_{r},r,\hat\pi_r,\hat\xi_{r})\big]^{2}\mathrm{d}r<+\infty,
\label{martingale}\end{equation}
where $(\hat{\pi},\hat{\xi}):=(\pi^{x,t,y;\hat{\Pi},\hat{\Xi}},\xi^{x,t,y;\hat{\Pi},\hat{\Xi}})$ is the regular-singular control process generated by $(\hat{\Pi},\hat{\Xi})$ at $(x,t,y)$, $\hat{X}:=X^{x,t,y;\hat{\Pi},\hat{\Xi}}$ is the corresponding state process, and $\tau_{n}:=\tau^{\hat{\Pi},\hat{\Xi}}\wedge n$.\\
(5). For any $(x,t,y)\in\mathcal{Q}$, any $s\in[0,t]$, it holds that 

\begin{align}
&\lim\limits_{n\rightarrow+\infty}\mathbb{E}_{x,t,y}f^{s}\big(\hat{X}_{\tau_{n}},\tau_{n},\hat{\xi}_{\tau_{n}}\big)=0,\label{finfty1}\\
&\lim\limits_{n\rightarrow+\infty}\mathbb{E}_{x,t,y}\big[V\big(\hat{X}_{\tau_{n}},\tau_{n},\hat{\xi}_{\tau_{n}}\big)-f\big(\hat{X}_{\tau_{n}},\tau_{n},\hat{\xi}_{\tau_{n}},\tau_{n}\big)\big]=0.\label{finfty2}
\end{align}
(6). For any $(x,t,y)\in\mathcal{Q}$, any admissible pair of perturbations $(u,\eta)\in\mathcal{\tilde D}^{x,t,y;\hat\Pi,\hat\Xi}$, there exists a constant $h_0>0$ such that

\begin{align}
&\mathbb{E}_{x,t,y}\int_{t}^{\tau^{u,\eta}\wedge(t+h_0)}\big[V_{x}\big(X^{u,\eta}_{r},r,\eta_{r}\big)\sigma\big(X^{u,\eta}_{r},r,u_r,\eta_{r}\big)\big]^{2}\mathrm{d}r<+\infty,\label{martingale2}\\
&\begin{aligned}\mathbb{E}_{x,t,y}\Big[&\sup\limits_{\substack{r\in(t,\tau^{u,\eta}\wedge(t+h_{0}))\\ h\in(0,h_{0})}}\Big|f_s\big(X^{u,\eta}_{r},r,\eta_{r},r\big)+(\beta(r\!-\!t)\!-\!1) H\big(X^{u,\eta}_r, r,u_r, \eta_r \big)\\
&-f_s\big(X^{u,\eta}_{\tau^{u,\eta}\wedge((t+h)-)},\tau^{u,\eta}\!\wedge\!(t\!+\!h),\eta_{\tau^{u,\eta}\wedge((t+h)-)},r\big)\Big|\Big]\!<\!+\!\infty,\!
\end{aligned}\label{integrability}\\
&\mathbb{E}_{x,t,y}\Big[\sup\limits_{r\in(t,\tau^{u,\eta}\wedge(t+h_{0}))}\Big|(\beta(r-t)-1)c\big(X^{u,\eta}_{r-},r,\eta_{r-}\big)\Big|\Big]<+\infty.\label{integrability2}
\end{align}%\eqref{integrability2}中a\in[0,\Delta\eta_r]因为\eqref{cxy}可以去掉,之前的条件中a\in[0,\Delta\eta_r]不能去掉。

Then $(\hat{\Pi},\hat{\Xi})$ is an equilibrium regular-singular control law and $V$ is the corresponding equilibrium value function. Moreover, $f$ has the following probabilistic interpretation
\begin{equation}
\begin{aligned}
\label{intf}
f(x,t,y,s)=&\mathbb{E}_{x,t,y}\left[\int_t^{\tau^{\hat{\Pi},\hat{\Xi}}} \beta(r-s) H\big(\hat{X}_r, r, \hat\pi_r, \hat{\xi}_r \big) \mathrm{d} r \right.\\
& \left.+\int_t^{\tau^{\hat{\Pi},\hat{\Xi}}} \beta(r-s) c\big(\hat{X}_{r-}, r, \hat{\xi}_{r-}\big) \mathrm{d} \hat{\xi}_r \right],\ \forall (x,t,y)\in\mathcal{Q}, s\in[0,t].
\end{aligned}
\end{equation}
\end{theorem}

\begin{proof}
We prove Theorem \ref{verificationthm} in the following  three steps:
\vskip 3pt
{\bf Step 1:} {\sl \bf  We show that $f$ has the probabilistic interpretation \eqref{intf}.}
\vskip 3pt
Fix any $(x,t,y)\in\mathcal{Q}$, $s\in[0,t]$, and $n\ge t$, applying the It\^{o}--Tanaka--Meyer formula to $f^{s}$, we obtain
\begin{equation}
\begin{aligned}
&f^{s}\big(\hat{X}_{\tau_{n}},\tau_{n},\hat{\xi}_{\tau_{n}}\big)-f^{s}(x,t,y)\\=&\int_{t}^{\tau_{n}}\big(\mathcal{A}^{\hat{\Pi}}f^{s}\big)\big(\hat{X}_{r},r,\hat{\xi}_{r}\big)\mathrm{d}r+\int_{t}^{\tau_{n}}f^{s}_{x}\big(\hat{X}_{r},r,\hat{\xi}_{r}\big)\sigma\big(\hat{X}_{r},r,\hat\pi_r,\hat{\xi}_{r}\big)\mathrm{d}B_{r}\\
&+\int_{t}^{\tau_{n}}\left[f^{s}_{y}\big(\hat{X}_{r-},r,\hat{\xi}_{r-}\big)-f^{s}_{x}\big(\hat{X}_{r-},r,\hat{\xi}_{r-}\big)\right]\mathrm{d}\hat{\xi}_{r}^{c}\\
&+\sum\limits_{r\in[t,\tau_{n}]}\int_{0}^{\Delta\hat{\xi}_{r}}\left[f^{s}_{y}\big(\hat{X}_{r-}-a,r,\hat{\xi}_{r-}+a\big)-f^{s}_{x}\big(\hat{X}_{r-}-a,r,\hat{\xi}_{r-}+a\big)\right]\mathrm{d}a.
\end{aligned}
\label{fsito}\end{equation}
As $\big(\hat{X}_{r},r,\hat{\xi}_{r}\big)\in\overline{W^{\hat{\Xi}}}$ almost surely for any $r\in[t,\tau^{\hat{\Pi},\hat{\Xi}})$,  using \eqref{fw},  we have $$\big(\mathcal{A}^{\hat{\Pi}}f^{s}\big)\big(\hat{X}_{r},r,\hat{\xi}_{r}\big)=-\beta(r-s)H\big(\hat{X}_{r},r,\hat\pi_r,\hat{\xi}_{r}\big),  \forall r\in[t,\tau_n).$$ Furthermore, using \eqref{martingale}, we  also have 
\begin{equation}
\mathbb{E}_{x,t,y}\left[\int_{t}^{\tau_{n}}f^{s}_{x}\big(\hat{X}_{r},r,\hat{\xi}_{r}\big)\sigma\big(\hat{X}_{r},r,\hat\pi_r,\hat{\xi}_{r}\big)\mathrm{d}B_{r}\right]=0.
\nonumber\end{equation}
Note that for any $r\in[t,\tau^{\hat{\Pi},\hat{\Xi}}]$ such that $\Delta\hat{\xi}_{r}>0$ and any $a\in[0,\Delta\hat{\xi}_{r}]$, it holds almost surely that $(\hat{X}_{r-}-a,r,\hat{\xi}_{r-}+a)\in P^{\hat{\Xi}}$. Combining with the definition of $P^{\hat{\Xi}}$, using \eqref{cxy} and \eqref{fp}, we have 
\begin{equation}\begin{aligned}
&\int_{t}^{\tau_{n}}\left[f^{s}_{y}\big(\hat{X}_{r-},r,\hat{\xi}_{r-}\big)-f^{s}_{x}\big(\hat{X}_{r-},r,\hat{\xi}_{r-}\big)\right]\mathrm{d}\hat{\xi}_{r}^{c}\\
&+\sum\limits_{r\in[t,\tau_{n}]}\int_{0}^{\Delta\hat{\xi}_{r}}\left[f^{s}_{y}\big(\hat{X}_{r-}-a,r,\hat{\xi}_{r-}+a\big)-f^{s}_{x}\big(\hat{X}_{r-}-a,r,\hat{\xi}_{r-}+a\big)\right]\mathrm{d}a\\=&
-\int_{t}^{\tau_{n}}\beta(r-s)c\big(\hat{X}_{r-},r,\hat{\xi}_{r-}\big)\mathrm{d}\hat{\xi}_{r}.
\end{aligned}\nonumber
\end{equation}
Therefore, based on the above analysis, taking expectations $\mathbb{E}_{x,t,y}$ on both sides of \eqref{fsito}, letting $n\uparrow+\infty$, using \eqref{finfty1} and the Dominated Convergence Theorem, we obtain the probabilistic interpretation \eqref{intf}.
\vskip 3pt \noindent
{\bf Step 2:} {\sl \bf  We show that
\begin{equation}
\label{step2eq}
V(x,t,y)=f(x,t,y,t)=J\big(x,t,y;\hat{\Pi},\hat{\Xi}\big),\ \forall(x,t,y)\in\mathcal{Q}.
\end{equation}}
\vskip 3pt \noindent
Similar to {\bf Step 1}, fix any $(x,t,y)\in\mathcal{Q}$ and $n\ge t$, applying the It\^{o}--Tanaka--Meyer formula to $V$, taking expectations $\mathbb{E}_{x,t,y}$, and using \eqref{cxy} and \eqref{martingale}, we have {\small
\begin{equation}\label{intv}
V(x,t,y)=\mathbb{E}_{x,t,y}\left[\int_t^{\tau_{n}} -\big(\mathcal{A}^{\hat\Pi}V\big)\big(\hat{X}_r, r, \hat{\xi}_r \big) \mathrm{d} r+V\big(\hat{X}_{\tau_{n}}, \tau_{n},\hat{\xi}_{\tau_{n}}\big)+\int_t^{\tau_{n}} c\big(\hat{X}_{r-}, r, \hat{\xi}_{r-}\big) \mathrm{d} \hat{\xi}_r \right].
\end{equation}}
Again, applying the It\^{o}--Tanaka--Meyer formula to $(x,t,y)\mapsto f(x,t,y,t)$, taking expectations $\mathbb{E}_{x,t,y}$, and using \eqref{cxy}, \eqref{fp} and \eqref{martingale}, we have
{\small 
\begin{equation}\label{intft}
f(x,t,y,t)=\mathbb{E}_{x,t,y}\left[\int_t^{\tau_{n}} \!\!\!\!-\!\big(\mathcal{A}^{\hat\Pi}f\big)\big(\hat{X}_r, r, \hat{\xi}_r,r \big) \mathrm{d} r\!+ \!f\big(\hat{X}_{\tau_{n}}, \tau_{n},\hat{\xi}_{\tau_{n}},\tau_{n}\big)\!+\!\!\int_t^{\tau_{n}}\!\!\! c\big(\hat{X}_{r-}, r, \hat{\xi}_{r-}\big) \mathrm{d} \hat{\xi}_r \right].
\end{equation}}
Because $\hat{\Pi}$ is continuous on $\overline{W^{\hat{\Xi}}}-W^{\hat{\Xi}}$  when $\hat{\Pi}$ is regarded as a function on $\overline{W^{\hat{\Xi}}}$, the combination of \eqref{v} and \eqref{fw} yields 
\begin{equation}
\label{vf}
\big(\mathcal{A}^{\hat\Pi}V\big)(x_0,t_0,y_0)=\big(\mathcal{A}^{\hat\Pi}f\big)(x_0,t_0,y_0,t_0),\ \forall (x_0,t_0,y_0)\in \overline{W^{\hat{\Xi}}}.
\end{equation}
Then, based on \eqref{intv}--\eqref{vf}, letting $n\uparrow+\infty$ and using \eqref{finfty2}, we obtain $V(x,t,y)=f(x,t,y,t)$. The second equality in \eqref{step2eq} follows directly from the probabilistic interpretation \eqref{intf}. Thus \eqref{step2eq} holds.
\vskip 5pt \noindent
{\bf Step 3:} {\sl \bf  We prove that $(\hat{\Pi},\hat{\Xi})$ is an equilibrium regular-singular control law.}
\vskip 5pt
Fixed any $(x,t,y)\in\mathcal{Q}$. For any admissible regular-singular control pair $(\pi,\xi)\in\mathcal{D}^{x,t,y}$, we define its value function  by

\begin{equation}\label{fpixi}
\begin{aligned}
f^{\pi,\xi}(x,t,y,s):=&\mathbb{E}_{x,t,y}\left[\int_t^{\tau^{\pi,\xi}} \beta(r-s) H\big(X^{\pi,\xi}_r, r,\pi_r, \xi_r \big) \mathrm{d} r\right. \\
& \left.+\int_t^{\tau^{\pi,\xi}} \beta(r-s) c\big(X^{\pi,\xi}_{r-}, r, \xi_{r-}\big) \mathrm{d}\xi_r \right],\ \forall  s\in[0,t].
\end{aligned}
\end{equation}
Clearly, we have $f^{\hat\pi,\hat\xi}(x,t,y,s)=f(x,t,y,s), \forall s\in [0,t]$. 

Fix any admissible pair of perturbations $(u,\eta)\in\mathcal{\tilde D}^{x,t,y;\hat\Pi,\hat\Xi}$, there exist constants $\tilde{h}>0$ and $\tilde{M}>0$ such that $\eta_{(t+h)-}-\eta_{t}\le \tilde{M}h, \text{a.s.}, \forall h\in(0,\tilde{h})$, and the perturbed regular-singular control pair $(\pi^{h},\xi^{h})$, $\forall h\in(0,\tilde{h})$ defined in \eqref{pihxih} and the corresponding state process $X^{\pi^{h},\xi^{h}}$ in \eqref{xpihxih} satisfy integrability condition (b) in Definition \ref{admissible2}. Thus, we can define $f^{\pi^{h},\xi^{h}}(x,t,y,s), \forall h\in(0,\tilde{h}), s\in[0,t]$ just like  \eqref{fpixi}. 

Moreover, based on  \eqref{pihxih}, \eqref{xpihxih} and \eqref{fT}, we  define 
\begin{equation}\begin{aligned}
&f^{\pi^{h},\xi^{h}}\big(X^{\pi^{h},\xi^{h}}_{\tau^{u,\eta}\wedge((t+h)-)},\tau^{u,\eta}\wedge(t+h),\xi^{h}_{\tau^{u,\eta}\wedge((t+h)-)},s\big)\\=&f^{\pi^{h},\xi^{h}}\big(X^{u,\eta}_{\tau^{u,\eta}\wedge((t+h)-)},\tau^{u,\eta}\wedge(t+h),\eta_{\tau^{u,\eta}\wedge((t+h)-)},s\big)\\=&f\big(X^{u,\eta}_{\tau^{u,\eta}\wedge((t+h)-)},\tau^{u,\eta}\wedge(t+h),\eta_{\tau^{u,\eta}\wedge((t+h)-)},s\big),\ \forall s\in[0,t], h\in(0,\tilde{h}).
\end{aligned}\nonumber\end{equation}
%Besides, we know that $\tau^{x,t,y;\pi^{h},\xi^{h}}_{t}=\tau^{x,t,y;u,\eta}_{t}$, when $\tau^{x,t,y;u,\eta}_{t}<t+h$, and $1_{\left\{\tau^{x,t,y;\pi^{h},\xi^{h}}_{t}\ge t+h\right\}}=1_{\left\{\tau^{x,t,y;u,\eta}_{t}\ge t+h\right\}}, \forall h\in(0,\tilde{h})$. Therefore, $\tau^{x,t,y;\pi^{h},\xi^{h}}\wedge(t+h)=\tau^{x,t,y;u,\eta}\wedge(t+h), \forall h\in(0,\tilde{h})$.

Hereafter, we assume that $h$ is arbitrary in $(0,\tilde{h}\wedge h_0)$, where $h_0>0$ is the constant in Condition (6). Observing that $\pi^{h}_{r}=u_{r}, \xi^{h}_{r}=\eta_{r}, X^{\pi^{h},\xi^{h}}_r=X^{u,\eta}_r, \forall r\in[t,\tau^{\pi^{h},\xi^{h}}\wedge((t+h)-)]$, we have 
\begin{equation}
\begin{aligned}
J\big(x,t,y;\pi^{h},\xi^{h}\big)=&\mathbb{E}_{x,t,y}\bigg[f^{\pi^{h},\xi^{h}}\big(X^{u,\eta}_{\tau^{u,\eta}\wedge((t+h)-)},\tau^{\pi^{h},\xi^{h}}\wedge(t+h),\eta_{\tau^{u,\eta}\wedge((t+h)-)},t\big)\\
  &\left.+\int_{t}^{\tau^{u,\eta}\wedge(t+h)}\beta(r-t) H\big(X^{u,\eta}_r, r,u_r, \eta_r \big) \mathrm{d} r \right.\\
  &\left.+\int_{t}^{\tau^{u,\eta}\wedge(t+h)}\beta(r-t) c\big(X^{u,\eta}_{r-}, r, \eta_{r-}\big) \mathrm{d}\eta^{c}_r\right.\\
  &+\sum\limits_{r\in[t,\tau^{u,\eta}\wedge((t+h)-)]}\beta(r-t) c\big(X^{u,\eta}_{r-}, r, \eta_{r-}\big) \Delta\eta_r\bigg].
\end{aligned}\nonumber\end{equation}
Then, using \eqref{step2eq} and the fact that {\small $f^{\pi^{h},\xi^{h}}\big(X^{u,\eta}_{\tau^{u,\eta}\wedge((t+h)-)},\tau^{u,\eta}\wedge(t+h),\eta_{\tau^{u,\eta}\wedge((t+h)-)},t\big)\\=f\big(X^{u,\eta}_{\tau^{u,\eta}\wedge((t+h)-)},\tau^{u,\eta}\wedge(t+h),\eta_{\tau^{u,\eta}\wedge((t+h)-)},t\big),$} we obtain {\small 
\begin{equation}
\begin{aligned}
&J\big(x,t,y;\pi^{h},\xi^{h}\big)-J\big(x,t,y;\hat\pi,\hat\xi\big)\\
=&\mathbb{E}_{x,t,y}\left[f\big(X^{u,\eta}_{\tau^{u,\eta}\wedge((t+h)-)},\tau^{u,\eta}\wedge(t+h),\eta_{\tau^{u,\eta}\wedge((t+h)-)},t\big)-V(x,t,y)\right.\\
&\left.+\int_{t}^{\tau^{u,\eta}\wedge(t+h)}\beta(r-t) H\big(X^{u,\eta}_r, r,u_r, \eta_r \big) \mathrm{d} r +\int_{t}^{\tau^{u,\eta}\wedge(t+h)}\beta(r-t) c\big(X^{u,\eta}_{r-}, r, \eta_{r-}\big) \mathrm{d}\eta^{c}_r\right.\\
&+\sum\limits_{r\in[t,\tau^{u,\eta}\wedge((t+h)-)]}\beta(r-t) c\big(X^{u,\eta}_{r-}, r, \eta_{r-}\big) \Delta\eta_r\Big].
\end{aligned}\nonumber
\end{equation}}
\vskip 3pt \noindent
Similar to {\bf Step 1}, applying the It\^{o}--Tanaka--Meyer formula to $V$, taking expectations $\mathbb{E}_{x,t,y}$, and using \eqref{martingale2} and \eqref{step2eq} yields {\small
\begin{equation}
\begin{aligned}
&\mathbb{E}_{x,t,y}\left[f\big(X^{u,\eta}_{\tau^{u,\eta}\wedge((t+h)-)},\tau^{u,\eta}\wedge(t+h),\eta_{\tau^{u,\eta}\wedge((t+h)-)},t\big)-V(x,t,y)\right]\\
=&\mathbb{E}_{x,t,y}\left[f\big(X^{u,\eta}_{\tau^{u,\eta}\wedge((t+h)-)},\tau^{u,\eta}\wedge(t+h),\eta_{\tau^{u,\eta}\wedge((t+h)-)},t\big)\right.\\ &\left.-f\big(X^{u,\eta}_{\tau^{u,\eta}\wedge((t+h)-)},\tau^{u,\eta}\wedge(t+h),\eta_{\tau^{u,\eta}\wedge((t+h)-)},\tau^{u,\eta}\wedge(t+h)\big)\right.\\ &\left.+ V\big(X^{u,\eta}_{\tau^{u,\eta}\wedge((t+h)-)},\tau^{u,\eta}\wedge(t+h),\eta_{\tau^{u,\eta}\wedge((t+h)-)}\big)-V(x,t,y)\right]\\
=&\mathbb{E}_{x,t,y}\Big[\int_{t}^{\tau^{u,\eta}\wedge(t+h)}\left[\big(\mathcal{A}^{u}V\big)\big(X^{u,\eta}_{r},r,\eta_{r}\big)\!-\!f_s\big(X^{u,\eta}_{\tau^{u,\eta}\wedge((t+h)-)},\tau^{u,\eta}\!\wedge\!(t\!+\!h),\eta_{\tau^{u,\eta}\wedge((t+h)-)},r\big)\right]\mathrm{d}r\\ &\left.+\int_{t}^{\tau^{u,\eta}\wedge(t+h)}\left[V_{y}\big(X^{u,\eta}_{r-},r,\eta_{r-}\big)-V_{x}\big(X^{u,\eta}_{r-},r,\eta_{r-}\big)\right]\mathrm{d}\eta_{r}^{c}\right.\\
&+\sum\limits_{r\in[t,\tau^{u,\eta}\wedge((t+h)-)]}\int_{0}^{\Delta\eta_{r}}\left[V_{y}\big(X^{u,\eta}_{r-}-a,r,\eta_{r-}+a\big)-V_{x}\big(X^{u,\eta}_{r-}-a,r,\eta_{r-}+a\big)\right]\mathrm{d}a\Big],
\end{aligned}\nonumber
\end{equation}}
where $u$ in $\big(\mathcal{A}^{u}V\big)$ refers to the regular perturbation process $u$, i.e.,
\begin{equation}\big(\mathcal{A}^{u}\varphi\big)(x,t,y):=\varphi_{t}(x,t,y)+\mu(x,t,u_t,y)\varphi_{x}(x,t,y)+\frac{1}{2}\sigma^{2}(x,t,u_t,y)\varphi_{xx}(x,t,y),\ \forall\varphi:\mathcal{Q}\rightarrow\mathbb{R}.\nonumber\end{equation}
Therefore, combining with \eqref{cxy}, the definition of $\hat\Pi$, and \eqref{v}, we have
{\small 
\begin{align}
&J\big(x,t,y;\pi^{h},\xi^{h}\big)-J\big(x,t,y;\hat\pi,\hat\xi\big)\label{est0}\\=&\mathbb{E}_{x,t,y}\bigg[\int_{t}^{\tau^{u,\eta}\wedge(t+h)}\left[\big(\mathcal{A}^{u}V\big)\big(X^{u,\eta}_{r},r,\eta_{r}\big)+\beta(r-t) H\big(X^{u,\eta}_r, r,u_r, \eta_r \big) \right]\mathrm{d} r\nonumber\\
  &\left.-\int_{t}^{\tau^{u,\eta}\wedge(t+h)}f_s\big(X^{u,\eta}_{\tau^{u,\eta}\wedge((t+h)-)},\tau^{u,\eta}\wedge(t+h),\eta_{\tau^{u,\eta}\wedge((t+h)-)},r\big)\mathrm{d}r \right.\nonumber\\
  &\left.+\int_{t}^{\tau^{u,\eta}\wedge(t+h)}\left[\beta(r-t) c\big(X^{u,\eta}_{r-}, r, \eta_{r-}\big)+\big(V_{y}-V_{x}\big)\big(X^{u,\eta}_{r-},r,\eta_{r-}\big)\right] \mathrm{d}\eta^{c}_r\right.\nonumber\\  &+\!\!\!\!\!\!\!\!\sum\limits_{r\in[t,\tau^{u,\eta}\wedge((t+h)-)]}\int_{0}^{\Delta\eta_{r}}\!\![\beta(r-t) c\big(X^{u,\eta}_{r-}-a, r, \eta_{r-}+a\big)+\big(V_{y}-V_{x}\big)\big(X^{u,\eta}_{r-}-a,r,\eta_{r-}+a\big)]\mathrm{d}a\bigg]\nonumber\\  \ge&\mathbb{E}_{x,t,y}\bigg[\int_{t}^{\tau^{u,\eta}\wedge(t+h)}\left[f_s\big(X^{u,\eta}_{r},r,\eta_{r},r\big)+(\beta(r-t)-1) H\big(X^{u,\eta}_r, r,u_r, \eta_r \big) \right]\mathrm{d} r\nonumber\\
  &\left.-\int_{t}^{\tau^{u,\eta}\wedge(t+h)}f_s\big(X^{u,\eta}_{\tau^{u,\eta}\wedge((t+h)-)},\tau^{u,\eta}\wedge(t+h),\eta_{\tau^{u,\eta}\wedge((t+h)-)},r\big)\mathrm{d}r \right.\nonumber\\
  &\left.+\int_{t}^{\tau^{u,\eta}\wedge(t+h)}(\beta(r-t)-1) c\big(X^{u,\eta}_{r-}, r, \eta_{r-}\big)\mathrm{d}\eta^{c}_r\right.\nonumber\\
&+\sum\limits_{r\in[t,\tau^{u,\eta}\wedge((t+h)-)]}\int_{0}^{\Delta\eta_{r}}(\beta(r-t)-1) c\big(X^{u,\eta}_{r-}-a, r, \eta_{r-}+a\big)\mathrm{d}a\bigg].\nonumber
\end{align}}
Using \eqref{integrability} and the Dominated Convergence Theorem, as  the regular control space $U\subset\mathbb{R}$ is a compact set and $u_r\in U$, we have 
{\small 
\begin{equation}
\begin{aligned}
&
\begin{aligned}\lim\limits_{h\downarrow 0}\mathbb{E}_{x,t,y}\frac{1}{h}\int_{t}^{\tau^{u,\eta}\wedge(t+h)}\Big[&f_s\big(X^{u,\eta}_{r},r,\eta_{r},r\big)+(\beta(r-t)-1) H\big(X^{u,\eta}_r, r,u_r, \eta_r \big) \\
  &-f_s\big(X^{u,\eta}_{\tau^{u,\eta}\wedge((t+h)-)},\tau^{u,\eta}\wedge(t+h),\eta_{\tau^{u,\eta}\wedge((t+h)-)},r\big)\Big]\mathrm{d}r
  \end{aligned}\\
=&\mathbb{E}_{x,t,y}\left[\left[f_s\big(X^{u,\eta}_{t},t,\eta_{t},t\big)+ (\beta(t-t)-1)H\big(X^{u,\eta}_t, t,u_t, \eta_t \big)-f_s\big(X^{u,\eta}_{t},t,\eta_{t},t\big)\right]1_{\left\{\tau^{u,\eta}>t\right\}} \right]=0.
\end{aligned}\label{est1}\end{equation} }
Because $\eta_{(t+h)-}-\eta_{t}\le \tilde{M}h, \text{a.s.}$, we know $\int_{t}^{t+h}\mathrm{d}\eta^{c}_{r}\le \tilde{M}h, \text{a.s.}$ and $\sum\limits_{r\in(t,t+h)}\Delta\eta_{r}\le \tilde{M}h, \text{a.s.}$. Then using \eqref{cxy}, \eqref{integrability2} and the Fatou's lemma, we have
\begin{equation}\begin{aligned}
&\liminf\limits_{h\downarrow 0}\mathbb{E}_{x,t,y}\left[\frac{1}{h}\int_{t}^{\tau^{u,\eta}\wedge(t+h)}\left(\beta(r-t)-1\right) c\big(X^{u,\eta}_{r-}, r, \eta_{r-}\big) \mathrm{d}\eta^{c}_r\right]\\
\geq&\mathbb{E}_{x,t,y}\left[\liminf\limits_{h\downarrow 0}\frac{1}{h}\int_{t}^{\tau^{u,\eta}\wedge(t+h)}\left(\beta(r-t)-1\right) c\big(X^{u,\eta}_{r}, r, \eta_{r}\big)\mathrm{d}\eta^{c}_r \right]\\
=&\mathbb{E}_{x,t,y}\left[ \liminf\limits_{h\downarrow 0}\frac{1}{h}\int_{t}^{\tau^{u,\eta}\wedge(t+h)}\left(\beta(t-t)-1\right) c\big(X^{u,\eta}_t, t, \eta_t\big)\mathrm{d}\eta^{c}_{r}\right]= 0,
\end{aligned}\label{est2}\end{equation}
and
\begin{equation}\begin{aligned}
&\liminf\limits_{h\downarrow 0}\mathbb{E}_{x,t,y}\left[\frac{1}{h}\sum\limits_{r\in[t,\tau^{u,\eta}\wedge((t+h)-)]}\int_{0}^{\Delta\eta_{r}}\left(\beta(r-t)-1\right) c\big(X^{u,\eta}_{r-}-a, r, \eta_{r-}+a\big)\mathrm{d}a\right]\\
=&\liminf\limits_{h\downarrow 0}\mathbb{E}_{x,t,y}\left[\frac{1}{h}\sum\limits_{r\in(t,\tau^{u,\eta}\wedge((t+h)-)]}\left(\beta(r-t)-1\right) c\big(X^{u,\eta}_{r-}, r, \eta_{r-}\big)\Delta\eta_{r}\right]\\
\geq&\mathbb{E}_{x,t,y}\left[\liminf\limits_{h\downarrow 0}\frac{1}{h}\sum\limits_{r\in(t,\tau^{u,\eta}\wedge((t+h)-)]}\left(\beta(r-t)-1\right) c\big(X^{u,\eta}_{r}, r, \eta_{r}\big)\Delta\eta_{r} \right]\\
=&\mathbb{E}_{x,t,y}\left[\liminf\limits_{h\downarrow 0}\frac{1}{h}\sum\limits_{r\in(t,\tau^{u,\eta}\wedge((t+h)-)]}\left(\beta(t-t)-1\right)c\big(X^{u,\eta}_t, t, \eta_t\big)\Delta\eta_{r}\right]= 0.
\end{aligned}\label{est3}
\end{equation}
Using \eqref{est0}--\eqref{est3} yields that the equilibrium condition \eqref{equicond} holds. Thus $(\hat{\Pi},\hat{\Xi})$ is an equilibrium  regular-singular control law.
\end{proof}

Generally speaking, the core sufficient condition for the equilibrium is to satisfy the extended HJB system \eqref{v}--\eqref{fT} given in (3) of Theorem \ref{verificationthm}. We will solve the extended HJB system \eqref{v}--\eqref{fT} in the effort and dividend problem in Section \ref{Application} to obtain a semi-explicit equilibrium.
%是否要加必要性定理？

\section{Mild equilibrium and mild verification theorem}\label{mild}
Similar to \cite{liang2023weak}, another definition of mild equilibrium only allows perturbations of regular control for fixed singular control law, and perturbations of singular control for fixed regular control law. This alternative equilibrium notion accommodates scenarios where the decision maker may deviate in only one type of control at a time, rather than both simultaneously. The mild equilibrium is defined as follows:
\begin{definition}[Mild equilibrium regular-singular control law]
\label{mildequil}
An admissible regular-singular control law $(\hat{\Pi},\hat{\Xi})$ is called a mild equilibrium for Problem \eqref{objective} if the following conditions (a) and (b) hold:\\
(a). For any initial $(x,t,y)\in\mathcal{Q}$, any $U$-valued $\{\mathcal{F}_{r}\}_{r\ge t}$-progressively measurable regular control perturbation process $u=\{u_{r}\}_{r\ge t}$ satisfying:\\
(i). The Skorokhod reflection type SDE
\begin{equation}
\left\{\begin{array}{l}
\mathrm{d}X^{u,\hat{\Xi}}_{r}=\mu\big(X^{u,\hat{\Xi}}_{r},r,u_{r},\xi^{u,\hat{\Xi}}_{r}\big)\mathrm{d}r+\sigma\big(X^{u,\hat{\Xi}}_{r},r,u_{r},\xi^{u,\hat{\Xi}}_{r}\big)\mathrm{d}B_{r}-\mathrm{d}\xi^{u,\hat{\Xi}}_{r},\ r\in[t,\tau^{u,\hat{\Xi}}],\\
\big(X^{u,\hat{\Xi}}_{r},r,\xi^{u,\hat{\Xi}}_{r}\big)\in \overline{W^{\hat{\Xi}}},\ r\in[t,\tau^{u,\hat{\Xi}}),\\
\xi^{u,\hat{\Xi}}_{r}=y+\int_{t}^{r}1_{\left\{\big(X^{u,\hat{\Xi}}_{l},l,\xi^{u,\hat{\Xi}}_{l}\big)\in P^{\hat{\Xi}}\right\}}\mathrm{d}\xi^{u,\hat{\Xi}}_{l},\quad\forall r\in[t,\tau^{u,\hat{\Xi}}],\\
X^{u,\hat{\Xi}}_{t-}=x
\end{array}\right.\nonumber
\end{equation}
has a unique strong solution $\big(X^{u,\hat{\Xi}},\xi^{u,\hat{\Xi}}\big)$, where $\tau^{u,\hat{\Xi}}:=\inf\big\{r\ge t\bigm|X^{u,\hat{\Xi}}_{r}= 0\big\}$;\\
(ii). There exists a constant $\tilde{h}>0$ such that $\forall h\in(0,\tilde{h})$, the perturbed regular-singular control pair $(\pi^{h},\xi^{h})$ defined by \eqref{pihxih} with $\eta_r:=\xi^{u,\hat{\Xi}}_r$ and the corresponding state process \eqref{xpihxih} satisfy the integrability condition (b) in Definition \ref{admissible2}, then the equilibrium condition \eqref{equicond} holds.\\%\eta_{r}是否需要增长条件\eqref{growthcond}？对于奇异控制率导出的奇异控制，增长条件很可能不满足。只对initial $(x,t,y)\in\mathcal{Q}$ 在W^{\hat{\Xi}}的内部的初始点进行定义？验证定理积分和极限交换还是不太过得去。还是直接假设c\le0？
(b). For any initial $(x,t,y)\in\mathcal{Q}$, any non-decreasing c\`{a}dl\`{a}g $\{\mathcal{F}_{r}\}_{r\ge t}$-adapted singular control perturbation process $\{\eta_{r}\} _{r\ge t}$ satisfying property \eqref{growthcond}, $\eta_{t-}=y$, $\Delta \eta_r=\eta_r-\eta_{r-}\le X_{r-}^{\hat\Pi,\eta}$, if:\\
(i). The SDE
\begin{equation}
\left\{\begin{array}{l}
\begin{aligned}
\mathrm{d}X^{\hat{\Pi},\eta}_{r}=&\mu\big(X^{\hat{\Pi},\eta}_{r},r,\hat{\Pi}\big(X^{\hat{\Pi},\eta}_{r},r,\eta_{r}\big),\eta_{r}\big)\mathrm{d}r\\&+\sigma\big(X^{\hat{\Pi},\eta}_{r},r,\hat{\Pi}\big(X^{\hat{\Pi},\eta}_{r},r,\eta_{r}\big),\eta_{r}\big)\mathrm{d}B_{r}-\mathrm{d}\eta_{r},\ r\in[t,\tau^{\hat{\Pi},\eta}],
\end{aligned}\\
X^{\hat{\Pi},\eta}_{t-}=x,
\end{array}\right.\nonumber
\end{equation}
has a unique strong solution $X^{\hat{\Pi},\eta}$, where $\tau^{\hat{\Pi},\eta}:=\inf\big\{r\ge t\bigm|X^{\hat{\Pi},\eta}_{r}= 0\big\}$;\\ (ii). There exists a constant $\tilde{h}>0$ such that $\forall h\in(0,\tilde{h})$, the perturbed regular-singular control pair $(\pi^{h},\xi^{h})$ defined by \eqref{pihxih} with $u_{r}:=\hat{\Pi}(X^{\hat{\Pi},\eta}_{r},r,\eta_{r})$ and the corresponding state process \eqref{xpihxih} satisfy the integrability condition (b) in Definition \ref{admissible2};\\
then the equilibrium condition \eqref{equicond} holds.

 Corresponding to the mild equilibrium regular-singular control law $(\hat{\Pi},\hat{\Xi})$, we define the mild equilibrium value function $V$ by $V(x,t,y):=J\big(x,t,y;\hat{\Pi},\hat{\Xi}\big)=J\big(x,t,y;\hat{\pi},\hat{\xi}\big)$.
\end{definition}

\begin{remark}
An equilibrium is not necessarily a mild equilibrium 
because the growth of singular strategy in the equilibrium may be of order $\sqrt{h}$, and a mild equilibrium is not necessarily an equilibrium (see Proposition \ref{toyeg}); the two are simply equilibria under two different perturbation structures.
\end{remark}

Similar to Theorem \ref{verificationthm}, we can give the mild verification theorem for the mild equilibrium. For simplicity, we assume that the cost function $c$ in the objective is non-positive hereafter in this section. In practice, the singular control usually represents dividends and this assumption is reasonable. 

\begin{theorem}[Mild Verification Theorem]\label{mildverificationthm}
Given a function $V(x,t,y)$ and a family of functions $\left\{f^{s}(x,t,y)\right\}_{s\in[0,+\infty)}$. Define $f(x,t,y,s):=f^{s}(x,t,y)$, and the infinitesimal operator by
\begin{equation}
\begin{aligned}
&\big(\mathcal{A}^{u}\varphi\big)(x,t,y):=\varphi_{t}(x,t,y)+\mu(x,t,u,y)\varphi_{x}(x,t,y)+\frac{1}{2}\sigma^{2}(x,t,u,y)\varphi_{xx}(x,t,y),\ \forall\varphi:\mathcal{Q}\rightarrow\mathbb{R},\\
&\big(\mathcal{A}^{\Pi}\varphi\big)(x,t,y):=\big(\mathcal{A}^{\Pi(x,t,y)}\varphi\big)(x,t,y),\ \forall\varphi:\mathcal{Q}\rightarrow\mathbb{R}.
\end{aligned}\nonumber
\end{equation} 
Assume that the following conditions hold:\\
(1).
\begin{equation}
\begin{aligned}
 V \in C^{2,1,1}\left(\mathcal{Q}\right), \  f \in C^{2,1,1,1}\left(\big\{(x,t,y,s)\bigm| (x,t,y)\in\mathcal{Q}, s\in [0,t]\big\}\right).
\end{aligned}\nonumber
\end{equation}
(2). There exists an admissible regular-singular control law $(\hat{\Pi},\hat{\Xi})$ such that the singular control law $\hat{\Xi}$ is given by
\begin{equation}
\begin{aligned}
&W^{\hat{\Xi}}:=\big\{(x,t,y)\in\mathcal{Q}\bigm|c(x,t,y)-V_{x}(x,t,y)+V_{y}(x,t,y)>0\big\},\\
&P^{\hat{\Xi}}:=\big\{(x,t,y)\in\mathcal{Q}\bigm|c(x,t,y)-V_{x}(x,t,y)+V_{y}(x,t,y)=0\big\},
\end{aligned}\nonumber
\end{equation}
and the regular control law $\hat{\Pi}$ satisfies : \  $ \forall (x,t,y)\in \overline{W^{\hat{\Xi}}}$,
\begin{equation}
\hat{\Pi}(x,t,y)\in\mathop{\mathrm{argmin}}\limits_{u\in U}\left\{\big(\mathcal{A}^{u}V\big)(x,t,y)+H(x,t,u,y)-f_s(x,t,y,t)\right\}.
\label{mildpi}
\end{equation}
$\hat{\Pi}$ is continuous on $\overline{W^{\hat{\Xi}}}-W^{\hat{\Xi}}$ when $\hat{\Pi}$ is regarded as a function on $\overline{W^{\hat{\Xi}}}$. \\
(3). $V(x,t,y)$ and $f^{s}(x,t,y)$ satisfy

\begin{subequations}
\begin{align}
&\begin{aligned}
\min\Big\{&\big(\mathcal{A}^{\hat{\Pi}}V\big)(x,t,y)+H\big(x,t,\hat{\Pi}(x,t,y),y\big)-f_s(x,t,y,t),\\
&c(x,t,y)-V_{x}(x,t,y)+V_{y}(x,t,y)\Big\}=0,\ \forall (x,t,y)\in\mathcal{Q},
\end{aligned}\label{mildv}\\
&V(0,t,y)=0,\ \forall(0,t,y)\in\mathcal{Q},\label{mildvT}\\
&\big(\mathcal{A}^{\hat{\Pi}}f^{s}\big)(x,t,y)+\beta(t-s)H\big(x,t,\hat\Pi(x,t,y),y\big)=0,\ \forall (x,t,y)\in \overline{W^{\hat{\Xi}}}, s\in[0,t],\label{mildfw}\\
& \beta(t-s)c(x,t,y)-f^{s}_{x}(x,t,y)+f^{s}_{y}(x,t,y)=0,\ \forall (x,t,y)\in P^{\hat{\Xi}}, s\in[0,t],\label{mildfp}\\
&f^{s}(0,t,y)=0,\ \forall (0,t,y)\in\mathcal{Q}, s\in[0,t].\label{mildfT}
\end{align}
\end{subequations}
(4). For functions $\varphi=f^{s},V,(x,t,y)\mapsto f(x,t,y,t)$, for any $(x,t,y)\in\mathcal{Q}$ with $s\in[0,t]$, and  $n\geq t$, it holds that 
\begin{equation}
\mathbb{E}_{x,t,y}\int_{t}^{\tau_n}\left[\varphi_{x}\big(\hat{X}_{r},r,\hat\xi_{r}\big)\sigma\big(\hat{X}_{r},r,\hat\pi_r,\hat\xi_{r}\big)\right]^{2}\mathrm{d}r<+\infty,
\nonumber\end{equation}
where $(\hat{\pi},\hat{\xi}):=(\pi^{x,t,y;\hat{\Pi},\hat{\Xi}},\xi^{x,t,y;\hat{\Pi},\hat{\Xi}})$ is the regular-singular control process generated by $(\hat{\Pi},\hat{\Xi})$ at $(x,t,y)$, $\hat{X}:=X^{x,t,y;\hat{\Pi},\hat{\Xi}}$ is the corresponding state process, and $\tau_{n}:=\tau^{\hat{\Pi},\hat{\Xi}}\wedge n$.\\
(5). For any $(x,t,y)\in\mathcal{Q}$ and $s\in[0,t]$, it holds that 

\begin{align}
&\lim\limits_{n\rightarrow+\infty}\mathbb{E}_{x,t,y}f^{s}\big(\hat{X}_{\tau_{n}},\tau_{n},\hat{\xi}_{\tau_{n}}\big)=0,\nonumber\\
&\lim\limits_{n\rightarrow+\infty}\mathbb{E}_{x,t,y}\left[V\big(\hat{X}_{\tau_{n}},\tau_{n},\hat{\xi}_{\tau_{n}}\big)-f\big(\hat{X}_{\tau_{n}},\tau_{n},\hat{\xi}_{\tau_{n}},\tau_{n}\big)\right]=0.\nonumber
\end{align}
(6). For any $(x,t,y)\in\mathcal{Q}$ and  $U$-valued $\{\mathcal{F}_{r}\}_{r\ge t}$-progressively measurable regular control perturbation process $u=\{u_{r}\}_{r\ge t}$ which satisfies conditions specified in Definition \ref{mildequil}, Item (a), and $\eta_{r}:=\xi^{u,\hat{\Xi}}_{r}$ (or any non-decreasing c\`{a}dl\`{a}g $\{\mathcal{F}_{r}\}_{r\ge t}$-adapted singular control perturbation process $\{\eta_{r}\}_{r\ge t}$ which satisfies conditions specified in Definition \ref{mildequil}, Item (b), and $u_{r}:=\hat{\Pi}(X^{\hat{\Pi},\eta}_{r},r,\eta_{r})$), there exists a constant $h_0>0$ such that 

\begin{align}
&\mathbb{E}_{x,t,y}\int_{t}^{\tau^{u,\eta}\wedge(t+h_0)}\left[V_{x}\big(X^{u,\eta}_{r},r,\eta_{r}\big)\sigma\big(X^{u,\eta}_{r},r,u_r,\eta_{r}\big)\right]^{2}\mathrm{d}r<+\infty,\nonumber\\
&\begin{aligned}\mathbb{E}_{x,t,y}\bigg[&\sup\limits_{\substack{r\in(t,\tau^{u,\eta}\wedge(t+h_{0}))\\ h\in(0,h_{0})}}\Big|f_s\big(X^{u,\eta}_{r},r,\eta_{r},r\big)+(\beta(r\!-\!t)\!-\!1) H\big(X^{u,\eta}_r, r,u_r, \eta_r \big)\\
&-f_s\big(X^{u,\eta}_{\tau^{u,\eta}\wedge((t+h)-)},\tau^{u,\eta}\!\wedge\!(t\!+\!h),\eta_{\tau^{u,\eta}\wedge((t+h)-)},r\big)\Big|\bigg]\!<\!+\!\infty.\!
\end{aligned}\nonumber
\end{align}
Then $(\hat{\Pi},\hat{\Xi})$ is a mild equilibrium regular-singular control law and $V$ is the corresponding mild equilibrium value function. Moreover, $f$ has the probabilistic interpretation
\begin{equation}
\begin{aligned}
\nonumber
f(x,t,y,s)=&\mathbb{E}_{x,t,y}\left[\int_t^{\tau^{\hat{\Pi},\hat{\Xi}}} \beta(r-s) H\big(\hat{X}_r, r, \hat\pi_r, \hat{\xi}_r \big) \mathrm{d} r \right.\\
& \left.+\int_t^{\tau^{\hat{\Pi},\hat{\Xi}}} \beta(r-s) c\big(\hat{X}_{r-}, r, \hat{\xi}_{r-}\big) \mathrm{d} \hat{\xi}_r \right],\ \forall (x,t,y)\in\mathcal{Q}, s\in[0,t].
\end{aligned}
\end{equation}
\end{theorem}

\begin{proof}
The proof is similar to Theorem \ref{verificationthm} and the only difference lies in {\bf Step 3}. Because the cost function $c$ is non-positive, it holds true that 
\begin{equation}
\begin{aligned}
\nonumber
\mathbb{E}_{x,t,y}\left[\int_{t}^{\tau^{u,\eta}\wedge(t+h)}\left(\beta(r-t)-1\right) c\big(X^{u,\eta}_{r-}, r, \eta_{r-}\big) \mathrm{d}\eta^{c}_r\right]\ge 0,
\end{aligned}
\end{equation}
and 
\begin{equation}
\nonumber
\mathbb{E}_{x,t,y}\bigg[\sum\limits_{r\in[t,\tau^{u,\eta}\wedge((t+h)-)]}\int_{0}^{\Delta\eta_{r}}\left(\beta(r-t)-1\right) c\big(X^{u,\eta}_{r-}-a, r, \eta_{r-}+a\big)\mathrm{d}a\bigg]\ge 0.
\end{equation}
The remaining proofs are the same.
\end{proof}

\begin{remark}\label{difference}
The main difference between Theorem \ref{verificationthm} and Theorem \ref{mildverificationthm} lies in \eqref{mildpi}, i.e., $\hat\Pi$ is defined on $\mathcal{Q}$ but is the minimum value point of $\big(\mathcal{A}^{u}V\big)(x,t,y)+H(x,t,u,y)-f_s(x,t,y,t)$ only on $\overline{W^{\hat{\Xi}}}$. This is because when perturbing regular control for fixed singular control law $\hat{\Xi}$, the perturbed state-time-(singular)control triple falls in $\overline{W^{\hat{\Xi}}}$, we only need to minimize $\big(\mathcal{A}^{u}V\big)(x,t,y)+H(x,t,u,y)-f_s(x,t,y,t)$ on $\overline{W^{\hat{\Xi}}}$. On the contrary, in Theorem \ref{verificationthm}, $\hat\Pi$ is the minimum value point of $\big(\mathcal{A}^{u}V\big)(x,t,y)+H(x,t,u,y)-f_s(x,t,y,t)$ on $\mathcal{Q}$ as the perturbed state-time-(singular)control triple may take values in the whole space $\mathcal{Q}$ when perturbing regular and singular control simultaneously. Furthermore, the condition $\big(\mathcal{A}^{\hat{\Pi}}V\big)(x,t,y)+H\big(x,t,\hat{\Pi}(x,t,y),y\big)-f_s(x,t,y,t)\ge 0$ on $\mathcal{Q}$ in \eqref{mildv} is required by $\hat\Pi$ because the perturbed state-time-(singular)control triple may take values in the whole space $\mathcal{Q}$ when perturbing singular control for fixed regular control law $\hat\Pi$. Besides, as  the cost function $c$ is non-positive, in Theorem \ref{mildverificationthm}, we can remove \eqref{integrability2} in the Condition (6) of Theorem \ref{verificationthm}.
\end{remark}

\section{Application to effort and dividend problem}\label{Application}
In this section, we illustrate the applicability of the theory to the effort and dividend problem. Specifically, solving the extended HJB system, we obtain a semi-explicit equilibrium regular-singular control law depending on the effort and dividend thresholds under a mixture of exponential discount functions and a pseudo-exponential discount function. The convexity of the value function is rigorously established, the equilibrium conditions are verified and the existence of the effort and dividend thresholds is proved.
\subsection{Cost minimization with costly effort and frictional dividends}
We consider a firm or financial institution whose management seeks to minimize the expected total net cost of running the business until failure. The firm engages in a costly operational activity and distributes dividends to shareholders. The ruin event terminates all operations. The objective is to find the optimal trade-off between the cost of effort and the benefits of dividend payouts. For a detailed explanation of this model, we refer the reader to \cite{cadenillas2007optimal}.

For any fixed $(x,t)\in \mathcal{\tilde{Q}}:=[0,+\infty)\times[0,+\infty)$, suppose that the controlled surplus level of the firm $X^{\pi,\xi}$ under effort $\pi$ and dividend $\xi$ evolves according to the SDE:
\begin{equation}
\left\{\begin{array}{l}
\mathrm{d}X^{\pi,\xi}_{r}=\left(\mu+\pi_r\right)\mathrm{d}r+\sigma\mathrm{d}B_{r}-\mathrm{d}\xi_{r},\ r\in[t,\tau^{\pi,\xi}],\\
X^{\pi,\xi}_{t-}=x,
\end{array}\right.\nonumber
\end{equation}
where $\mu > 0$ is the baseline expected growth rate of the surplus, generated by the firm's core, non-discretionary business activities. This drift is autonomous and requires no active managerial intervention.  The regular control process $\pi=\left\{\pi_r\right\}_{r\ge t}$ is the level of effort the firm applies, representing the intensity of a costly revenue-enhancing activity. This activity could be the exploitation of a risk-free arbitrage opportunity, a performance-linked subsidy program, or an intensive marketing campaign. The activity generates a deterministic benefit by increasing the surplus drift (hence the term in the SDE), but the activity itself is costly to operate. This operational cost is what the objective functional seeks to minimize. Suppose that the level of effort is chosen from the interval $U=[0,M]$, where $M>0$ is the maximum effort level. $\sigma > 0$ is a constant diffusion coefficient. The term captures the exogenous, irreducible baseline risk inherent in the firm's operating environment. This risk is structural and entirely independent of management's actions. $\xi=\left\{\xi_r\right\}_{r\ge t}$ is the singular control process which is nondecreasing, c\`adl\`ag, and $\Delta \xi_r=\xi_r-\xi_{r-}\le X_{r-}^{\pi,\xi}$. It represents the cumulative dividends distributed to shareholders, and the firm cannot pay more dividends than it owns. $\tau^{\pi,\xi}:=\inf\big\{r\ge t\bigm|X^{\pi,\xi}_{r}= 0\big\}$ denotes the ruin time, which is the first time the surplus level hits zero.

The firm's objective is to choose the effort intensity and the dividend payment to minimize the total expected discounted net cost until the ruin time. The firm's objective is to find a regular-singular control $(\pi,\xi)$ that minimizes the following objective functional
\begin{equation}
\begin{aligned}
J(x, t; \pi,\xi):=&\mathbb{E}_{x, t}\left[\int_t^{\tau^{\pi,\xi}} \beta(r-t)\pi_r\mathrm{d}r -c\int_t^{\tau^{\pi,\xi}} \beta(r-t)  \mathrm{d} \xi_r\right],
\end{aligned}\label{objexmple}
\end{equation}
where $\mathbb{E}_{x, t}$ denotes the expectation conditioning on $X^{\pi,\xi}_{t-}=x$. $\beta\in C([0,+\infty))$ is a general discount function which is non-negative and nonincreasing with $\beta(0)=1$. 

$\int_t^{\tau^{\pi,\xi}} \beta(r-t)\pi_r\mathrm{d}r $ in the objective is the operational cost of effort. The term $\pi_r$ here represents the instantaneous cost rate of the revenue-enhancing activity. Despite its beneficial effect on the surplus drift, the activity itself consumes resources that constitute a direct cost to the firm. Thus the firm seeks to reduce it.

$c\int_t^{\tau^{\pi,\xi}} \beta(r-t)  \mathrm{d} \xi_r$ is the benefit of dividend distributions. This benefit reduces the firm's overall net cost. The parameter $c\in(0,1)$ is the net fractional value that shareholders derive from each unit of surplus distributed as dividends. For every unit of surplus paid out, a fraction $1-c$ is lost to taxes, transaction costs, or agency frictions. Only the remaining fraction $c$ constitutes a genuine benefit to the firm's stakeholders. A larger dividend payout implies a larger deduction from the total cost, making the firm better off.

First of all, we can obtain the following simple proposition for the effort and dividend problem \eqref{objexmple}. Specifically, a mild equilibrium is constructed, yet not a standard equilibrium, highlighting the distinction between the two definitions.

\begin{proposition}\label{toyeg}
Suppose $M\ge \frac{c\mu}{1-c}$ and $\beta\in C^1([0,+\infty))$. Let $\hat\Pi(x,t):=M$ and  $\hat\Xi:=\big(W^{\hat{\Xi}},P^{\hat{\Xi}}\big)$, where
\begin{equation}
\begin{aligned}
W^{\hat{\Xi}}:=\emptyset,\ P^{\hat{\Xi}}:=\mathcal{\tilde{Q}}.
\end{aligned}\nonumber
\end{equation}
Then, $(\hat\Pi,\hat\Xi)$ is a mild equilibrium for the effort and dividend problem \eqref{objexmple}. Furthermore, $(\hat\Pi,\hat\Xi)$ is not an equilibrium for the objective \eqref{objexmple}.
\end{proposition} 
\begin{proof}
Our proof will proceed according to the definition. First, it is easy to show that $(\hat\Pi,\hat\Xi)$ is an admissible regular-singular control law for the objective \eqref{objexmple}. Then, we show that $(\hat\Pi,\hat\Xi)$ is a mild equilibrium.

According to Definition \ref{mildequil}, for any initial $(x,t)\in\mathcal{\tilde{Q}}$, we know that $J\big(x, t; \hat\pi,\hat\xi\big)=-cx$. For any pair of perturbations process $(u,\eta)$ satisfying the conditions specified in Item (a) of Definition \ref{mildequil}, and for any $h>0$, we have $J(x, t; \pi^h,\xi^h)=-cx$. Therefore, Item (a) of Definition \ref{mildequil} holds. 

For any pair of perturbations process $(u,\eta)$ satisfying the conditions specified in Item (b) of Definition \ref{mildequil}, and for any $h\in(0,\tilde h)$, we have $\tau^{\pi^h,\xi^h}=\tau^{u,\eta}\wedge (t+h)$ and
\begin{equation}
\begin{aligned}
X^{\pi^h,\xi^h}_{r}=&x+\int_{t}^{r}\big(\mu+\pi^h_l\big)\mathrm{d} l+\sigma (B_r-B_t)-\xi^h_r+\xi^h_{t-}\\
=& x+(\mu+M)(r-t)+\sigma (B_r-B_t)-\eta_r+\eta_{t-},\ \forall r\in[t,\tau^{u,\eta}\wedge ((t+h)-)].
\end{aligned}\nonumber
\end{equation}
Therefore,
{\small 
\begin{align*}
&J(x, t; \pi^h,\xi^h)=\mathbb{E}_{x, t}\Big [\int_t^{\tau^{\pi^h,\xi^h}} \beta(r-t)\pi^h_r\mathrm{d}r -c\int_t^{\tau^{\pi^h,\xi^h}} \beta(r-t)  \mathrm{d} \xi^h_r\Big]\\
&\ge\mathbb{E}_{x, t}\Big[\beta\big((\tau^{u,\eta}-t)\wedge h\big)\big((\tau^{u,\eta}-t)\wedge h\big)M-c\int_t^{\tau^{\pi^h,\xi^h}} \beta(r-t)  \mathrm{d} \xi^h_r\Big ]\\
&\ge\mathbb{E}_{x, t}\Big[\beta\big((\tau^{u,\eta}\!-\!t)\wedge h\big)\big((\tau^{u,\eta}\!-\!t)\wedge h\big)M\!-\!c\eta_{\tau^{u,\eta}\wedge ((t+h)-)}\!+\!c\eta_{t-}\!-\!c\beta\big((\tau^{u,\eta}\!-\!t)\wedge h\big)X^{\pi^h,\xi^h}_{\tau^{u,\eta}\wedge ((t+h)-)}\Big]\\
&=\mathbb{E}_{x, t}\bigg[\beta\big((\tau^{u,\eta}-t)\wedge h\big)\big((\tau^{u,\eta}-t)\wedge h\big)M-c\eta_{\tau^{u,\eta}\wedge ((t+h)-)}+c\eta_{t-}-c\beta\big((\tau^{u,\eta}-t)\wedge h\big)\Big[x\\
&+(\mu+M)\big((\tau^{u,\eta}-t)\wedge h\big)+\sigma (B_{\tau^{u,\eta}\wedge (t+h)}-B_t)-\eta_{\tau^{u,\eta}\wedge ((t+h)-)}+\eta_{t-}\Big]\bigg].
\end{align*}
}
Since
\begin{align*}
&\big|\mathbb{E}_{x, t}\big[\beta\big((\tau^{u,\eta}-t)\wedge h\big)(B_{\tau^{u,\eta}\wedge (t+h)}-B_t)\big]\big|=\big|\mathbb{E}_{x, t}\big[\big[\beta\big((\tau^{u,\eta}-t)\wedge h\big)-1\big](B_{\tau^{u,\eta}\wedge (t+h)}-B_t)\big]\big|\\
&\le \tilde{C}h\mathbb{E}_{x, t}\big[|B_{\tau^{u,\eta}\wedge (t+h)}-B_t|\big]\le \tilde{C}h^{3/2},
\end{align*}
where $\tilde{C}=\max\limits_{r\in [0,\tilde{h}]}{|\beta^{\prime}(r)|}$, combining with $M\ge \frac{c\mu}{1-c}$, we have
{\small\begin{align*}
J(x, t; \pi^h,\xi^h)\ge-c\mathbb{E}_{x, t}\left[\left[1-\beta\big((\tau^{u,\eta}-t)\wedge h\big)\right]\big(\eta_{\tau^{u,\eta}\wedge ((t+h)-)}-\eta_{t-}\big)+\beta\big((\tau^{u,\eta}-t)\wedge h\big)x\right]-c\sigma\tilde{C}h^{3/2}.
\end{align*}}
Thus, combining with the facts that $\Delta \eta_t=\eta_t-\eta_{t-}\le x$ and $\eta$ satisfies Property \eqref{growthcond}, we have {\small
\begin{equation}
\begin{aligned}
&\frac{J\big(x, t; \pi^h,\xi^h\big)-J\big(x, t; \hat\pi,\hat\xi\big)}{h}\ge-\frac{c}{h}\mathbb{E}_{x, t}\left[\left(1-\beta\big((\tau^{u,\eta}-t)\wedge h\big)\right)\big(\eta_{\tau^{u,\eta}\wedge ((t+h)-)}-\eta_{t-}-x\big)\right]-c\sigma\tilde{C}h^{1/2}\\&\ge-\frac{c}{h}\mathbb{E}_{x, t}\left[(1-\beta(h))\big(\eta_{\tau^{u,\eta}\wedge ((t+h)-)}-\eta_{t}\big)\right]-c\sigma\tilde{C}h^{1/2}\ge -\frac{c}{h}\mathbb{E}_{x, t}\left[(1-\beta(h))\big(\eta_{(t+h)-}-\eta_{t}\big)\right]-c\sigma\tilde{C}h^{1/2}\\&\ge-c\tilde{M}(1-\beta(h))-c\sigma\tilde{C}h^{1/2}. 
\end{aligned}\nonumber
\end{equation}
}
Hence, $\liminf\limits_{h\downarrow 0}\frac{J\big(x,t,y;\pi^{h},\xi^{h}\big)-J\big(x,t,y;\hat{\pi},\hat{\xi}\big)}{h}\ge\liminf\limits_{h\downarrow 0}\big(-c\tilde{M}(1-\beta(h))\big)=0$ and Item (b) in Definition \ref{mildequil} holds. Thus, $(\hat\Pi,\hat\Xi)$ is a mild equilibrium.

Finally, we show that $(\hat\Pi,\hat\Xi)$ is not an equilibrium. For initial $(x,t)\in\mathcal{\tilde{Q}}$ such that $x>0$ and $\beta^{\prime}(0)x+\mu>0$,  let $u_r=0, \eta_r=\eta_{t-}, \forall r\ge t$. We have the pair of perturbations process $(u,\eta)\in \mathcal{\tilde D}^{x,t;\hat\Pi,\hat\Xi}$. For any $h>0$, we have 
\begin{equation}
\begin{aligned}
X^{\pi^h,\xi^h}_{r}=x+\mu(r-t)+\sigma (B_r-B_t),\ \forall r\in[t,\tau^{u,\eta}\wedge ((t+h)-)].
\end{aligned}\nonumber
\end{equation}
Therefore, 
\begin{equation}
\begin{aligned}
J\big(x, t; \pi^h,\xi^h\big)=& -c\mathbb{E}_{x, t}\left[\beta\big((\tau^{u,\eta}-t)\wedge h\big)X^{\pi^h,\xi^h}_{\tau^{u,\eta}\wedge ((t+h)-)}\right]\\=&-c\mathbb{E}_{x, t}\left[\beta\big((\tau^{u,\eta}-t)\wedge h\big)\left(x+\mu\big((\tau^{u,\eta}-t)\wedge h\big)\right)\right]+O(h^{3/2}).\end{aligned}\nonumber
\end{equation}
Hence, because $\tau^{u,\eta}>t$,  and using the Dominated Convergence Theorem, we have
 \begin{equation}
\begin{aligned}
&\liminf\limits_{h\downarrow 0}\frac{J\big(x,t,y;\pi^{h},\xi^{h}\big)\!-\!J\big(x,t,y;\hat{\pi},\hat{\xi}\big)}{h} \\ &=\liminf\limits_{h\downarrow 0}\left\{\!-\frac{c}{h}\mathbb{E}_{x, t}\left[\beta\big((\tau^{u,\eta}\!-\!t)\wedge h\big)\left(x\!+\!\mu\big((\tau^{u,\eta}\!-\!t)\wedge h\big)\right)\!-\!x\right]\right\}\\
&=-c(\beta^\prime(0)x+\mu)<0.
\end{aligned}\nonumber
\end{equation}
Thus, $(\hat\Pi,\hat\Xi)$ is not an equilibrium.
\end{proof}

Suppose $M\ge \frac{c\mu}{1-c}$ and $\beta\in C^1([0,+\infty))$. 
Proposition \ref{toyeg} shows that applying the maximal level of effort and paying the dividend once the surplus exceeds the threshold $0$ is the mild equilibrium regular-singular control law but not the equilibrium.

 The extended HJB system associated with the effort and dividend problem is
\begin{equation}
\begin{aligned}\label{extendedhjb}
&\begin{aligned}\min\Big\{&V_t(x,t)+\big(\mu+\hat\Pi(x,t)\big)V_x(x,t)+\frac{1}{2}\sigma^2V_{xx}(x,t)+\hat\Pi(x,t)-f_s(x,t,t),\\
&-c-V_{x}(x,t)\Big\}=0,\ \forall (x,t)\in\mathcal{\tilde{Q}},
\end{aligned}\\
&V(0,t)=0,\ \forall t\ge 0,\\
&f_t(x,t,s)\!+\!\big(\mu\!+\!\hat\Pi(x,t)\big)f_x(x,t,s)\!+\!\frac{1}{2}\sigma^2f_{xx}(x,t,s)\!+\!\beta(t\!-\!s)\hat\Pi(x,t)=0,\ \forall (x,t)\in \overline{W^{\hat{\Xi}}}, s\in[0,t],\\
& -\beta(t-s)c-f_{x}(x,t,s)=0,\ \forall (x,t)\in P^{\hat{\Xi}}, s\in[0,t],\\
&f(0,t,s)=0,\ \forall t\ge 0, s\in[0,t],
\end{aligned}
\end{equation}
where $\hat\Pi(x,t)\in\mathop{\mathrm{argmin}}\limits_{u\in[0,M]}\left\{V_t(x,t)+(\mu+u)V_x(x,t)+\frac{1}{2}\sigma^2V_{xx}(x,t)+u-f_s(x,t,t)\right\}, \forall (x,t)\in\mathcal{\tilde{Q}}$, and
\begin{equation}
\begin{aligned}
W^{\hat{\Xi}}:=\big\{(x,t)\in\mathcal{\tilde{Q}}\bigm|-c-V_{x}(x,t)>0\big\},\ \ P^{\hat{\Xi}}:=\big\{(x,t)\in\mathcal{\tilde{Q}}\bigm|-c-V_{x}(x,t)=0\big\}.
\end{aligned}\nonumber
\end{equation}

\begin{remark}
From the perspective of the extended HJB system, $\hat\Pi(x,t)=M$ is not the minimum value point to $\min\limits_{u\in[0,M]}\left\{uV_x(x,t)+u\right\}=\min\limits_{u\in[0,M]}\left\{(1-c)u\right\}$. Thus, $(\hat\Pi,\hat\Xi)$ in Proposition \ref{toyeg} is not a standard equilibrium for the objective \eqref{objexmple}. Furthermore, even if we set $\tilde{\Pi}(x,t) = 0$ being the minimum value point, $V(t,x)=-cx, f(x,t,s)=-c\beta(t-s)x$ is still not a solution to the extended HJB system as $V_t(x,t)+\big(\mu+\tilde{\Pi}(x,t)\big)V_x(x,t)+\frac{1}{2}\sigma^2V_{xx}(x,t)+\tilde{\Pi}(x,t)-f_s(x,t,t)=-c\mu-c\beta^{\prime}(0)x\ge 0$ does not hold identically on $\mathcal{\tilde{Q}}$, whereas $\hat\Pi(x,t)=M$ satisfies the inequality when $M\ge \frac{c\mu}{1-c}$ and $\beta\in C^1([0,+\infty))$. Thus, $(\hat\Pi,\hat\Xi)$ is only a mild equilibrium.
\end{remark}
In the following subsections, we will derive some non-trivial equilibria based on  solving the extended HJB system \eqref{extendedhjb} under certain special forms of discount functions.

\subsection{A mixture of exponential discount functions}
In this subsection, we give a semi-explicit solution to the extended HJB system above under a discount function, which is a mixture of exponential discount functions:
\begin{equation}\label{eq:mix_exp_discount}
\beta(t) := \sum\limits_{i=1}^{2}\omega_i \mathrm{e}^{-\delta_i t},\  \forall t\ge 0,
\end{equation}
where $\omega_i>0$, $\sum\limits_{i=1}^{2}\omega_i=1$, $0<\delta_1< \delta_2$, and $\delta_i$ is the constant discount rate. \eqref{eq:mix_exp_discount} can be viewed as a special case of weighted discount where the firm consists of two groups of shareholders with different discount rates (see \cite{ebert2020weighted} for more detail). 

Because the problem is time-homogeneous, we conjecture that $V$ is independent of time $t$. Inspired by \cite{hu2025equilibrium}, $V$ and $f$ can be decomposed into two components $V_1(x)$ and $V_2(x)$ with respect to the discount rates $\delta_{1}$ and $\delta_{2}$:
\begin{equation}\begin{aligned}\label{fv}
	&f(x,t,s)=\sum\limits_{i=1}^{2}\omega_i \mathrm{e}^{-\delta_i (t-s)}V_i(x),\  \forall (x,t)\in\mathcal{\tilde{Q}}, s\in[0,t],\\
	&V(x)=\sum\limits_{i=1}^{2}\omega_i V_i(x),\ \forall x\ge 0.\end{aligned}\end{equation}
Furthermore, we conjecture that the candidate regular-singular control law $(\hat\Pi,\hat\Xi)$ is threshold-type, specified by two thresholds $0 < x_1 < x_2$:
\begin{equation}\label{hatpi}
\hat \Pi(x,t) := \begin{cases}
M,&0 \le x < x_1,\\
0,&x \geq x_1,
\end{cases}\quad\forall (x,t)\in\mathcal{\tilde{Q}},
\end{equation}
and
\begin{equation}
\begin{aligned}
W^{\hat{\Xi}}:=\big\{(x,t)\in\mathcal{\tilde{Q}}\bigm|0\leq x< x_2\big\},\ \
P^{\hat{\Xi}}:=\big\{(x,t)\in\mathcal{\tilde{Q}}\bigm|x\ge x_2\big\}.
\end{aligned}\label{hatxi}
\end{equation}
%To ensure that $(\hat\Pi,\hat\Xi)$ forms an equilibrium control law for the HJB system, the value function $V(x)$ is required to satisfy the following marginal properties: $V'(x) < -1$ for $0\le x < x_1$, $V'(x_1)=-1$, $-1 < V'(x) < -c$ for $x_1 < x < x_2$, and $V'(x) =-c$ for $x\ge x_2$.

Consequently, each component $V_i(x)$ ($i=1,2$) satisfies the following system of ODEs
\begin{equation}\label{eq:Vi_ODE}
\begin{cases}
\frac{1}{2} \sigma^2 V^{\prime\prime}_i(x) + (\mu + M) V^\prime_i(x) - \delta_i V_i(x) +M = 0, & x \in [0, x_1),\\
\frac{1}{2} \sigma^2 V^{\prime\prime}_i(x) + \mu V^\prime_i(x) - \delta_i V_i(x) = 0, \qquad & x \in [x_1, x_2),\\
-c- V^\prime_i(x) = 0, & x \in [x_2, +\infty),\\
V_i(0)=0.
\end{cases}
\end{equation}
Define 
\begin{equation}\begin{aligned}\nonumber
&\theta_{i,1}:= \frac{-(\mu+M) + \sqrt{(\mu+M)^2 + 2 \delta_i \sigma^2}}{\sigma^2} ,\quad \theta_{i,2}:= \frac{\mu+M + \sqrt{(\mu+M)^2 + 2 \delta_i \sigma^2}}{\sigma^2},\\
&\theta_{i,3}:= \frac{-\mu + \sqrt{\mu^2 + 2 \delta_i \sigma^2}}{\sigma^2} ,\quad\theta_{i,4}:= \frac{\mu + \sqrt{\mu^2 + 2 \delta_i \sigma^2}}{\sigma^2}.\end{aligned}\end{equation}
Then a general solution to the ODEs \eqref{eq:Vi_ODE} has the form
\begin{equation}\label{Vi}
V_i(x) = \begin{cases}
\frac{M}{\delta_i} + A_{i,1} \mathrm{e}^{\theta_{i,1}x} + A_{i,2} \mathrm{e}^{-\theta_{i,2}x}, & x \in [0, x_1),\\
A_{i,3} \mathrm{e}^{\theta_{i,3} x} + A_{i,4} \mathrm{e}^{-\theta_{i,4}x}, & x \in [x_1, x_2),\\
-cx + A_{i,5}, \qquad & x \in [x_2, +\infty),
\end{cases}
\end{equation}
where $A_{i,j}$ are the constants to be determined. To determine the values of $A_{i,j}$, we use the boundary condition and smooth fit principle, i.e.,
\begin{equation}\nonumber
V_i(0) \!=\! 0,\ V_i(x_1 -) \!=\! V_i(x_1+), \ V^\prime_i(x_1 -) \!=\! V^\prime_i(x_1+),\ V_i(x_2 -) \!=\! V_i(x_2+),\ V^\prime_i(x_2 -) \!=\! V^\prime_i(x_2+).
\end{equation}
Then $A_{i,j}$ satisfy the following linear equation system:
\begin{equation}\nonumber
\begin{cases}
\frac{M}{\delta_i} + A_{i,1} + A_{i,2} = 0, \\
\frac{M}{\delta_i} + A_{i,1} \mathrm{e}^{\theta_{i,1}x_1} + A_{i,2} \mathrm{e}^{-\theta_{i,2}x_1} = A_{i,3} \mathrm{e}^{\theta_{i,3}x_1} + A_{i,4} \mathrm{e}^{-\theta_{i,4}x_1}, \\
A_{i,1} \theta_{i,1} \mathrm{e}^{\theta_{i,1}x_1} - A_{i,2} \theta_{i,2} \mathrm{e}^{-\theta_{i,2}x_1} = A_{i,3} \theta_{i,3} \mathrm{e}^{\theta_{i,3}x_1} - A_{i,4} \theta_{i,4} \mathrm{e}^{-\theta_{i,4}x_1}, \\
A_{i,3} \mathrm{e}^{\theta_{i,3}x_2} + A_{i,4} \mathrm{e}^{-\theta_{i,4}x_2} = -cx_2 + A_{i,5},\\
A_{i,3}\theta_{i,3} \mathrm{e}^{\theta_{i,3}x_2} - A_{i,4} \theta_{i,4}\mathrm{e}^{-\theta_{i,4}x_2}  = -c.
\end{cases}
\end{equation}
Solving the above system of linear equations, we have

\begin{equation}\nonumber
\begin{aligned}
 A_{i,j}  =\frac{MB_{i,j}}{\delta_i D_i}+\frac{cC_{i,j}}{D_i},\ j=1,2,3,4,\end{aligned}
\end{equation}
where
\begin{align*}
B_{i,1} &=\theta_{i,3}\theta_{i,4} \mathrm{e}^{\theta_{i,3} x_1 - \theta_{i,4} x_2} - \theta_{i,3}\theta_{i,4} \mathrm{e}^{-\theta_{i,4} x_1 + \theta_{i,3} x_2}-\theta_{i,4}(\theta_{i,2} +\theta_{i,3}) \mathrm{e}^{(-\theta_{i,2} + \theta_{i,3}) x_1- \theta_{i,4} x_2} \\&\quad- \theta_{i,3}(\theta_{i,2} - \theta_{i,4})  \mathrm{e}^{(-\theta_{i,2} - \theta_{i,4}) x_1 + \theta_{i,3} x_2},\\
B_{i,2} &=-\theta_{i,3}\theta_{i,4} \mathrm{e}^{\theta_{i,3} x_1 - \theta_{i,4} x_2} + \theta_{i,3}\theta_{i,4} \mathrm{e}^{-\theta_{i,4} x_1 + \theta_{i,3} x_2}-\theta_{i,4}(\theta_{i,1} -\theta_{i,3}) \mathrm{e}^{(\theta_{i,1} + \theta_{i,3}) x_1- \theta_{i,4} x_2} \\&\quad- \theta_{i,3}(\theta_{i,1} + \theta_{i,4})  \mathrm{e}^{(\theta_{i,1} - \theta_{i,4}) x_1 + \theta_{i,3} x_2},\\
B_{i,3} &=\theta_{i,2}\theta_{i,4} \mathrm{e}^{-\theta_{i,2} x_1 - \theta_{i,4} x_2} + \theta_{i,1}\theta_{i,4} \mathrm{e}^{\theta_{i,1} x_1- \theta_{i,4} x_2}-\theta_{i,4}(\theta_{i,1} +\theta_{i,2}) \mathrm{e}^{(\theta_{i,1} - \theta_{i,2}) x_1- \theta_{i,4} x_2} ,\\
B_{i,4} &=\theta_{i,2}\theta_{i,3} \mathrm{e}^{-\theta_{i,2} x_1 + \theta_{i,3} x_2} + \theta_{i,1}\theta_{i,3} \mathrm{e}^{\theta_{i,1} x_1+\theta_{i,3} x_2}-\theta_{i,3}(\theta_{i,1} +\theta_{i,2}) \mathrm{e}^{(\theta_{i,1} - \theta_{i,2}) x_1+ \theta_{i,3} x_2},  \\
C_{i,1} &=-(\theta_{i,4}+\theta_{i,3})\mathrm{e}^{(\theta_{i,3}-\theta_{i,4})x_1},\\
C_{i,2} &=(\theta_{i,4}+\theta_{i,3})\mathrm{e}^{(\theta_{i,3}-\theta_{i,4})x_1},\\
C_{i,3} &=(\theta_{i,4}-\theta_{i,2})\mathrm{e}^{(-\theta_{i,2}-\theta_{i,4})x_1}-(\theta_{i,1}+\theta_{i,4})\mathrm{e}^{(\theta_{i,1}-\theta_{i,4})x_1},\\
C_{i,4} &=(\theta_{i,2}+\theta_{i,3})\mathrm{e}^{(-\theta_{i,2}+\theta_{i,3})x_1}+(\theta_{i,1}-\theta_{i,3})\mathrm{e}^{(\theta_{i,1}+\theta_{i,3})x_1},
\\
D_i=&\theta_{i,4}(\theta_{i,1} - \theta_{i,3}) \mathrm{e}^{(\theta_{i,1}+\theta_{i,3}) x_1 - \theta_{i,4} x_2}  + \theta_{i,3}(\theta_{i,1}+\theta_{i,4})  \mathrm{e}^{(\theta_{i,1} - \theta_{i,4}) x_1 + \theta_{i,3} x_2}  \\
&+\theta_{i,4}(\theta_{i,2} + \theta_{i,3})\mathrm{e}^{(-\theta_{i,2}+\theta_{i,3} ) x_1 - \theta_{i,4} x_2} +\theta_{i,3}(\theta_{i,2} - \theta_{i,4})  \mathrm{e}^{(-\theta_{i,2}-\theta_{i,4}) x_1 + \theta_{i,3} x_2},	
\end{align*}
and $A_{i,5}=cx_2+A_{i,3} \mathrm{e}^{\theta_{i,3}x_2} + A_{i,4} \mathrm{e}^{-\theta_{i,4}x_2}$.
\vskip 3pt 
To determine the values of $x_1$ and $x_2$, we apply the smooth-fit conditions $V^\prime(x_1) = -1$ (or equivalently, $V^{\prime\prime}(x_1-)=V^{\prime\prime}(x_1+)$) and $V^{\prime\prime}(x_2-) = 0=V^{\prime\prime}(x_2+)$, that is,
\begin{equation}\label{eq:x1x2}
\begin{cases}
\sum\limits_{i=1}^{2}\omega_i \big(A_{i,1} \theta_{i,1} \mathrm{e}^{\theta_{i,1}x_1} - A_{i,2} \theta_{i,2} \mathrm{e}^{-\theta_{i,2}x_1}\big)=-1,\\
\sum\limits_{i=1}^{2}\omega_i\big(A_{i,3}\theta_{i,3}^2 \mathrm{e}^{\theta_{i,3}x_2} + A_{i,4} \theta_{i,4}^2\mathrm{e}^{-\theta_{i,4}x_2}\big) = 0.
\end{cases}
\end{equation}
Assuming that there exist $0 < x_1 < x_2$ solving \eqref{eq:x1x2}, we are able to verify that \eqref{fv} is a solution to the extended HJB system \eqref{extendedhjb} under the discount function \eqref{eq:mix_exp_discount}. Our proof is inspired by \cite{hu2025equilibrium}. The crux is to establish the convexity of $V$.

\begin{proposition}\label{convex1}
Suppose that there exist  $0 < x_1 < x_2$ that solve \eqref{eq:x1x2}. Then, the functions $V$ and $f$ specified in \eqref{fv}, with $V_i$ given by \eqref{Vi}, constitute a solution to the extended HJB system \eqref{extendedhjb} under the discount function \eqref{eq:mix_exp_discount}. In addition, the function $V$ is decreasing and convex on $[0,+\infty)$.
\end{proposition}
\begin{proof}
{\bf It is sufficient to show $V'(x) < -1$  for  $0\le x < x_1$, $V'(x_1)=-1$;  $-1 < V'(x) < -c$  for  $x_1 < x < x_2$;  $V'(x) =-c$ for $x\ge x_2$. }
\vskip 5pt
For fixed $0 < x_1 < x_2$, it is easy to show 
\begin{equation}\nonumber\begin{aligned}
D_{i}=&\big(\mathrm{e}^{-\theta_{i,2} x_1}-\mathrm{e}^{\theta_{i,1} x_1}\big)\big(\theta_{i,3}\theta_{i,4} \mathrm{e}^{\theta_{i,3} x_1 - \theta_{i,4} x_2} - \theta_{i,3}\theta_{i,4} \mathrm{e}^{-\theta_{i,4} x_1 + \theta_{i,3} x_2}\big)\\&+
\big(\theta_{i,2}\mathrm{e}^{-\theta_{i,2} x_1}+\theta_{i,1}\mathrm{e}^{\theta_{i,1} x_1}\big)\big(\theta_{i,4} \mathrm{e}^{\theta_{i,3} x_1 - \theta_{i,4} x_2} + \theta_{i,3} \mathrm{e}^{-\theta_{i,4} x_1 + \theta_{i,3} x_2}\big)>0,\\
B_{i,1}=&\big(1-\mathrm{e}^{-\theta_{i,2} x_1}\big)\big(\theta_{i,3}\theta_{i,4} \mathrm{e}^{\theta_{i,3} x_1 - \theta_{i,4} x_2} - \theta_{i,3}\theta_{i,4} \mathrm{e}^{-\theta_{i,4} x_1 + \theta_{i,3} x_2}\big)\\&-\theta_{i,2}\theta_{i,4} \mathrm{e}^{(-\theta_{i,2} + \theta_{i,3}) x_1- \theta_{i,4} x_2} - \theta_{i,2}\theta_{i,3}  \mathrm{e}^{(-\theta_{i,2} - \theta_{i,4}) x_1 + \theta_{i,3} x_2} < 0,\end{aligned}\end{equation}
and $C_{i,1} < 0$,  thus $A_{i,1} < 0$. 
\vskip 3pt
First, consider $V_i(x)$ for $x \in [x_2, +\infty)$. It is straightforward that $V^\prime_i(x) = -c < 0$ and $V^{\prime\prime}_i(x) =0$ for $x \in [x_2, +\infty)$. Therefore, $V^\prime(x) = -c < 0$ and $V^{\prime\prime}(x) =0$ for $x \in [x_2, +\infty)$. 
\vskip 4pt
Second, for $x \in [x_1, x_2)$, as  $V^\prime_i(x_2-) = V^\prime_i(x_2+)=-c < 0$, it follows that  $V^\prime_i(x_2-)=A_{i,3} \theta_{i,3} \mathrm{e}^{\theta_{i,3} x_2} - A_{i,4} \theta_{i,4} \mathrm{e}^{-\theta_{i,4} x_2} < 0$. By examining possible signs of $A_{i,3}$ and $A_{i,4}$, we claim that $V^\prime_i(x) < 0$ for all $x \in [x_1, x_2)$.  
\vskip 4pt
If $A_{i,3} \ge 0$, then in order to make $V^\prime_i(x_2-) < 0$, we can only have $A_{i,4} > 0$,  hence $V^{\prime\prime}_i(x) = A_{i,3} \theta_{i,3}^2 \mathrm{e}^{\theta_{i,3} x} + A_{i,4} \theta_{i,4}^2 \mathrm{e}^{-\theta_{i,4} x} > 0$ for all $x \in [x_1, x_2)$, indicating $V^\prime_i(x) < V^\prime_i(x_2-) < 0$, $\forall x \in [x_1, x_2)$. If $A_{i,3} < 0$ and $A_{i,4} \ge 0$, then $V^\prime_i(x)=A_{i,3} \theta_{i,3} \mathrm{e}^{\theta_{i,3} x} - A_{i,4} \theta_{i,4} \mathrm{e}^{-\theta_{i,4} x} < 0, \forall x \in [x_1, x_2)$. If $A_{i,3} < 0$ and $A_{i,4} < 0$, suppose by contradiction that  $V_i^\prime(x_1+) \ge 0$.  According to the smooth-fit condition, $V_i^\prime(x_1-) = A_{i,1} \theta_{i,1} \mathrm{e}^{\theta_{i,1}x_1} - A_{i,2} \theta_{i,2} \mathrm{e}^{-\theta_{i,2}x_1}=V_i^\prime(x_1+) \ge 0$.
As $A_{i,1} < 0$,  $A_{i,2}$ can only be negative, resulting in $V_i^{\prime\prime}(x) = A_{i,1} \theta_{i,1}^2 \mathrm{e}^{\theta_{i,1}x} + A_{i,2} \theta_{i,2}^2 \mathrm{e}^{-\theta_{i,2}x} < 0$ for all $x \in [0, x_1)$.   Thus, $V_i^\prime(x) > V_i^\prime(x_1-)\ge 0$ for all $x \in [0, x_1)$,  leading to $V_i(0) < V_i(x_1-) = V_i(x_1+) = A_{i,3} \mathrm{e}^{\theta_{i,3}x_1} + A_{i,4} \mathrm{e}^{-\theta_{i,4}x_1} < 0$ which contradicts the boundary condition $V_i(0) = 0$.  
 Thus $V_i^\prime(x_1+)<0$, and $V^{\prime\prime}_i(x) =  A_{i,3} \theta_{i,3}^2 \mathrm{e}^{\theta_{i,3} x} + A_{i,4} \theta_{i,4}^2 \mathrm{e}^{-\theta_{i,4} x} < 0$ for all $x \in [x_1, x_2)$, which shows  $V_i^\prime(x) \le V_i^\prime(x_1+) < 0$, $\forall x \in [x_1, x_2)$.
\vskip 4pt
Next, we show $V^{\prime\prime}(x) = \sum\limits_{i=1}^{2}\omega_i V^{\prime\prime}_i(x) > 0$, $\forall x \in [x_1, x_2)$. We know that $V^\prime(x_2-) = V^\prime(x_2+) = -c$ and $V^{\prime\prime}(x_2-)= 0$. Then we introduce a change of variable $\varphi(x) := \int_{x_1}^x \mathrm{e}^{\frac{-2 \mu a}{\sigma^2}} \mathrm{d}a$, $x \in [x_1, x_2)$,
and define $l(z) := \sum\limits_{i=1}^{2}\omega_il_i(z)$ with $l_i(z) := V_i^\prime(\varphi^{-1}(z))$, $z\in[0,\bar z)$, where $\bar z = \int_{x_1}^{x_2} \mathrm{e}^{\frac{-2 \mu a}{\sigma^2}} \mathrm{d}a$. As $l_i(z)=V_i^{\prime}(\varphi^{-1}(z))<0, \forall z\in[0,\bar z)$, according to ODEs \eqref{eq:Vi_ODE}, we have that $l_i(z)$ satisfies 
\begin{equation}\nonumber\begin{aligned}
 l_i^{\prime\prime}(z) =\left[\varphi^\prime(\varphi^{-1}(z))\right]^{-2} \frac{2 \delta_i}{\sigma^2} l_i(z)<0,\ z\in[0,\bar z).
\end{aligned}\end{equation}
Hence,
 \begin{equation}\nonumber\begin{aligned}
l^{\prime\prime}(z) = \sum\limits_{i=1}^{2}\omega_il_i^{\prime\prime}(z)
= \left[\varphi^\prime(\varphi^{-1}(z))\right]^{-2} \frac{2}{\sigma^2} \sum\limits_{i=1}^{2}\omega_i \delta_i l_i(z)<0,\ z\in[0,\bar z).
\end{aligned}\end{equation}
Therefore, 
\begin{eqnarray*}
l^{\prime}(z)>l^{\prime}(\bar z-)&=&\sum\limits_{i=1}^{2}\omega_il_i^{\prime}(\bar z-)=\left[\varphi^\prime(\varphi^{-1}(\bar z-))\right]^{-1}\sum\limits_{i=1}^{2}\omega_iV_i^{\prime\prime}(\varphi^{-1}(\bar z-))\\
&=&\left[\varphi^\prime(\varphi^{-1}(\bar z-))\right]^{-1}V^{\prime\prime}(x_2-)=0, \forall z\in[0,\bar z). 
\end{eqnarray*}
Because $l^{\prime}(z)=\left[\varphi^\prime(\varphi^{-1}(z))\right]^{-1}V^{\prime\prime}(\varphi^{-1}(z))$, we have $V^{\prime\prime}(x)>0, \forall x \in [x_1, x_2)$. Therefore, $-1<V^{\prime}(x)<-c, \forall x \in (x_1, x_2)$.
\vskip 4pt
Finally, for $x \in [0, x_1)$, it is already known that $V_i(0) = 0$, $V^\prime_i(x_1-) = V^\prime_i(x_1+) < 0$, and $V^\prime(x_1-) = V^\prime(x_1+) =-1$. Besides, according to ODEs \eqref{eq:Vi_ODE}, we have
\begin{equation}\nonumber\begin{aligned}
&V^{\prime\prime}(x_1-) =\sum\limits_{i=1}^{2}\omega_i V^{\prime\prime}_i(x_1-)= \frac{2}{\sigma^2} \Big[-(\mu+M) V^\prime(x_1-) - M + \sum\limits_{i=1}^{2}\omega_i\delta_iV_i(x_1-)\Big ]\\
&= \frac{2}{\sigma^2} \Big[-\mu V^\prime(x_1+) + \sum\limits_{i=1}^{2}\omega_i\delta_iV_i(x_1+)\Big]= \sum\limits_{i=1}^{2}\omega_iV_i^{\prime\prime}(x_1+) = V^{\prime\prime}(x_1+) > 0.
\end{aligned}\end{equation}
Similarly, we introduce a change of variable $\phi(x) := \int_{0}^x \mathrm{e}^{\frac{-2 (\mu+M) a}{\sigma^2}} \mathrm{d}a$, $x \in [0, x_1)$,
and define $m(z) := \sum\limits_{i=1}^{2}\omega_im_i(z)$ with $m_i(z) := V_i^\prime(\phi^{-1}(z))$, $z\in[0,\tilde z)$, where $\tilde z = \int_{0}^{x_1} \mathrm{e}^{\frac{-2 (\mu+M) a}{\sigma^2}} \mathrm{d}a$. According to ODEs \eqref{eq:Vi_ODE}, $m_i(z)$ is a solution to the following ODE:
\begin{equation}\nonumber\begin{aligned}
 m_i^{\prime\prime}(z) =\left[\phi^\prime(\phi^{-1}(z))\right]^{-2} \frac{2 \delta_i}{\sigma^2} m_i(z),\ z\in[0,\tilde z).
\end{aligned}\end{equation}
Hence,
 \begin{equation}\nonumber\begin{aligned}
m^{\prime\prime}(z) = \sum\limits_{i=1}^{2}\omega_im_i^{\prime\prime}(z)
= \left[\phi^\prime(\phi^{-1}(z))\right]^{-2} \frac{2}{\sigma^2} \sum\limits_{i=1}^{2}\omega_i \delta_i m_i(z),\ z\in[0,\tilde z).
\end{aligned}\end{equation}
We know 
\begin{eqnarray*}
m^{\prime}(\tilde z-)&=&\sum\limits_{i=1}^{2}\omega_im_i^{\prime}(\tilde z-)=\left[\phi^\prime(\phi^{-1}(\tilde z-))\right]^{-1}\sum\limits_{i=1}^{2}\omega_iV_i^{\prime\prime}(\phi^{-1}(\tilde z-))\\
&=&\left[\phi^\prime(\phi^{-1}(\tilde z-))\right]^{-1}V^{\prime\prime}(x_1-)=\left[\phi^\prime(\phi^{-1}(\tilde z-))\right]^{-1}V^{\prime\prime}(x_1+)>0, \\
m_i(\tilde z-)&=&V_i^\prime(x_1-)<0,\\
m(\tilde z-)&=&V^\prime(x_1-)=-1<0.
\end{eqnarray*}
Suppose by contradiction that there exists $z^* \in (0, \tilde z)$ such that $m^\prime(z^*) = 0$ and $m^\prime(z) > 0$ for $z \in (z^*, \tilde z)$.  Then $m^{\prime\prime}(z^*) \ge 0$, i.e., $m^{\prime\prime}(z^*)=\left[\phi^\prime(\phi^{-1}(z^*))\right]^{-2} \frac{2}{\sigma^2} \sum\limits_{i=1}^{2}\omega_i \delta_i m_i(z^*)\ge 0$, i.e., $\sum\limits_{i=1}^{2}\omega_i \delta_i m_i(z^*)\ge 0$. Because $m^\prime(z) > 0$ for $z \in (z^*, \tilde z)$, we have $m(z^*)<m(\tilde z-)=V^\prime(x_1-)<0$, i.e., $\sum\limits_{i=1}^{2}\omega_i m_i(z^*)< 0$. Recall that $\delta_1<\delta_2$, we have $m_2(z^*)>0$ and $m_1(z^*)<0$.  Combining with the fact that $m_2(\tilde z-)<0$, we know that there exists $z_0\in(z^*, \tilde z)$ such that $m_2(z_0)=0$. Based on \cite[Lemma 4.1]{shreve1984optimal}, $m^\prime_2$ has no zero point in $[0, \tilde z)$, thus $m^\prime_2(z)<0, \forall z\in[0, \tilde z)$. Therefore,
\begin{equation}\label{ineq:lprime}
\sum\limits_{i=1}^{2}\omega_i \delta_i m^\prime_i(z)<\sum\limits_{i=1}^{2}\omega_i \delta_1 m^\prime_i(z)=\delta_1 m^\prime(z), \ \forall z\in[0, \tilde z).
\end{equation}
Thus, $\sum\limits_{i=1}^{2}\omega_i \delta_i m^\prime_i(z^*)<\delta_1 m^\prime(z^*)=0$, and there exists $\epsilon>0$ such that $\sum\limits_{i=1}^{2}\omega_i \delta_i m_i(z)>0, \forall \ z\in  (z^*-\epsilon,z^*)$.

We then claim that $\sum\limits_{i=1}^{2}\omega_i \delta_i m_i(z)>0$, i.e., $m^{\prime\prime}(z) > 0, \forall z\in[0,z^*)$. \\
Indeed, if there exists $z_1\in[0,z^*)$ such that $m^{\prime\prime}(z_1)=0$ and $m^{\prime\prime}(z)>0, \forall z\in(z_1,z^*)$.  Consequently, $m^{\prime}(z_1) < m^{\prime}(z^*) = 0$, and  \eqref{ineq:lprime} indicates that  $\sum\limits_{i=1}^{2}\omega_i \delta_i m^\prime_i(z_1)<\delta_1 m^\prime(z_1)<0$. On the other hand, as  $\sum\limits_{i=1}^{2}\omega_i \delta_i m_i(z_1)=0$ and $\sum\limits_{i=1}^{2}\omega_i \delta_i m_i(z)>0, \forall z\in(z_1,z^*)$, we have $\sum\limits_{i=1}^{2}\omega_i \delta_i m^\prime_i(z_1)\ge 0$, which is a contradiction.

Because $V^{\prime\prime}(x^*)=\phi^\prime(\phi^{-1}(z^*))m^{\prime}(z^*)=0$ and $V^{\prime}(x^*)=m(z^*)<m(\tilde z-)=V^{\prime}(x_1-)=-1$, where $x^* = \phi^{-1}(z^*)$, combining with \eqref{eq:Vi_ODE}, we have
\begin{equation}\label{ineq:v_x*}
\sum\limits_{i=1}^{2}\omega_i \delta_i V_i(x^*) = \frac{1}{2} \sigma^2 V^{\prime\prime}(x^*) + (\mu + M) V^{\prime}(x^*) + M < -\mu<0.
\end{equation}
However, $\sum\limits_{i=1}^{2}\omega_i \delta_i V_i^{\prime}(x)=\sum\limits_{i=1}^{2}\omega_i \delta_i m_i(\phi(x))>0, \forall x\in[0,x^*)$, resulting in $\sum\limits_{i=1}^{2}\omega_i \delta_i V_i(x^*)>\sum\limits_{i=1}^{2}\omega_i \delta_i V_i(0)=0$, which contradicts \eqref{ineq:v_x*}. Therefore, there cannot exist such $z^* \in (0, \tilde z)$, and $m^\prime(z) > 0$ for all $z \in (0, \tilde z)$, i.e., $V^{\prime\prime}(x) > 0, \forall x\in(0,x_1)$. Thus, $V^{\prime}(x)<V^{\prime}(x_1-)=-1, \forall x\in[0,x_1)$. 

Thus $V$ is decreasing and convex on $[0,+\infty)$, and \eqref{fv} and \eqref{Vi} are a solution to the extended HJB system \eqref{extendedhjb} under the discount function \eqref{eq:mix_exp_discount}.
\end{proof}

\begin{remark}\label{differ}
Suppose that there exist  $0 < x_1 < x_2$   solving \eqref{eq:x1x2}. Although the solution \eqref{fv} and \eqref{Vi} only satisfy
\begin{equation}
\begin{aligned}
f\in &C^{1,1,1}\big(\big\{(x,t,s)\bigm| (x,t)\in\mathcal{\tilde{Q}}, s\in [0,t]\big\}\big)\\&\bigcap C^{2,1,1}\big(\big\{(x,t,s)\bigm| (x,t)\in\mathcal{\tilde{Q}},x\in[0,x_1)\cup(x_1,x_2)\cup(x_2,+\infty), s\in [0,t]\big\}\big),\end{aligned}\nonumber
\end{equation}
if we use $f_{xx}(x_1+,t,s)$ ($f_{xx}(x_2-,t,s)$) when $x=x_1$ ($x=x_2$) in the Verification Theorem \ref{verificationthm} and the Mild Verification Theorem \ref{mildverificationthm}, the It\^{o}--Tanaka--Meyer formula still holds (see \cite[Theorem 3.1 and Remark 3.3]{peskir2007change}), thus the proof remains valid. 
\end{remark}

Having obtained the solution to the extended HJB system, we verify the equilibrium conditions in the following proposition.
\begin{proposition}\label{equilver}
Suppose that there exist  $0 < x_1 < x_2$ solving \eqref{eq:x1x2}. Then, the candidate regular-singular control law $(\hat\Pi,\hat\Xi)$ defined in \eqref{hatpi} and \eqref{hatxi} is an equilibrium regular-singular control law, and $V$ in \eqref{fv} and \eqref{Vi} is the corresponding equilibrium value function for the objective \eqref{objexmple} under the discount function \eqref{eq:mix_exp_discount}. In addition, $(\hat\Pi,\hat\Xi)$ is a mild equilibrium.
 \end{proposition}
\begin{proof}
 We verify that each condition of Theorem \ref{verificationthm} holds to complete the proof.
 
 Conditions (1) and (3) in Theorem \ref{verificationthm} are demonstrated in Proposition \ref{convex1} and Remark \ref{differ}. To verify Condition (2) in Theorem \ref{verificationthm}, it is sufficient to show that$(\hat\Pi,\hat\Xi)$ defined in \eqref{hatpi} and \eqref{hatxi} is an admissible regular-singular control law. For any $(x,t)\in \mathcal{\tilde{Q}}$, as the drift term in the Skorokhod reflection type SDE associated with $(\hat\Pi,\hat\Xi)$ is not (Lipschitz) continuous w.r.t. the state variable $x$, we can't use some classic existence and uniqueness results for the Skorokhod reflection problem. Fortunately, as the drift term is bounded, thanks to \cite[Theorem 3.1]{ZHANG1994135}, the Skorokhod reflection type SDE associated with $(\hat\Pi,\hat\Xi)$ has a unique strong solution. Besides, we have 
\begin{equation}
\begin{aligned}
\mathbb{E}_{x, t}\left[\!\int_t^{\tau^{\hat\Pi,\hat\Xi}} \!\!\!\beta(r\!-\!t) \left|\hat\pi_r\right| \mathrm{d} r\!\right]\!\le\! M\mathbb{E}_{x, t}\left[\!\int_t^{+\infty} \!\!\!\beta(r\!-\!t)  \mathrm{d} r\!\right]\!=\!M\sum\limits_{i=1}^{2}\left(\!\omega_i\!\int_t^{+\infty} \!\!\!\mathrm{e}^{-\delta_i (r-t)}  \mathrm{d} r\!\right)\!<\!+\infty. 
\end{aligned}\nonumber
\end{equation}
Furthermore, 
\begin{equation}
\begin{aligned}
&\hat\xi_{\tau^{\hat\Pi,\hat\Xi}\wedge(r+1)}-\hat\xi_{(\tau^{\hat\Pi,\hat\Xi}\wedge r)-}\\
=&\hat{X}_{(\tau^{\hat\Pi,\hat\Xi}\wedge r)-}-\hat{X}_{\tau^{\hat\Pi,\hat\Xi}\wedge(r+1)}+\int^{\tau^{\hat\Pi,\hat\Xi}\wedge(r+1)}_{\tau^{\hat\Pi,\hat\Xi}\wedge r}\left(\mu+\hat\pi_l\right)\mathrm{d}l+\sigma\big(B_{\tau^{\hat\Pi,\hat\Xi}\wedge (r+1)}-B_{\tau^{\hat\Pi,\hat\Xi}\wedge r}\big)\\
\le &2\max\{x,x_2\}+\mu+M +\sigma \big(B_{\tau^{\hat\Pi,\hat\Xi}\wedge (r+1)}-B_{\tau^{\hat\Pi,\hat\Xi}\wedge r}\big), \ \forall r\ge t.
\end{aligned}\nonumber
\end{equation}
 Therefore, using the Monotone Convergence Theorem and letting  $K=2\max\{x,x_2\}+\mu+M $, we have {\small 
\begin{equation}
\begin{aligned}
&\mathbb{E}_{x, t}\left[\int_t^{\tau^{\hat\Pi,\hat\Xi}} \mathrm{e}^{-\delta_i (r-t)}  \mathrm{d} \hat\xi_r\right]= \sum\limits_{n=0}^{+\infty}\mathbb{E}_{x, t}\left[\int^{\tau^{\hat\Pi,\hat\Xi}\wedge((t+n+1)-)}_{\tau^{\hat\Pi,\hat\Xi}\wedge(t+n)} \mathrm{e}^{-\delta_i (r-t)}  \mathrm{d} \hat\xi_r\right]\\
\le &\sum\limits_{n=0}^{+\infty}\mathbb{E}_{x, t}\left[\int^{\tau^{\hat\Pi,\hat\Xi}\wedge(t+n+1)}_{\tau^{\hat\Pi,\hat\Xi}\wedge(t+n)} \mathrm{e}^{-\delta_i n}  \mathrm{d} \hat\xi_r\right]
= \sum\limits_{n=0}^{+\infty}\mathrm{e}^{-\delta_i n}\mathbb{E}_{x, t}\left[\hat\xi_{\tau^{\hat\Pi,\hat\Xi}\wedge(t+n+1)}-\hat\xi_{(\tau^{\hat\Pi,\hat\Xi}\wedge(t+n))-}\right]\\  
\le &\sum\limits_{n=0}^{+\infty}\mathrm{e}^{-\delta_i n}\Big\{K+\sigma \mathbb{E}_{x, t}\left[\big(B_{\tau^{\hat\Pi,\hat\Xi}\wedge (t+n+1)}-B_{\tau^{\hat\Pi,\hat\Xi}\wedge (t+n)}\big)\right]\Big\}
=K\sum\limits_{n=0}^{+\infty}\mathrm{e}^{-\delta_i n}<+\infty.
\end{aligned}\nonumber
\end{equation} }
As such
\begin{equation}
\begin{aligned}
\mathbb{E}_{x, t}\left[\int_t^{\tau^{\hat\Pi,\hat\Xi}} \beta(r-t)  \mathrm{d} \hat\xi_r\right]=\sum\limits_{i=1}^{2}\omega_i\mathbb{E}_{x, t}\left[\int_t^{\tau^{\hat\Pi,\hat\Xi}}  \mathrm{e}^{-\delta_i (r-t)} \mathrm{d} \hat\xi_r\right]<+\infty.
\end{aligned}\nonumber
\end{equation}
Thus $(\hat\Pi,\hat\Xi)$ defined in \eqref{hatpi} and \eqref{hatxi} is an admissible regular-singular control law and Condition (2) in Theorem \ref{verificationthm} holds.
 \vskip 4pt
For any $(x,t)\in \mathcal{\tilde{Q}}$, as  $\hat{X}_r\in[0,x_2], \forall r\in[t,\tau^{\hat\Pi,\hat\Xi}]$, we know that $V_x\big(\hat{X}_r\big)$ and  $f_x\big(\hat{X}_r,r,s\big)$ are uniformly bounded for all $r\in[t,\tau^{\hat\Pi,\hat\Xi}], s\in[0,t]$, which indicates that Condition (4) in Theorem \ref{verificationthm} holds.
\vskip 4pt 
 For any $(x,t)\in \mathcal{\tilde{Q}}$, using the same argument as above, we know that $f\big(\hat{X}_r,r,s\big)$ is uniformly bounded for all $r\in[t,\tau^{\hat\Pi,\hat\Xi}], s\in[0,t]$ and $\lim\limits_{n\rightarrow +\infty}f\big(\hat{X}_{\tau_n},\tau_n,s\big)=0, \forall  s\in[0,t]$, where $\tau_n=\tau^{\hat\Pi,\hat\Xi}\wedge n$. Therefore, using the Dominated Convergence Theorem, we have $\lim\limits_{n\rightarrow+\infty}\mathbb{E}_{x,t}f\big(\hat{X}_{\tau_{n}},\tau_{n},s\big)=0$. Besides, it is easy to show that $V\big(\hat{X}_{\tau_{n}},\tau_{n}\big)=f\big(\hat{X}_{\tau_{n}},\tau_{n},\tau_{n}\big)$. Thus, Condition (5) in Theorem \ref{verificationthm} holds.
\vskip 4pt
For any $(x,t)\in\mathcal{\tilde{Q}}$, any admissible pair of perturbations $(u,\eta)\in\mathcal{\tilde D}^{x,t;\hat\Pi,\hat\Xi}$, we can choose an arbitrary $h_0>0$. Based on the properties of Brownian motion, we have $\mathbb{E}_{x,t}\big[\mathrm{e}^{k\max\limits_{r\in[t,t+h_0]}\{B_r-B_t\}}\big]<+\infty$ for any $k\in \mathbb{R}$. Then, combined with
\begin{equation}
\begin{aligned}
X^{u,\eta}_{r}=&x+\int_{t}^{r}\left(\mu+u_l\right)\mathrm{d} l+\sigma (B_r-B_t)-\int_{t}^{r}\mathrm{d} \eta_l
\le x+(\mu+M)(r-t)+\sigma (B_r-B_t)\\
\le& x+(\mu+M)h_0+\sigma \max\limits_{r\in[t,t+h_0]}\{B_r-B_t\},\ \forall r\in[t,\tau^{u,\eta}\wedge (t+h_0)],
\end{aligned}\nonumber
\end{equation}
it follows that
\begin{equation}
\begin{aligned}
\mathbb{E}_{x,t}\int_{t}^{\tau^{u,\eta}\wedge(t+h_0)}\left(V^{\prime}\left(X^{u,\eta}_{r}\right)\right)^{2}\mathrm{d}r\le h_0\mathbb{E}_{x,t}\left[K_1\mathrm{e}^{K_2\max\limits_{r\in[t,t+h_0]}\{B_r-B_t\}}\right]<+\infty,
\end{aligned}\nonumber
\end{equation}
 where $K_1$ and $ K_2>0$ are some constants. Moreover, {\small 
\begin{equation}
\begin{aligned}
&\mathbb{E}_{x,t}\bigg[\sup\limits_{\substack{r\in(t,\tau^{u,\eta}\wedge(t+h_{0}))\\ h\in(0,h_{0})}}\Big|f_s\big(X^{u,\eta}_{r},r,r\big)+\big(\beta(r-t)-1\big) u_r-f_s\big(X^{u,\eta}_{\tau^{u,\eta}\wedge((t+h)-)},\tau^{u,\eta}\wedge(t+h),r\big)\Big|\bigg]\\
\le& \mathbb{E}_{x,t}\left[K_3\mathrm{e}^{K_4\max\limits_{r\in[t,t+h_0]}\{B_r-B_t\}}\right]<+\infty,
\end{aligned}\nonumber
\end{equation} }
where $K_3, K_4>0$ are some constants. Furthermore, it is easy to show that
\begin{equation}
\begin{aligned}
&\mathbb{E}_{x,t}\big[\sup\limits_{r\in(t,\tau^{u,\eta}\wedge(t+h_{0}))}c\left|(\beta(r-t)-1)\right|\big]<+\infty.
\end{aligned}\nonumber
\end{equation}
Hence, Condition (6) in Theorem \ref{verificationthm} holds.
\vskip 4pt
In addition, we can verify that each condition of Theorem \ref{mildverificationthm} holds similarly. Thus, $(\hat\Pi,\hat\Xi)$ is a mild equilibrium.
\end{proof}

Proposition \ref{equilver} shows that, if there exist  $0 < x_1 < x_2$ that solve \eqref{eq:x1x2},  the equilibrium regular-singular control law is to apply the maximal level of effort when the surplus is below the effort threshold $x_1$ and to pay dividends whenever the surplus exceeds the dividend threshold $x_2$.
\vskip 5pt
To facilitate establishing the existence of a solution to \eqref{eq:x1x2} for the case $0 < x_1 < x_2$, we first investigate another simpler case $0=x_1<x_2$.
\vskip 2pt
For the case that $0=x_1<x_2$, the candidate regular-singular control law $(\hat\Pi,\hat\Xi)$ is 
\begin{equation}\label{hatpideg}
\hat \Pi(x,t) := 0,\ \forall (x,t)\in\mathcal{\tilde{Q}},
\end{equation}
and
\begin{equation}
\begin{aligned}
&W^{\hat{\Xi}}:=\big\{(x,t)\in\mathcal{\tilde{Q}}\bigm|0\leq x< x_2\big\},\ \ P^{\hat{\Xi}}:=\big\{(x,t)\in\mathcal{\tilde{Q}}\bigm|x\ge x_2\big\}.
\end{aligned}\label{hatxideg}
\end{equation}
For fixed $0< x_2$, $V_i(x)$ is given by the system of ODEs
\begin{equation}\label{eq:Vi_ODE2}
\begin{cases}
\frac{1}{2} \sigma^2 V^{\prime\prime}_i(x) + \mu V^\prime_i(x) - \delta_i V_i(x) = 0, \qquad & x \in [0, x_2),\\
-c- V^\prime_i(x) = 0, & x \in [x_2, +\infty),\\
V_i(0)=0.
\end{cases}
\end{equation}
Then a general solution to the ODEs \eqref{eq:Vi_ODE2} has the following  form
\begin{equation}\label{Videg}
V_i(x) = \begin{cases}
\tilde{A}_{i,3} \mathrm{e}^{\theta_{i,3} x} + \tilde{A}_{i,4} \mathrm{e}^{-\theta_{i,4}x}, & x \in [0, x_2),\\
-cx + \tilde{A}_{i,5}, & x \in [x_2, +\infty),
\end{cases}
\end{equation}
where $\tilde{A}_{i,j}$ are the constants to be determined.

To determine the values of $\tilde{A}_{i,j}$, we use the boundary condition and smooth fit principle, that is,
\begin{equation}\nonumber
V_i(0) = 0,\ V_i(x_2 -) = V_i(x_2+),\ V^\prime_i(x_2 -) = V^\prime_i(x_2+).
\end{equation}
Then $\tilde{A}_{i,j}$ satisfy the following linear equation system:
\begin{equation}\nonumber
\begin{cases}
\tilde{A}_{i,3} + \tilde{A}_{i,4} = 0, \\
\tilde{A}_{i,3} \mathrm{e}^{\theta_{i,3}x_2} + \tilde{A}_{i,4} \mathrm{e}^{-\theta_{i,4}x_2} = -cx_2 + \tilde{A}_{i,5},\\
\tilde{A}_{i,3}\theta_{i,3} \mathrm{e}^{\theta_{i,3}x_2} - \tilde{A}_{i,4} \theta_{i,4}\mathrm{e}^{-\theta_{i,4}x_2}  = -c.
\end{cases}
\end{equation}
Solving the above system yields
\begin{equation}\nonumber
\begin{aligned}
 \tilde{A}_{i,3} =- \tilde{A}_{i,4}=-\frac{c}{\theta_{i,3} \mathrm{e}^{\theta_{i,3}x_2}+\theta_{i,4}\mathrm{e}^{-\theta_{i,4}x_2}},\ 
\tilde{A}_{i,5}=cx_2+\tilde{A}_{i,3} \mathrm{e}^{\theta_{i,3}x_2} + \tilde{A}_{i,4} \mathrm{e}^{-\theta_{i,4}x_2}.
\end{aligned}
\end{equation}
It remains to determine the value of $x_2$, which is solved from $V^{\prime\prime}(x_2-) = 0$, that is,
\begin{equation}\label{eq:x1x22}
\sum\limits_{i=1}^{2}\omega_i\big(\tilde{A}_{i,3}\theta_{i,3}^2 \mathrm{e}^{\theta_{i,3}x_2} + \tilde{A}_{i,4} \theta_{i,4}^2\mathrm{e}^{-\theta_{i,4}x_2}\big) = 0,
\end{equation}
i.e.,
\begin{equation}\label{eq:x1x23}
\sum\limits_{i=1}^{2}\omega_i\left[-\frac{c}{\theta_{i,3} \mathrm{e}^{\theta_{i,3}x_2}+\theta_{i,4}\mathrm{e}^{-\theta_{i,4}x_2}}\big(\theta_{i,3}^2 \mathrm{e}^{\theta_{i,3}x_2} - \theta_{i,4}^2\mathrm{e}^{-\theta_{i,4}x_2}\big)\right] = 0.
\end{equation}
Let $F(x):=\sum\limits_{i=1}^{2}\omega_i\left[-\frac{c}{\theta_{i,3} \mathrm{e}^{\theta_{i,3}x}+\theta_{i,4}\mathrm{e}^{-\theta_{i,4}x}}\big(\theta_{i,3}^2 \mathrm{e}^{\theta_{i,3}x} - \theta_{i,4}^2\mathrm{e}^{-\theta_{i,4}x}\big)\right]$. It is straightforward to verify that $F(0)=\frac{2c\mu}{\sigma^2}>0$,  $F(+\infty)=-c\sum\limits_{i=1}^{2}\omega_i\theta_{i,3}<0$, and $F^{\prime}(x)<0, \forall x\in [0,+\infty)$. Thus, \eqref{eq:x1x23} (i.e., \eqref{eq:x1x22}) has a unique solution $x_2>0$.

\begin{proposition}\label{cdeg}
Suppose $0<c\le \big[\sum\limits_{i=1}^{2}\omega_i\frac{\theta_{i,3}+\theta_{i,4}}{\theta_{i,3} \mathrm{e}^{\theta_{i,3}x_2}+\theta_{i,4}\mathrm{e}^{-\theta_{i,4}x_2}}\big]^{-1}$, where $x_2$ is the unique positive solution to \eqref{eq:x1x23}. Then, \eqref{fv} and \eqref{Videg} are a solution to the extended HJB system \eqref{extendedhjb} under the discount function \eqref{eq:mix_exp_discount}. In addition, the solution $V$ is decreasing and convex on $[0,+\infty)$. Furthermore, the candidate regular-singular control law $(\hat\Pi,\hat\Xi)$ defined in \eqref{hatpideg} and \eqref{hatxideg} is an equilibrium regular-singular control law and $V$ in \eqref{fv} and \eqref{Videg} is the corresponding equilibrium value function for the objective \eqref{objexmple} under the discount function \eqref{eq:mix_exp_discount}. Besides, $(\hat\Pi,\hat\Xi)$ is a mild equilibrium.
\end{proposition}
\begin{proof}
Using the same argument as Proposition \ref{convex1}, we know that $V^{\prime\prime}>0, \forall x\in[0,x_2)$. Therefore, \eqref{fv} and \eqref{Videg} are a solution to the extended HJB system \eqref{extendedhjb} under the discount function \eqref{eq:mix_exp_discount} if $V^{\prime}(0)=-\sum\limits_{i=1}^{2}\omega_ic\frac{\theta_{i,3}+\theta_{i,4}}{\theta_{i,3} \mathrm{e}^{\theta_{i,3}x_2}+\theta_{i,4}\mathrm{e}^{-\theta_{i,4}x_2}}\ge -1$, i.e., $c\le \big[\sum\limits_{i=1}^{2}\omega_i\frac{\theta_{i,3}+\theta_{i,4}}{\theta_{i,3} \mathrm{e}^{\theta_{i,3}x_2}+\theta_{i,4}\mathrm{e}^{-\theta_{i,4}x_2}}\big]^{-1}$	, where $x_2$ is the unique positive solution to \eqref{eq:x1x23}. The remaining proof is similar to Remark \ref{differ} and Proposition \ref{equilver}.
\end{proof}

When $0<c\le \big[\sum\limits_{i=1}^{2}\omega_i\frac{\theta_{i,3}+\theta_{i,4}}{\theta_{i,3} \mathrm{e}^{\theta_{i,3}x_2}+\theta_{i,4}\mathrm{e}^{-\theta_{i,4}x_2}}\big]^{-1}$ with $x_2$ being the unique positive solution to \eqref{eq:x1x23}, Proposition \ref{cdeg} indicates that the equilibrium regular-singular control law is to pay dividends whenever the surplus exceeds the dividend threshold $x_2$, while applying no effort even when the firm approaches ruin. Economically, this occurs because $c$ is relatively small and the cost for dividend is relatively large.

\begin{remark}
If we set $x_1=x_2=0$, then $V_i(x)=-cx, \forall x\ge0$. It is easy to show that \eqref{fv} in this case is not a solution to the extended HJB system \eqref{extendedhjb} under the discount function \eqref{eq:mix_exp_discount}. 
\end{remark}

Based on the above analysis for the case $0=x_1<x_2$, we have the following proposition that provides a sufficient condition for the existence of a solution $0<x_1<x_2$ to \eqref{eq:x1x2}.

\begin{proposition}\label{suff}
Suppose that $\frac{\theta_{1,4}-\theta_{1,3}}{\theta_{2,1}}<\frac{\omega_2}{\omega_1}$, $\frac{\theta_{2,4}-\theta_{2,3}}{\theta_{1,1}}<\frac{\omega_1}{\omega_2}$, and\\ $\max\Big\{\frac{2M}{\sigma^2(\theta_{1,2}+\theta_{1,3}-\theta_{1,4})},\frac{2M}{\sigma^2(\theta_{2,2}+\theta_{2,3}-\theta_{2,4})},\big[\sum\limits_{i=1}^{2}\omega_i\frac{\theta_{i,3}+\theta_{i,4}}{\theta_{i,3} \mathrm{e}^{\theta_{i,3}x_{2,0}}+\theta_{i,4}\mathrm{e}^{-\theta_{i,4}x_{2,0}}}\big]^{-1}\Big\}<c<1$, where $x_{2,0}$ is the unique positive solution to \eqref{eq:x1x23}. Then, there exists a solution $0<x_1<x_2$ to \eqref{eq:x1x2}.	
\end{proposition}
\begin{proof}
Using the smooth fit principle and combining with the ODE system \eqref{eq:Vi_ODE}, we find that
\begin{equation}\label{eq:smooth1}
V^{\prime}_i(x_2-) = A_{i,3}\theta_{i,3} \mathrm{e}^{\theta_{i,3}x_2} - A_{i,4} \theta_{i,4}\mathrm{e}^{-\theta_{i,4}x_2}  = -c,
\end{equation}
and \begin{equation}\label{eq:smooth2}
V^{\prime\prime}(x_2-) = \sum\limits_{i=1}^{2}\omega_i\big(A_{i,3}\theta_{i,3}^2 \mathrm{e}^{\theta_{i,3}x_2} + A_{i,4} \theta_{i,4}^2\mathrm{e}^{-\theta_{i,4}x_2}\big) = 0.
\end{equation}
Denote
\begin{equation}\begin{aligned}\nonumber
F_i(x)&:=\frac{M}{\delta_i}-\frac{M}{\delta_i}\frac{(\theta_{i,1}+\theta_{i,2})\mathrm{e}^{(\theta_{i,1}-\theta_{i,2})x}}{\theta_{i,1}\mathrm{e}^{\theta_{i,1}x}+\theta_{i,2}\mathrm{e}^{-\theta_{i,2}x}},\\
G_i(x)&:=\frac{\mathrm{e}^{\theta_{i,1}x}-\mathrm{e}^{-\theta_{i,2}x}}{\theta_{i,1}\mathrm{e}^{\theta_{i,1}x}+\theta_{i,2}\mathrm{e}^{-\theta_{i,2}x}},\\
H_i(x)&:=\mathrm{e}^{-\theta_{i,3}x}-\mathrm{e}^{\theta_{i,4}x},\\
L_i(x)&:=\theta_{i,4}\mathrm{e}^{-\theta_{i,3}x}+\theta_{i,3}\mathrm{e}^{\theta_{i,4}x}.
\end{aligned}\end{equation}
Because $V_i(0)=0$, $V_i(x_1+)=V_i(x_1-)$, and $V^{\prime}_i(x_1+)=V^{\prime}_i(x_1-)$, i.e.,
\begin{equation}\nonumber
\begin{cases}
\frac{M}{\delta_i} + A_{i,1} + A_{i,2} = 0, \\
A_{i,3} \mathrm{e}^{\theta_{i,3}x_1} + A_{i,4} \mathrm{e}^{-\theta_{i,4}x_1}=\frac{M}{\delta_i} + A_{i,1} \mathrm{e}^{\theta_{i,1}x_1} + A_{i,2} \mathrm{e}^{-\theta_{i,2}x_1}, \\
A_{i,3} \theta_{i,3} \mathrm{e}^{\theta_{i,3}x_1} - A_{i,4} \theta_{i,4} \mathrm{e}^{-\theta_{i,4}x_1}=A_{i,1} \theta_{i,1} \mathrm{e}^{\theta_{i,1}x_1} - A_{i,2} \theta_{i,2} \mathrm{e}^{-\theta_{i,2}x_1}, 
\end{cases}
\end{equation}
and \eqref{eq:smooth1}, we have
\begin{equation}\label{eq:smooth3}
A_{i,3}\mathrm{e}^{\theta_{i,3}x_2}=\frac{\theta_{i,4}F_i(x_1)-c\theta_{i,4}G_i(x_1)\mathrm{e}^{\theta_{i,4}(x_2-x_1)}-c\mathrm{e}^{\theta_{i,4}(x_2-x_1)}}{L_i(x_2-x_1)-\theta_{i,3}\theta_{i,4}G_i(x_1)H_i(x_2-x_1)}.
\end{equation}
Moreover, substituting \eqref{eq:smooth1} and \eqref{eq:smooth3} into \eqref{eq:smooth2}, we have 
{\small\begin{equation}\nonumber\begin{aligned}
\sum\limits_{i=1}^{2}\omega_i\frac{\theta_{i,3}\theta_{i,4}(\theta_{i,3}\!+\!\theta_{i,4})F_i(x_1)\!-\!c\theta_{i,3}\theta_{i,4}G_i(x_1)L_i(x_2\!-\!x_1)\!-\!c\big(\theta_{i,3}^2\mathrm{e}^{\theta_{i,4}(x_2-x_1)}\!-\!\theta_{i,4}^2\mathrm{e}^{-\theta_{i,3}(x_2-x_1)}\big)}{L_i(x_2\!-\!x_1)\!-\!\theta_{i,3}\theta_{i,4}G_i(x_1)H_i(x_2\!-\!x_1)}=0.
\end{aligned}\end{equation}}
On the other hand, using the smooth fit principle
\begin{equation}\label{eq:smooth4}
V^{\prime}(x_1+)=\sum\limits_{i=1}^{2}\omega_i \big(A_{i,3} \theta_{i,3} \mathrm{e}^{\theta_{i,3}x_1} - A_{i,4} \theta_{i,4} \mathrm{e}^{-\theta_{i,4}x_1}\big)=-1,
\end{equation}
substituting \eqref{eq:smooth1} and \eqref{eq:smooth3} into \eqref{eq:smooth4}, we have 
\begin{equation}\begin{aligned}\nonumber
\sum\limits_{i=1}^{2}\omega_i\frac{\theta_{i,3}\theta_{i,4}F_i(x_1)H_i(x_2-x_1)-c(\theta_{i,3}+\theta_{i,4})\mathrm{e}^{(\theta_{i,4}-\theta_{i,3})(x_2-x_1)}}{L_i(x_2-x_1)-\theta_{i,3}\theta_{i,4}G_i(x_1)H_i(x_2-x_1)}=-1.
\end{aligned}\end{equation}
Fix $x_1\ge0$, let $R(z):=\sum\limits_{i=1}^{2}\omega_i\frac{\theta_{i,3}\theta_{i,4}(\theta_{i,3}+\theta_{i,4})F_i(x_1)-c\theta_{i,3}\theta_{i,4}G_i(x_1)L_i(z)-c\left(\theta_{i,3}^2\mathrm{e}^{\theta_{i,4}z}-\theta_{i,4}^2\mathrm{e}^{-\theta_{i,3}z}\right)}{L_i(z)-\theta_{i,3}\theta_{i,4}G_i(x_1)H_i(z)}$, we have
{\small\begin{equation}\begin{aligned}\nonumber
R^{\prime}(z)=&\sum\limits_{i=1}^{2}\omega_i\frac{-c\theta_{i,3}\theta_{i,4}(\theta_{i,3}+\theta_{i,4})^2\mathrm{e}^{(\theta_{i,4}-\theta_{i,3})z}\big(-\theta_{i,3}\theta_{i,4}G_i^2(x_1)+(\theta_{i,4}-\theta_{i,3})G_i(x_1)+1\big)}{\big(L_i(z)-\theta_{i,3}\theta_{i,4}G_i(x_1)H_i(z)\big)^2}\\
&-\sum\limits_{i=1}^{2}\omega_i\frac{\theta_{i,3}\theta_{i,4}(\theta_{i,3}+\theta_{i,4})F_i(x_1)\big(L_i^{\prime}(z)-\theta_{i,3}\theta_{i,4}G_i(x_1)H_i^{\prime}(z)\big)}{\big(L_i(z)-\theta_{i,3}\theta_{i,4}G_i(x_1)H_i(z)\big)^2}.
\end{aligned}\end{equation}}
It is easy to show that $F_i(x_1)\ge 0, 0\le G_i(x_1)< \frac{1}{\theta_{i,1}}$ for any $x_1\ge 0$, and that $L_i^{\prime}(z)\ge 0, H_i^{\prime}(z)=\!-\theta_{i,3}\mathrm{e}^{-\theta_{i,3}z}\!-\!\theta_{i,4}\mathrm{e}^{\theta_{i,4}z}=-\mathrm{e}^{(\theta_{i,4}-\theta_{i,3})z}\big(\theta_{i,3}\mathrm{e}^{-\theta_{i,4}z}\!+\!\theta_{i,4}\mathrm{e}^{\theta_{i,3}z}\big)\le-\mathrm{e}^{(\theta_{i,4}-\theta_{i,3})z}(\theta_{i,3}\!+\!\theta_{i,4})$ for any $z\ge0$. Therefore, we have
{\small\begin{equation}\begin{aligned}\nonumber
R^{\prime}(z)\!\le\!\sum\limits_{i=1}^{2}\omega_i\frac{\theta_{i,3}\theta_{i,4}(\theta_{i,3}\!+\!\theta_{i,4})^2\mathrm{e}^{(\theta_{i,4}-\theta_{i,3})z}\big(c\theta_{i,3}\theta_{i,4}G_i^2(x_1)\!-\!c(\theta_{i,4}\!-\!\theta_{i,3})G_i(x_1)\!-\!c\!-\!\theta_{i,3}\theta_{i,4}F_i(x_1)G_i(x_1)\big)}{\big(L_i(z)\!-\!\theta_{i,3}\theta_{i,4}G_i(x_1)H_i(z)\big)^2}.
\end{aligned}\end{equation}}
It is easy to show that $R(0)=S(x_1):=\sum\limits_{i=1}^{2}\omega_i\left(\theta_{i,3}\theta_{i,4}F_i(x_1)-c\theta_{i,3}\theta_{i,4}G_i(x_1)-c\theta_{i,3}+c\theta_{i,4}\right)$ and $S(0)=\sum\limits_{i=1}^{2}\omega_i(-c\theta_{i,3}+c\theta_{i,4})=\frac{2c\mu}{\sigma^2}>0$. Because $c\!>\!\max\!\left\{\!\frac{2M}{\sigma^2(\theta_{1,2}+\theta_{1,3}-\theta_{1,4})}\!,\!\frac{2M}{\sigma^2(\theta_{2,2}+\theta_{2,3}-\theta_{2,4})}\!\right\}$, we know that $S(+\infty)\!=\!\sum\limits_{i=1}^{2}\omega_i\big(\theta_{i,3}\theta_{i,4}\frac{M}{\delta_i}\!-\!c\theta_{i,3}\theta_{i,4}\frac{1}{\theta_{i,1}}\!-\!c\theta_{i,3}\!+\!c\theta_{i,4}\big)\!=\!\sum\limits_{i=1}^{2}\omega_i\big(\frac{2M}{\sigma^2}\!-\!c\theta_{i,2}\!-\!c\theta_{i,3}\!+\!c\theta_{i,4}\big)\!<\!0$. Hence, there exists $x_{1,0}>0$ such that $S(x_1)>0, \forall x_1\in[0,x_{1,0})$ and $S(x_{1,0})=0$. Hereafter, we will confine the discussion of this problem to $x_1\in[0,x_{1,0}]$. Let 
\begin{equation}\begin{aligned}\nonumber
S_i(x_1):=\theta_{i,3}\theta_{i,4}F_i(x_1)-c\theta_{i,3}\theta_{i,4}G_i(x_1)-c\theta_{i,3}+c\theta_{i,4},
\end{aligned}\end{equation}
we know that $S_i^{\prime}(x_1)=\frac{\theta_{i,3}\theta_{i,4}(\theta_{i,1}+\theta_{i,2})\mathrm{e}^{(\theta_{i,1}-\theta_{i,2})x_1}}{\big(\theta_{i,1}\mathrm{e}^{\theta_{i,1}x
_1}+\theta_{i,2}\mathrm{e}^{-\theta_{i,2}x_1}\big)^2}\left[\frac{M}{\delta_i}\theta_{i,1}\theta_{i,2}\big(\mathrm{e}^{\theta_{i,1}x_1}-\mathrm{e}^{-\theta_{i,2}x_1}\big)-c(\theta_{i,1}+\theta_{i,2})\right]$. Therefore, $S_i^{\prime}(x_1)<0$ when $\frac{M}{\delta_i}\theta_{i,1}\theta_{i,2}\big(\mathrm{e}^{\theta_{i,1}x_1}-\mathrm{e}^{-\theta_{i,2}x_1}\big)<c(\theta_{i,1}+\theta_{i,2})$, and $S_i^{\prime}(x_1)>0$ when $\frac{M}{\delta_i}\theta_{i,1}\theta_{i,2}\big(\mathrm{e}^{\theta_{i,1}x_1}-\mathrm{e}^{-\theta_{i,2}x_1}\big)>c(\theta_{i,1}+\theta_{i,2})$. Because $S_i(0)=\frac{2c\mu}{\sigma^2}>0>\theta_{i,3}\theta_{i,4}\frac{M}{\delta_i}-c\theta_{i,3}\theta_{i,4}\frac{1}{\theta_{i,1}}-c\theta_{i,3}+c\theta_{i,4}=S_i(+\infty)$, we know that $S_i(x_1)\le S_i(0)=c(\theta_{i,4}-\theta_{i,3}), \forall x_1\ge0$.

For any $x_1\in[0,x_{1,0}]$, if $S_i(x_1)\le 0$, as  $S(x_1)\ge0$ and $G_i(x_1)<\frac{1}{\theta_{i,1}}$, then $S_i(x_1)G_i(x_1)+c\ge S_i(x_1)\frac{1}{\theta_{i,1}}+c\ge -\frac{\omega_j}{\omega_i\theta_{i,1}}S_j(x_1)+c\ge -\frac{\omega_j}{\omega_i\theta_{i,1}}S_j(0)+c=c\big(-\frac{\omega_j(\theta_{j,4}-\theta_{j,3})}{\omega_i\theta_{i,1}}+1\big)>0$, where $j\neq i$. Therefore, for any $i=1,2$, $W_i(x_1):=-S_i(x_1)G_i(x_1)-c<0$. Thus, 
 \begin{equation}\begin{aligned}\nonumber
R^{\prime}(z)\le\sum\limits_{i=1}^{2}\omega_i\frac{\theta_{i,3}\theta_{i,4}(\theta_{i,3}+\theta_{i,4})^2\mathrm{e}^{(\theta_{i,4}-\theta_{i,3})z}W_i(x_1)}{\left(L_i(z)-\theta_{i,3}\theta_{i,4}G_i(x_1)H_i(z)\right)^2}<0.
\end{aligned}\end{equation}
Therefore, for any $x_1\in[0,x_{1,0}]$, as  $R(0)\ge0$, $R(+\infty)=-c\sum\limits_{i=1}^{2}\omega_i\theta_{i,3}<0$ and $R^{\prime}(z)<0, \forall z\ge0$, we know that there exists a unique $z(x_1)$, such that $R(z(x_1))=0$, $z(x_1)$ is continuous w.r.t $x_1$, $z(0)=x_{2,0}$ and $z(x_{1,0})=0$. Let \[Q(x_1):=\sum\limits_{i=1}^{2}\omega_i\frac{\theta_{i,3}\theta_{i,4}F_i(x_1)H_i(z(x_1))-c(\theta_{i,3}+\theta_{i,4})\mathrm{e}^{(\theta_{i,4}-\theta_{i,3})z(x_1)}}{L_i(z(x_1))-\theta_{i,3}\theta_{i,4}G_i(x_1)H_i(z(x_1))}\] be a continuous function w.r.t $x_1$. Then,  as  $c>\big[\sum\limits_{i=1}^{2}\omega_i\frac{\theta_{i,3}+\theta_{i,4}}{\theta_{i,3} \mathrm{e}^{\theta_{i,3}x_{2,0}}+\theta_{i,4}\mathrm{e}^{-\theta_{i,4}x_{2,0}}}\big]^{-1}$, using the proof of Proposition \ref{cdeg}, we have \[Q(0)=\sum\limits_{i=1}^{2}\omega_i\frac{-c(\theta_{i,3}+\theta_{i,4})\mathrm{e}^{(\theta_{i,4}-\theta_{i,3})z(0)}}{L_i(z(0))}=\sum\limits_{i=1}^{2}\omega_i\frac{-c(\theta_{i,3}+\theta_{i,4})\mathrm{e}^{(\theta_{i,4}-\theta_{i,3})x_{2,0}}}{L_i(x_{2,0})}<-1.\] Besides, we have $Q(x_{1,0})=\sum\limits_{i=1}^{2}\omega_i\frac{\theta_{i,3}\theta_{i,4}F_i(x_{1,0})H_i(z(x_{1,0}))-c(\theta_{i,3}+\theta_{i,4})\mathrm{e}^{(\theta_{i,4}-\theta_{i,3})z(x_{1,0})}}{L_i(z(x_{1,0}))-\theta_{i,3}\theta_{i,4}G_i(x_1)H_i(z(x_{1,0}))}=-c>-1$. Therefore, there exists $x_1\in(0,x_{1,0})$ such that $Q(x_1)=-1$. Hence, there exists a solution $0<x_1<x_2$ to  \eqref{eq:x1x2}.
\end{proof}

\begin{remark}
The conditions in Proposition \ref{suff} are easy to meet. As long as $\mu$ is sufficiently small or $\delta_2\sigma^2$ is sufficiently large, $\frac{\theta_{1,4}-\theta_{1,3}}{\theta_{2,1}}=\frac{2\mu}{-(\mu+M)+\sqrt{(\mu+M)^2+2\delta_2\sigma^2}}<\frac{\omega_2}{\omega_1}$ can be achieved. Similarly, as long as $\mu$ is sufficiently small or $\delta_1\sigma^2$ is sufficiently large, the inequality $\frac{\theta_{2,4}-\theta_{2,3}}{\theta_{1,1}}=\frac{2\mu}{-(\mu+M)+\sqrt{(\mu+M)^2+2\delta_1\sigma^2}}<\frac{\omega_1}{\omega_2}$ will hold.

Besides, it is easy to show $\frac{2M}{\sigma^2(\theta_{i,2}+\theta_{i,3}-\theta_{i,4})}=\frac{2M}{-\mu+M+\sqrt{(\mu+M)^2+2\delta_i\sigma^2}}<1$. Furthermore, let $x_{2,0}$ be the unique positive solution to \eqref{eq:x1x23}, we have
\begin{equation}\nonumber
\begin{aligned}
\ln\left( \sum\limits_{i=1}^{2}\omega_i\frac{\theta_{i,3}+\theta_{i,4}}{\theta_{i,3} \mathrm{e}^{\theta_{i,3}x_{2,0}}+\theta_{i,4}\mathrm{e}^{-\theta_{i,4}x_{2,0}}}\right)&\ge\sum\limits_{i=1}^{2}\omega_i\ln\left(\frac{\theta_{i,3}+\theta_{i,4}}{\theta_{i,3} \mathrm{e}^{\theta_{i,3}x_{2,0}}+\theta_{i,4}\mathrm{e}^{-\theta_{i,4}x_{2,0}}}\right)\\&>\sum\limits_{i=1}^{2}\omega_i\ln\left(\frac{\theta_{i,3}+\theta_{i,4}}{\theta_{i,3}+\theta_{i,4}}\right)=0.\end{aligned}
\end{equation}

The last inequality holds because $x_{2,0}$ is the unique positive solution to \eqref{eq:x1x23}. As $F(x)>0, \forall x \in[0,x_{2,0})$ and $\frac{\partial}{\partial x}\sum\limits_{i=1}^{2}\omega_i\ln\left(\frac{\theta_{i,3}+\theta_{i,4}}{\theta_{i,3} \mathrm{e}^{\theta_{i,3}x}+\theta_{i,4}\mathrm{e}^{-\theta_{i,4}x}}\right)=\frac{F(x)}{c}$,  we know $\sum\limits_{i=1}^{2}\omega_i\frac{\theta_{i,3}+\theta_{i,4}}{\theta_{i,3} \mathrm{e}^{\theta_{i,3}x_{2,0}}+\theta_{i,4}\mathrm{e}^{-\theta_{i,4}x_{2,0}}}>1$. Hence, there exists a value $c\in(0,1)$ such that 
\begin{equation}\nonumber
\begin{aligned}c> \max\Big\{\frac{2M}{\sigma^2(\theta_{1,2}+\theta_{1,3}-\theta_{1,4})},\frac{2M}{\sigma^2(\theta_{2,2}+\theta_{2,3}-\theta_{2,4})},\big[\sum\limits_{i=1}^{2}\omega_i\frac{\theta_{i,3}+\theta_{i,4}}{\theta_{i,3} \mathrm{e}^{\theta_{i,3}x_{2,0}}+\theta_{i,4}\mathrm{e}^{-\theta_{i,4}x_{2,0}}}\big]^{-1}\Big\}.\end{aligned}
\end{equation}

\end{remark}

\begin{remark}
Under mild conditions, the equilibrium policy prescribes exerting the maximal level of effort when the surplus falls below the effort threshold and paying dividends once the surplus exceeds the dividend threshold when the cost for dividend is relatively small; otherwise, when the cost for dividend is relatively large, the equilibrium policy is to pay dividends once the surplus exceeds the dividend threshold and to apply no effort even if the firm approaches ruin.
\end{remark}

\begin{remark}
Let 
\begin{equation}\nonumber
\hat{\Pi}_a(x,t) := \begin{cases}
M, & 0 \le x < x_1, \\
0, & x_1 \le x \le x_2, \\
a, & x > x_2,
\end{cases} \quad \forall (x,t) \in \mathcal{\tilde{Q}},
\end{equation}
and
\begin{equation}
\begin{aligned}
W^{\hat{\Xi}}:=\big\{(x,t)\in\mathcal{\tilde{Q}}\bigm|0\leq x< x_2\big\},\ \ P^{\hat{\Xi}}:=\big\{(x,t)\in\mathcal{\tilde{Q}}\bigm|x\ge x_2\big\}.
\end{aligned}\nonumber
\end{equation}
Suppose that there exist $0<x_1<x_2$ such that $V$ in \eqref{fv} and \eqref{Vi} satisfies $V^\prime(x_1) = -1$
 ($\Leftrightarrow V^{\prime\prime}(x_1-)=V^{\prime\prime}(x_1+)$) and $V^{\prime\prime}(x_2-) = b\in[0,\frac{2a(1-c)}{\sigma^2}]$.
 Although we only have
 \begin{equation}
\begin{aligned}V\in C^{1,1}\big(\mathcal{\tilde{Q}}\big)\bigcap C^{2,1}\big(\big\{(x,t)\in\mathcal{\tilde{Q}}\bigm| x\in[0,x_2)\cup(x_2,+\infty)\big\}\big),\end{aligned}\nonumber
\end{equation}
the Mild Verification Theorem \ref{mildverificationthm} still holds if we use $V^{\prime\prime}(x_2-)$ when $x=x_2$. Analogously to the above discussion, we can show that, for any $a\in[0,M]$, $(\hat\Pi_a,\hat\Xi)$ is a mild equilibrium regular-singular control law and $V$ in \eqref{fv} and \eqref{Vi} is the corresponding mild equilibrium value function for the objective \eqref{objexmple} under the discount function \eqref{eq:mix_exp_discount}. All the above-mentioned discussions can be directly transferred here. For brevity, the details are omitted here.
\end{remark}

\begin{remark}
Effort outside the threshold $x_2$ serves as a feedback commitment following a deviation, not as the effort expenditure actually realized along the equilibrium path.
\end{remark}

\subsection{Pseudo-exponential discount function}
In this subsection, we consider a pseudo-exponential discount function defined as
\begin{equation}\label{eq:pseudo_exp_discount}
\beta(t) := (1+\lambda t) \mathrm{e}^{-\delta t},\  \forall t\ge 0,
\end{equation}
where $0<\lambda< \delta$ are given parameters. Pseudo-exponential discount function \eqref{eq:pseudo_exp_discount} implies a rising discount rate over time, thus the firm overweights current costs and heavily penalizes future expenses. For a detailed explanation of this discount function, we refer the reader to \cite{ekeland2008invest}. 

Because the problem is time-homogeneous, we conjecture that $V$ is independent of time $t$. Inspired by \cite{ZHAO20141}, we consider the following ansatz:
\begin{equation}\begin{aligned}\label{fv2}
	f(x,t,s)&= \mathrm{e}^{-\delta (t-s)}\left[\lambda(t-s)V_3(x)+V_4(x)\right],\  \forall (x,t)\in\mathcal{\tilde{Q}}, s\in[0,t],\\
	V(x)&=V_4(x),\ \forall x\ge 0.\end{aligned}\end{equation}
Furthermore, we conjecture that the candidate regular-singular control law is threshold-type, specified by two thresholds $0 < x_1 < x_2$:
\begin{equation}\label{hatpi2}
\hat \Pi(x,t) := \begin{cases}
 M,&0 \le x < x_1, \\
0,&x \geq x_1, 
\end{cases}\quad\forall (x,t)\in\mathcal{\tilde{Q}},
\end{equation}
and
\begin{equation}
\begin{aligned}
W^{\hat{\Xi}}:=\big\{(x,t)\in\mathcal{\tilde{Q}}\bigm| 0\leq x< x_2\big\},\ \ P^{\hat{\Xi}}:=\big\{(x,t)\in\mathcal{\tilde{Q}}\bigm| x\ge x_2\big\}.
\end{aligned}\label{hatxi2}
\end{equation}
For fixed $0 < x_1 < x_2$, the functions $V_3(x)$ and $V_4(x)$ are governed by the systems of ODEs
\begin{equation}\label{eq:V3_ODE}
\begin{cases}
\frac{1}{2} \sigma^2 V^{\prime\prime}_3(x) + (\mu + M) V^\prime_3(x) - \delta V_3(x) +M = 0, & x \in [0, x_1),\\
\frac{1}{2} \sigma^2 V^{\prime\prime}_3(x) + \mu V^\prime_3(x) - \delta V_3(x) = 0, \qquad & x \in [x_1, x_2),\\
-c- V^\prime_3(x) = 0, & x \in [x_2, +\infty),\\
V_3(0)=0,
\end{cases}
\end{equation}
and
\begin{equation}\label{eq:V4_ODE}
\begin{cases}
\frac{1}{2} \sigma^2 V^{\prime\prime}_4(x) + (\mu + M) V^\prime_4(x) - \delta V_4(x) +\lambda V_3(x)+M = 0, & x \in [0, x_1),\\
\frac{1}{2} \sigma^2 V^{\prime\prime}_4(x) + \mu V^\prime_4(x) - \delta V_4(x)+\lambda V_3(x) = 0, \qquad & x \in [x_1, x_2),\\
-c- V^\prime_4(x) = 0, & x \in [x_2, +\infty),\\
V_4(0)=0.
\end{cases}
\end{equation}
Denote by
\begin{equation}\begin{aligned}\nonumber
&\theta_{1}:= \frac{-(\mu+M) + \sqrt{(\mu+M)^2 + 2 \delta \sigma^2}}{\sigma^2} ,\quad \theta_{2}:= \frac{\mu+M + \sqrt{(\mu+M)^2 + 2 \delta \sigma^2}}{\sigma^2},\\
&\theta_{3}:= \frac{-\mu + \sqrt{\mu^2 + 2 \delta \sigma^2}}{\sigma^2} ,\quad\theta_{4}:= \frac{\mu + \sqrt{\mu^2 + 2 \delta \sigma^2}}{\sigma^2}.\end{aligned}\end{equation}
Thus, a general solution to the ODEs \eqref{eq:V3_ODE} has the form
\begin{equation}\label{V3}
V_3(x) = \begin{cases}
\frac{M}{\delta} + A_{1} \mathrm{e}^{\theta_{1}x} + A_{2} \mathrm{e}^{-\theta_{2}x}, & x \in [0, x_1),\\
A_{3} \mathrm{e}^{\theta_{3} x} + A_{4} \mathrm{e}^{-\theta_{4}x}, & x \in [x_1, x_2),\\
-cx + A_{5},  & x \in [x_2, +\infty),
\end{cases}
\end{equation}
where $A_{i}$ are the constants to be determined.
\vskip 4pt
To determine the values of $A_{i}$, we use the boundary condition and smooth fit principle,i.e., 
\begin{equation}\nonumber
V_3(0) \!=\! 0,\ V_3(x_1 -) \!=\! V_3(x_1+), \ V^\prime_3(x_1 -) \!=\! V^\prime_3(x_1+),\ V_3(x_2 -) \!=\! V_3(x_2+),\ V^\prime_3(x_2 -) \!=\! V^\prime_3(x_2+).
\end{equation}
Then $A_{i}$ satisfy the following linear equation system:
\begin{equation}\label{Alinear}
\begin{cases}
\frac{M}{\delta} + A_{1} + A_{2} = 0, \\
\frac{M}{\delta} + A_{1} \mathrm{e}^{\theta_{1}x_1} + A_{2} \mathrm{e}^{-\theta_{2}x_1} = A_{3} \mathrm{e}^{\theta_{3}x_1} + A_{4} \mathrm{e}^{-\theta_{4}x_1}, \\
A_{1} \theta_{1} \mathrm{e}^{\theta_{1}x_1} - A_{2} \theta_{2} \mathrm{e}^{-\theta_{2}x_1} = A_{3} \theta_{3} \mathrm{e}^{\theta_{3}x_1} - A_{4} \theta_{4} \mathrm{e}^{-\theta_{4}x_1}, \\
A_{3} \mathrm{e}^{\theta_{3}x_2} + A_{4} \mathrm{e}^{-\theta_{4}x_2} = -cx_2 + A_{5},\\
A_{3}\theta_{3} \mathrm{e}^{\theta_{3}x_2} - A_{4} \theta_{4}\mathrm{e}^{-\theta_{4}x_2}  = -c.
\end{cases}
\end{equation}
After obtaining $V_3$, a general solution to the ODEs \eqref{eq:V4_ODE} has the form: {\small 
\begin{equation}\label{V4}
V_4(x) \!=\begin{cases}
\!\frac{M}{\delta}\big(\frac{\lambda}{\delta}\!+\!1\big)\!-\!\frac{\lambda A_1}{\mu\!+\!M\!+\!\sigma^2\theta_1}x\mathrm{e}^{\theta_{1}x}\!+\!\frac{\lambda A_2}{-(\mu\!+\!M)\!+\!\sigma^2\theta_2}x\mathrm{e}^{-\theta_{2}x}\!+\! B_{1} \mathrm{e}^{\theta_{1}x} \!+\! B_{2} \mathrm{e}^{-\theta_{2}x},&x \!\in \![0, x_1),\\
-\frac{\lambda A_3}{\mu\!+\!\sigma^2\theta_3}x\mathrm{e}^{\theta_{3}x}\!+\!\frac{\lambda A_4}{-\!\mu\!+\!\sigma^2\theta_4}x\mathrm{e}^{-\theta_{4}x}\!+\!B_{3} \mathrm{e}^{\theta_{3} x} \!+\! B_{4} \mathrm{e}^{-\theta_{4}x},  &x \!\in \![x_1, x_2),\\
-cx \!+\!B_{5},  & x \!\in \![x_2, +\infty),
\end{cases}
\end{equation}
}
where $B_{i}$ are the constants to be determined.

Then we use the boundary condition and smooth fit principle to determine the values of $B_{i}$, i.e.,
\begin{equation}\nonumber
V_4(0) \!=\! 0,\ V_4(x_1 -) \!=\! V_4(x_1+), \ V^\prime_4(x_1 -) \!=\! V^\prime_4(x_1+),\ V_4(x_2 -) \!=\! V_4(x_2+),\ V^\prime_4(x_2 -) \!=\! V^\prime_4(x_2+).
\end{equation}
From which we know that the constants $B_{i}$ satisfy a linear equation system whose coefficient matrix is similar to that of \eqref{Alinear}, and is of full rank. Therefore, we can obtain the unique solution $B_{i}$.
\vskip 4pt 
It remains to determine the values of $x_1$ and $x_2$, which are solved from $V^\prime_4(x_1) = -1$ (or equivalently, $ V^{\prime\prime}_4(x_1-)=V^{\prime\prime}_4(x_1+)$) and $V^{\prime\prime}_4(x_2-)=0=V^{\prime\prime}_4(x_2+)$, i.e., {\small 
\begin{equation}\label{eq2:x1x2}
\begin{cases}
\frac{\lambda A_4}{-\mu\!+\!\sigma^2\theta_4}\big(\!\mathrm{e}^{-\theta_{4}x_1}\!-\!\theta_{4}x_1\mathrm{e}^{-\theta_{4}x_1}\!\big)\!-\!\frac{\lambda A_3}{\mu\!+\!\sigma^2\theta_3}\big(\!\mathrm{e}^{\theta_{3}x_1}\!+\!\theta_{3}x_1\mathrm{e}^{\theta_{3}x_1}\!\big)\!+\!B_{3} \theta_{3}\mathrm{e}^{\theta_{3} x_1}\!-\!B_{4} \theta_{4}\mathrm{e}^{-\theta_{4}x_1}\!=\!-1,\\
\frac{\lambda A_4}{-\mu\!+\!\sigma^2\theta_4}\big(\!\theta_{4}^2x_2\mathrm{e}^{-\theta_{4}x_2}\!-\!2\theta_{4}\mathrm{e}^{-\theta_{4}x_2}\!\big)\!-\!\frac{\lambda A_3}{\mu\!+\!\sigma^2\theta_3}\big(\!\theta_{3}^2x_2\mathrm{e}^{\theta_{3}x_2}\!+\!2\theta_{3}\mathrm{e}^{\theta_{3}x_2}\!\big)\!+\!B_{3} \theta_{3}^2\mathrm{e}^{\theta_{3} x_2}\!+\!B_{4} \theta_{4}^2\mathrm{e}^{-\theta_{4}x_2} \!=\! 0.
\end{cases}
\end{equation}
}
\vskip 5pt 
Supposing that there exist  $0 < x_1 < x_2$  that solve \eqref{eq2:x1x2}, we show that $V$ and $f$ given by \eqref{fv2} with \eqref{V3} and \eqref{V4} form a solution to the extended HJB system \eqref{extendedhjb} under the pseudo-exponential discount function \eqref{eq:pseudo_exp_discount}. The crux is to establish the convexity of $V$.

\begin{proposition}\label{convex2}
Suppose that there exist  $0 < x_1 < x_2$ that solve \eqref{eq2:x1x2}. Then, the functions $V$ and $f$ specified in \eqref{fv2}, with $V_{3},V_{4}$ given by \eqref{V3} and \eqref{V4}, constitute a solution to the extended HJB system \eqref{extendedhjb} under the pseudo-exponential discount function \eqref{eq:pseudo_exp_discount}. In addition, the function $V$ is decreasing and convex on $[0,+\infty)$.
\end{proposition}
\begin{proof} {\sl \bf 
It is sufficient to show $V'(x) < -1$ for  $0\le x < x_1$ with $V'(x_1)=-1$; \ $-1 < V'(x) < -c$ for  $x_1 < x < x_2$;  $V'(x) =-c$ for  $x\ge x_2$.}
\vskip 4pt 
First, consider $V(x)=V_4(x)$ for $x \in [x_2, +\infty)$. It is straightforward that $V^\prime(x)=V_4^\prime(x) = -c < 0$ and $V^{\prime\prime}(x)=V_4^{\prime\prime}(x) =0$ for $x \in [x_2, +\infty)$. 
\vskip 4pt 
Second, for $x \in [x_1, x_2)$, as $V^\prime_3(x_2-) = V^\prime_3(x_2+)=V^\prime_4(x_2-) = V^\prime_4(x_2+)=-c < 0$. Similar to the proof of Proposition \ref{convex1}, we can show that $V^\prime_3(x) < 0$ for all $x \in [x_1, x_2)$ by considering different signs of $A_{3}$ and $A_{4}$. According to ODEs \eqref{eq:V4_ODE} and $V_4^{\prime\prime}(x_2-) =0$, we have $\frac{1}{2}\sigma^2V_4^{\prime\prime\prime}(x_2-)=\delta V^\prime_4(x_2-)-\lambda V^\prime_3(x_2-)-\mu V^{\prime\prime}_4(x_2-)=c(\lambda-\delta)<0$. Therefore, there exists $\epsilon_1>0$ sufficiently small such that $V^{\prime\prime}_4(x)>0, \forall x\in(x_2-\epsilon_1,x_2)$.
\vskip 4pt 
Next, we show $V_4^{\prime\prime}(x)>0, \forall x\in(x_1,x_2)$. Similar to the proof of Proposition \ref{convex1}, we introduce a change of variable $\varphi(x) := \int_{x_1}^x \mathrm{e}^{\frac{-2 \mu a}{\sigma^2}} \mathrm{d}a$, $x \in [x_1, x_2)$,
and define $l_i(z) := V_i^\prime(\varphi^{-1}(z))$, $z\in[0,\bar z)$, where $\bar z = \int_{x_1}^{x_2} \mathrm{e}^{\frac{-2 \mu a}{\sigma^2}} \mathrm{d}a$. We know that $l_3(z)=V_3^{\prime}(\varphi^{-1}(z))<0, \forall z\in[0,\bar z)$. According to ODEs \eqref{eq:V3_ODE} and \eqref{eq:V4_ODE}, we have that $l_3(z)$ and $l_4(z)$ are solutions to the following ODEs:
\begin{equation}\nonumber\begin{aligned}
 &l_3^{\prime\prime}(z) =\left[\varphi^\prime(\varphi^{-1}(z))\right]^{-2} \frac{2 \delta}{\sigma^2} l_3(z)<0,\ z\in[0,\bar z),\\
 &l_4^{\prime\prime}(z) =\left[\varphi^\prime(\varphi^{-1}(z))\right]^{-2} \frac{2 }{\sigma^2} \left[\delta l_4(z)-\lambda l_3(z)\right],\ z\in[0,\bar z).
\end{aligned}\end{equation}
\vskip 4pt 
Suppose  that there exists $z^* \in (0, \bar z)$ such that $l_4^\prime(z^*) = 0$ and $l_4^\prime(z) > 0$ for $z \in (z^*, \bar z)$.  Then $l_4^{\prime\prime}(z^*) \ge 0$, i.e., $\delta l_4(z^*)-\lambda l_3(z^*)\ge 0$. As $l_4^\prime(z) > 0$ for $z \in (z^*, \bar z)$, we have $l_4(z^*)<l_4(\bar z-)=V_4^\prime(x_2-)=-c<0$.  As such, $l_3(z^*)\le \frac{\delta}{\lambda}l_4(z^*)<-c\frac{\delta}{\lambda}<-c=l_3(\bar z-)$. Hence, there exists $z_1\in(z^*,\bar z)$ such that $l_3^{\prime}(z_1)>0$. Because $l_3^{\prime\prime}(z)<0, \forall z\in[0,\bar z)$, we know that $l_3^{\prime}(z)\ge l_3^{\prime}(z_1)>0, \forall z\in [0,z_1]$. Because $l_4^{\prime}(z^*)=0$ and $l_3^{\prime}(z^*)>0$, there exists $\epsilon_2>0$ sufficiently small such that $\delta l_4(z)-\lambda l_3(z)>0, \forall z\in(z^*-\epsilon_2,z^*)$.

We then claim that $\delta l_4(z)-\lambda l_3(z)>0$, i.e., $l_4^{\prime\prime}(z) > 0, \forall z\in[0,z^*)$.
\vskip 4pt 
Indeed,  assume that there exists $z_2\in[0,z^*)$ such that $l_4^{\prime\prime}(z_2)=0$ and $l_4^{\prime\prime}(z)>0, \forall z\in(z_2,z^*)$. Consequently, $l_4^{\prime}(z_2) < l_4^{\prime}(z^*) = 0$. Because $l_3^{\prime}(z)>0, \forall z\in [0,z_1]$, we have $l_3^{\prime}(z_2)>0$ and $\delta l_4^{\prime}(z_2)-\lambda l_3^{\prime}(z_2)<0$. On the other hand, as  $\delta l_4(z_2)-\lambda l_3(z_2)=0$ and $\delta l_4(z)-\lambda l_3(z)>0, \forall z\in(z_2,z^*)$, we have $\delta l_4^{\prime}(z_2)-\lambda l_3^{\prime}(z_2)\ge 0$, which is a contradiction.
\vskip 4pt \noindent
Similarly, we introduce a change of variable $\phi(x) := \int_{0}^x \mathrm{e}^{\frac{-2 (\mu+M) a}{\sigma^2}} \mathrm{d}a$, $x \in [0, x_1)$,
and define $m_i(z) := V_i^\prime(\phi^{-1}(z))$, $z\in[0,\tilde z)$, where $\tilde z = \int_{0}^{x_1} \mathrm{e}^{\frac{-2 (\mu+M) a}{\sigma^2}} \mathrm{d}a$. According to ODEs \eqref{eq:V3_ODE} and \eqref{eq:V4_ODE}, we have that $m_i(z)$ are solutions to the following ODEs:
\begin{equation}\nonumber\begin{aligned}
 &m_3^{\prime\prime}(z) =\left[\phi^\prime(\phi^{-1}(z))\right]^{-2} \frac{2 \delta}{\sigma^2} m_3(z),\ z\in[0,\tilde z),\\
 &m_4^{\prime\prime}(z) =\left[\phi^\prime(\phi^{-1}(z))\right]^{-2} \frac{2 }{\sigma^2} \left[\delta m_4(z)-\lambda m_3(z)\right],\ z\in[0,\tilde z).
\end{aligned}\end{equation}
Because $l_3^{\prime}(z)>0, \forall z\in [0,z_1]$, we have $l_3^{\prime}(0)=\left[\varphi^\prime(x_1)\right]^{-1}V_3^{\prime\prime}(x_1+)>0$,  i.e., $V_3^{\prime\prime}(x_1+)>0$. And $\delta l_4(0)-\lambda l_3(0)>0$, i.e., $V_3^{\prime}(x_1-)=V_3^{\prime}(x_1+)=l_3(0)<\frac{\delta}{\lambda}l_4(0)=\frac{\delta}{\lambda}V_4^{\prime}(x_1+)=-\frac{\delta}{\lambda}<-1$. Therefore, the ODE \eqref{eq:V3_ODE} indicates that
\begin{equation}\nonumber\begin{aligned}
&V^{\prime\prime}_3(x_1-) = \frac{2}{\sigma^2} \left[-(\mu+M) V_3^\prime(x_1-)+\delta V_3(x_1-) - M )\right]\\
&> \frac{2}{\sigma^2} \left[-\mu V_3^\prime(x_1+) + \delta V_3(x_1+)\right]= V_3^{\prime\prime}(x_1+) > 0.
\end{aligned}\end{equation}
Then, we claim that $V^{\prime\prime}_3(x)>0, \forall x\in(0,x_1)$, i.e., $m_3^\prime(z)>0,\forall z\in(0,\tilde z)$. Otherwise, there exists $z_3\in(0,\tilde z)$ such that $m_3^\prime(z_3)=0$ and $m_3^\prime(z)>0, \forall z\in(z_3,\tilde z)$. Thus, $m_3^{\prime\prime}(z_3)\ge 0$. On the other hand, $m_3(z_3)<m_3(\tilde z-)=V_3^\prime(x_1-)<-1<0$  leads to $m_3^{\prime\prime}(z_3)<0$, which is a contradiction.
\vskip 4pt 
Besides, as  $l_4^{\prime\prime}(z)>0, \forall z\in[0,z^*)$, we have $l_4^{\prime}(0)<l_4^{\prime}(z^*)=0$, i.e., $V_4^{\prime\prime}(x_1+)=V_4^{\prime\prime}(x_1-)<0$, i.e., $m^\prime_4(\tilde z-)<0$. Next, we claim $m^\prime_4(z)<0, \forall z\in[0,\tilde z)$. Otherwise, there exists $z_4\in[0,\tilde z)$ such that $m_4^\prime(z_4)=0$ and $m_4^\prime(z)<0, \forall z\in(z_4,\tilde z)$. Therefore, $m_4^{\prime\prime}(z_4)\le 0$, i.e., $\delta m_4(z_4)-\lambda m_3(z_4)\le 0$. On the other hand, as $m_4(z_4)>m_4(\tilde z-)=l_4(0)$ and  $m_3(z_4)<m_3(\tilde z-)=l_3(0)$, we have $\delta m_4(z_4)-\lambda m_3(z_4)>\delta l_4(0)-\lambda l_3(0)>0$, i.e., $m_4^{\prime\prime}(z_4)> 0$, which is a contradiction. 
\vskip 4pt 
Therefore, $m_4(z)>m_4(\tilde z-)=l_4(0), \forall z\in[0,\tilde z)$, $m_3(z)<m_3(\tilde z-)=l_3(0), \forall z\in[0,\tilde z)$, and $\delta m_4(z)-\lambda m_3(z)>\delta l_4(0)-\lambda l_3(0)>0, \forall z\in[0,\tilde z)$. Thus, $\delta V_4^\prime(x)-\lambda V_3^\prime(x)>0, \forall x \in[0,x^*)$, where $x^* = \varphi^{-1}(z^*)$. Hence, $\delta V_4(x^*)-\lambda V_3(x^*)>\delta V_4(0)-\lambda V_3(0)=0$. On the other hand, as $V^{\prime\prime}_4(x^*)=l_4^{\prime}(z^*)=0$ and $V_4^{\prime}(x^*)=l_4(z^*)<l_4(\bar z-)=V_4^{\prime}(x_2-)=-c$, combining with \eqref{eq:V4_ODE}, we have
\begin{equation}\nonumber
\delta V_4(x^*)-\lambda V_3(x^*) = \frac{1}{2} \sigma^2 V_4^{\prime\prime}(x^*) + \mu  V^{\prime}_4(x^*)  < -c\mu<0,
\end{equation}
which is a contradiction. Thus there cannot exist such $z^* \in (0, \bar z)$, and $l_4^\prime(z) > 0$ for all $z \in (0, \bar z)$, i.e., $V_4^{\prime\prime}(x) > 0, \forall x\in(x_1,x_2)$. Thus, $-1=V^{\prime}_4(x_1+)<V^{\prime}(x)=V^{\prime}_4(x)<V^{\prime}_4(x_2-)=-c, \forall x\in(x_1,x_2)$.
\vskip 5pt 
Finally, for $x \in [0, x_1)$, it is already known that $V_3(0)=V_4(0)= 0$, $V^\prime(x_1)=V^\prime_4(x_1-) = V^\prime_4(x_1+)=-1 < 0$ and $V^{\prime\prime}_4(x_1-)=V^{\prime\prime}_4(x_1+)\ge 0$.
\vskip 4pt 
We show $V_4^{\prime\prime}(x)>0, \forall x\in(0,x_1]$. Otherwise, suppose that there exists $z^{**} \in (0, \tilde z]$ such that $m_4^\prime(z^{**}) = 0$ and $m_4^\prime(z) > 0$ for $z \in (z^{**}, \tilde z]$.  Then $m_4^{\prime\prime}(z^{**}-) \ge 0$, i.e., $\delta m_4(z^{**})-\lambda m_3(z^{**})\ge 0$ (if $z^{**}=\tilde z$ and $m_4^\prime(\tilde z)=V_4^{\prime\prime}(x_1)= 0$, by \eqref{eq:V4_ODE} and $V_4^{\prime\prime}(x) > 0, \forall x\in(x_1,x_2)$, we also have $V_4^{\prime\prime\prime}(x_1-) =V_4^{\prime\prime\prime}(x_1+)\ge 0$ and $\delta m_4(z^{**})-\lambda m_3(z^{**})=\delta V_4^{\prime}(x_1)-\lambda V_3^{\prime}(x_1)\ge0$). Therefore, $m_4(z^{**})\le m_4(\tilde z)=V^\prime_4(x_1)=-1$ and $m_3(z^{**})\le \frac{\delta}{\lambda}m_4(z^{**})\le -\frac{\delta}{\lambda}<-1<-c$. Thus, $V_3^\prime(x^{**})<-c=V_3^\prime(x_2)$, where $x^{**} = \phi^{-1}(z^{**})$. Thus, there exists $x_0\in [x^{**},x_2)$ such that $V_3^{\prime\prime}(x_0-)>0$. If $V_3^\prime(x_1)>-1$, we can choose $x_0\in [x^{**},x_1]$. If $V_3^\prime(x_1)\le -1$ and $x_0> x_1$, we know that $l_3^{\prime\prime}(z)<0, \forall z\in[0,\bar z)$, thus $l_3^{\prime}(0)=l_3^{\prime}(\varphi(x_1))> l_3^{\prime}(\varphi(x_0))>0$, which indicates $V_3^{\prime\prime}(x_1-)\ge V_3^{\prime\prime}(x_1+)>0$ and we can choose $x_0=x_1 \in [x^{**},x_1]$. Thereby, we can always choose $x_0\in [x^{**},x_1]$ such that $m^\prime_3(\phi(x_0)-)>0$.
\vskip 4pt 
 If $m_3$ has a zero point in $[0,\tilde z]$, then based on \cite[Lemma 4.1]{shreve1984optimal}, $m^\prime_3$ has no zero point in $[0,\tilde z]$, thus $m^\prime_3(z-)>0, \forall z\in[0,\tilde z]$. Otherwise, $m_3(z)<0, \forall z\in[0,\tilde z]$ and $m_3^{\prime\prime}(z-)<0, \forall z\in[0,\tilde z]$, which indicates $m_3^\prime(z-)\ge m_3^\prime(\phi(x_0)-)>0, \forall z\in[0,\phi(x_0)]$.  In conclusion, we have $m_3^\prime(z-)>0, \forall z\in [0, z^{**}]$. Then, as  $m_4^{\prime}(z^{**})=0$, $m_3^{\prime}(z^{**}-)>0$ and $\delta m_4(z^{**})-\lambda m_3(z^{**})\ge 0$, there exists $\epsilon_3>0$ sufficiently small such that $\delta m_4(z)-\lambda m_3(z)>0, \forall z\in(z^{**}-\epsilon_3,z^{**})$.
\vskip 4pt 
We then claim that $\delta m_4(z)-\lambda m_3(z)>0$, i.e., $m_4^{\prime\prime}(z) > 0, \forall z\in[0,z^{**})$. Otherwise, there exists $z_5\in[0,z^{**})$ such that $m_4^{\prime\prime}(z_5)=0$ and $m_4^{\prime\prime}(z)>0, \forall z\in(z_5,z^{**})$. Consequently, $m_4^{\prime}(z_5) < m_4^{\prime}(z^{**}) = 0$. Because $m_3^{\prime}(z-)>0, \forall z\in [0,z^{**}]$, we have $m_3^{\prime}(z_5)>0$ and $\delta m_4^{\prime}(z_5)-\lambda m_3^{\prime}(z_5)<0$. On the other hand, as $\delta m_4(z_5)-\lambda m_3(z_5)=0$ and $\delta m_4(z)-\lambda m_3(z)>0, \forall z\in(z_5,z^{**})$, we have $\delta m_4^{\prime}(z_5)-\lambda m_3^{\prime}(z_5)\ge 0$, which is a contradiction. Thus, $\delta m_4(z)-\lambda m_3(z)>0, \forall z\in[0,z^{**})$, i.e., $\delta V_4^{\prime}(x)-\lambda V_3^{\prime}(x)>0, \forall x\in[0,x^{**})$.
\vskip 4pt 
Hence, $\delta V_4(x^{**})-\lambda V_3(x^{**})>\delta V_4(0)-\lambda V_3(0)=0$. However, as $V^{\prime\prime}_4(x^{**})=m_4^{\prime}(z^{**})=0$ and $V_4^{\prime}(x^{**})=m_4(z^{**})\le m_4(\tilde z-)=V_4^{\prime}(x_1-)=-1$, combining with \eqref{eq:V4_ODE}, we have
\begin{equation}\nonumber
\delta V_4(x^{**})-\lambda V_3(x^{**}) = \frac{1}{2} \sigma^2 V_4^{\prime\prime}(x^{**}) + (\mu+M)  V^{\prime}_4(x^{**})+M  \le -\mu<0,
\end{equation}
which is a contradiction. Therefore, there cannot exist such $z^{**} \in (0, \tilde z]$, and $m_4^\prime(z) > 0$ for all $z \in (0, \tilde z]$, i.e., $V_4^{\prime\prime}(x) > 0, \forall x\in(0,x_1]$. Thus, $V^{\prime}(x)=V^{\prime}_4(x)<V^{\prime}_4(x_1-)=-1, \forall x\in[0,x_1)$. 

Hence, $V=V_4$ is decreasing and convex on $[0,+\infty)$,  and \eqref{fv2}, \eqref{V3} and \eqref{V4} are a solution to the extended HJB system \eqref{extendedhjb} under the pseudo-exponential discount function \eqref{eq:pseudo_exp_discount}.
\end{proof}

\begin{remark}
Suppose that there exist $0 < x_1 < x_2$ solving \eqref{eq2:x1x2}. Similar to Remark \ref{differ}, although the function $f$ given by \eqref{fv2}, \eqref{V3} and \eqref{V4} only satisfies
\begin{equation}
\begin{aligned}
f\in &C^{1,1,1}\big(\big\{(x,t,s)\bigm| (x,t)\in\mathcal{\tilde{Q}}, s\in [0,t]\big\}\big)\\&\bigcap C^{2,1,1}\big(\big\{(x,t,s)\bigm| (x,t)\in\mathcal{\tilde{Q}},x\in[0,x_1)\cup(x_1,x_2)\cup(x_2,+\infty), s\in [0,t]\big\}\big),\end{aligned}\nonumber
\end{equation}
the It\^{o}--Tanaka--Meyer formula still holds  if we use $f_{xx}(x_1+,t,s)$ ($f_{xx}(x_2-,t,s)$) when $x=x_1$ ($x=x_2$) in the Verification Theorem \ref{verificationthm} and the Mild Verification Theorem \ref{mildverificationthm} (see \cite[Theorem 3.1 and Remark 3.3]{peskir2007change}). Thus, the proof remains valid. 
\end{remark}

Analogous to Proposition \ref{equilver}, we have the following result.
\begin{proposition}\label{equilver2}
Suppose that there exist $0 < x_1 < x_2$ solving \eqref{eq2:x1x2}. The candidate regular-singular control law $(\hat\Pi,\hat\Xi)$ defined in \eqref{hatpi2} and \eqref{hatxi2} is an equilibrium regular-singular control law, and the function $V$ given by \eqref{fv2} and \eqref{V4} is the corresponding equilibrium value function for the objective \eqref{objexmple} under the pseudo-exponential discount function \eqref{eq:pseudo_exp_discount}. In addition, $(\hat\Pi,\hat\Xi)$ is a mild equilibrium.
\end{proposition}

Proposition \ref{equilver2} shows that, if there exist  $0 < x_1 < x_2$ that solve \eqref{eq2:x1x2}, the equilibrium regular-singular control law is to apply the maximal level of effort when the surplus is below the effort threshold $x_1$ and to pay dividends once the surplus exceeds the dividend threshold $x_2$.

As in the previous example, we can also establish the existence of $0=x_{1}<x_{2}$ for the degenerate case under some conditions on the parameters.

Considering the case $x_1=0$, the candidate regular-singular control law $(\hat\Pi,\hat\Xi)$ is 
\begin{equation}\label{hatpideg2}
\hat \Pi(x,t) := 0,\ \forall (x,t)\in\mathcal{\tilde{Q}},
\end{equation}
and
\begin{equation}
\begin{aligned}
W^{\hat{\Xi}}:=\big\{(x,t)\in\mathcal{\tilde{Q}}\bigm|0\leq x< x_2\big\},\ \ P^{\hat{\Xi}}:=\big\{(x,t)\in\mathcal{\tilde{Q}}\bigm|x\ge x_2\big\}.
\end{aligned}\label{hatxideg2}
\end{equation}
For fixed $0< x_2$, $V_3(x)$ and $V_4(x)$ are given by the systems of ODEs
\begin{equation}\label{eq:V3_ODE2}
\begin{cases}
\frac{1}{2} \sigma^2 V^{\prime\prime}_3(x) + \mu V^\prime_3(x) - \delta V_3(x) = 0, \qquad & x \in [0, x_2),\\
-c- V^\prime_3(x) = 0, & x \in [x_2, +\infty),\\
V_3(0)=0,
\end{cases}
\end{equation}
and
\begin{equation}\label{eq:V4_ODE2}
\begin{cases}
\frac{1}{2} \sigma^2 V^{\prime\prime}_4(x) + \mu V^\prime_4(x) - \delta V_4(x)+\lambda V_3(x) = 0, \qquad & x \in [0, x_2),\\
-c- V^\prime_4(x) = 0, & x \in [x_2, +\infty),\\
V_4(0)=0.
\end{cases}
\end{equation}
Thus, a general solution to the ODEs \eqref{eq:V3_ODE2} has the form
\begin{equation}\label{V3deg}
V_3(x) = \begin{cases}
\tilde{A}_{3} \mathrm{e}^{\theta_{3} x} + \tilde{A}_{4} \mathrm{e}^{-\theta_{4}x}, & x \in [0, x_2),\\
-cx + \tilde{A}_{5}, \qquad & x \in [x_2, +\infty),
\end{cases}
\end{equation}
where $\tilde{A}_{i}$ are constants to be determined.
\vskip 4pt 
To determine the values of $\tilde{A}_{i}$, we use the boundary condition and the smooth fit principle, i.e., 
\begin{equation}\nonumber
V_3(0) = 0,\ V_3(x_2 -) = V_3(x_2+),\ V^\prime_3(x_2 -) = V^\prime_3(x_2+).
\end{equation}
Then 
\begin{equation}\nonumber
\begin{aligned}
 &\tilde{A}_{3} =- \tilde{A}_{4}=-\frac{c}{\theta_{3} \mathrm{e}^{\theta_{3}x_2}+\theta_{4}\mathrm{e}^{-\theta_{4}x_2}},\ \ \tilde{A}_{5}=cx_2+\tilde{A}_{3} \mathrm{e}^{\theta_{3}x_2} + \tilde{A}_{4} \mathrm{e}^{-\theta_{4}x_2}.
 \end{aligned}
\end{equation}
\vskip 4pt \noindent
After obtaining $V_3$, a general solution to the ODEs \eqref{eq:V4_ODE2} has the form
\begin{equation}\label{V4deg}
V_4(x) = \begin{cases}
-\frac{\lambda \tilde{A}_3}{\mu+\sigma^2\theta_3}x\mathrm{e}^{\theta_{3}x}+\frac{\lambda \tilde{A}_4}{-\mu+\sigma^2\theta_4}x\mathrm{e}^{-\theta_{4}x}+\tilde{B}_{3} \mathrm{e}^{\theta_{3} x} + \tilde{B}_{4} \mathrm{e}^{-\theta_{4}x}, & x \in [0, x_2),\\
-cx +\tilde{B}_{5}, & x \in [x_2, +\infty),
\end{cases}
\end{equation}
where $\tilde{B}_{i}$ are the constants to be determined. 
\vskip 4pt \noindent
Using the boundary condition and the smooth fit principle to determine the values of $\tilde{B}_{i}$, we have 
\begin{equation}\nonumber
V_4(0) = 0,\ V_4(x_2 -) = V_4(x_2+),\ V^\prime_4(x_2 -) = V^\prime_4(x_2+).
\end{equation}
Then $\tilde{B}_{i}$ satisfy the following linear equation system: {\small 
\begin{equation}\nonumber
\begin{cases}
\tilde{B}_{3} \!+\! \tilde{B}_{4} \!=\! 0, \\
\frac{\lambda \tilde{A}_4}{-\mu\!+\!\sigma^2\theta_4}x_2\mathrm{e}^{-\theta_{4}x_2}\!-\!\frac{\lambda \tilde{A}_3}{\mu+\sigma^2\theta_3}x_2\mathrm{e}^{\theta_{3}x_2}\!+\!\tilde{B}_{3} \mathrm{e}^{\theta_{3} x_2} \!+\! \tilde{B}_{4} \mathrm{e}^{-\theta_{4}x_2}\! =\! -cx_2 \!+\! \tilde{B}_{5},\\
\frac{\lambda \tilde{A}_4}{-\mu\!+\!\sigma^2\theta_4}\big(\!\mathrm{e}^{-\theta_{4}x_2}\!-\!\theta_{4}x_2\mathrm{e}^{-\theta_{4}x_2}\!\big)\!-\!\frac{\lambda \tilde{A}_3}{\mu\!+\!\sigma^2\theta_3}\big(\!\mathrm{e}^{\theta_{3}x_2}\!+\!\theta_{3}x_2\mathrm{e}^{\theta_{3}x_2}\!\big)\!+\!\tilde{B}_{3} \theta_{3}\mathrm{e}^{\theta_{3} x_2}\!-\!\tilde{B}_{4} \theta_{4}\mathrm{e}^{-\theta_{4}x_2}  \!=\! -c.
\end{cases}
\end{equation} }
\vskip 4pt \noindent
Solving the above system yields
\begin{equation}\nonumber
\begin{aligned}
 \tilde{B}_{3} \!=\!- \!\tilde{B}_{4}&\!=\!\frac{1}{\theta_{3} \mathrm{e}^{\theta_{3}x_2}\!+\!\theta_{4}\mathrm{e}^{-\theta_{4}x_2}}\left[\!\frac{\lambda \tilde{A}_3}{\sqrt{\mu^2 \!+\! 2 \delta \sigma^2}}\big(\!\mathrm{e}^{\theta_{3}x_2}\!+\!\theta_{3}x_2\mathrm{e}^{\theta_{3}x_2}\!+\!\mathrm{e}^{-\theta_{4}x_2}\!-\!\theta_{4}x_2\mathrm{e}^{-\theta_{4}x_2}\!\big)\!-\!c\!\right]\\
 &\!=\!-\frac{c}{\theta_{3} \mathrm{e}^{\theta_{3}x_2}\!+\!\theta_{4}\mathrm{e}^{-\theta_{4}x_2}}\left[\!1\!+\!\frac{\lambda }{\sqrt{\mu^2 \!+\! 2 \delta \sigma^2}}\frac{\mathrm{e}^{\theta_{3}x_2}\!+\!\theta_{3}x_2\mathrm{e}^{\theta_{3}x_2}\!+\!\mathrm{e}^{-\theta_{4}x_2}\!-\!\theta_{4}x_2\mathrm{e}^{-\theta_{4}x_2}}{\theta_{3} \mathrm{e}^{\theta_{3}x_2}\!+\!\theta_{4}\mathrm{e}^{-\theta_{4}x_2}}\!\right],\end{aligned}
\end{equation}
and $\tilde{B}_{5}=cx_2-\frac{\lambda \tilde{A}_3}{\mu+\sigma^2\theta_3}x_2\mathrm{e}^{\theta_{3}x_2}+\frac{\lambda \tilde{A}_4}{-\mu+\sigma^2\theta_4}x_2\mathrm{e}^{-\theta_{4}x_2}+\tilde{B}_{3} \mathrm{e}^{\theta_{3} x_2} + \tilde{B}_{4} \mathrm{e}^{-\theta_{4}x_2}$.
\vskip 4pt \noindent
It remains to determine the value of $x_2$, which is solved from $V^{\prime\prime}_4(x_2-) = 0$, i.e., {\small 
\begin{equation}\nonumber
 \frac{\lambda \tilde{A}_4}{-\mu\!+\!\sigma^2\theta_4}\big(\!\theta_{4}^2x_2\mathrm{e}^{-\theta_{4}x_2}\!-\!2\theta_{4}\mathrm{e}^{-\theta_{4}x_2}\!\big)\!-\!\frac{\lambda \tilde{A}_3}{\mu\!+\!\sigma^2\theta_3}\big(\!\theta_{3}^2x_2\mathrm{e}^{\theta_{3}x_2}\!+\!2\theta_{3}\mathrm{e}^{\theta_{3}x_2}\!\big)\!+\!\tilde{B}_{3} \theta_{3}^2\mathrm{e}^{\theta_{3} x_2}\!+\!\tilde{B}_{4} \theta_{4}^2\mathrm{e}^{-\theta_{4}x_2} \!=\! 0,
\end{equation}}
which is equivalent to {\small 
\begin{equation}\label{eq2:x1x23}
\begin{aligned}
&\mathrm{e}^{(\theta_{3}+\theta_{4})x_2}\big(-\lambda\theta_{3}^2+\theta_{3}^3\sqrt{\mu^2 + 2 \delta \sigma^2}\big)+\big(\lambda\theta_{3}^2-\lambda\theta_{4}^2+\theta_{3}\theta_{4}(\theta_{3}-\theta_{4})\sqrt{\mu^2 + 2 \delta \sigma^2}\big)\\&+\mathrm{e}^{-(\theta_{3}+\theta_{4})x_2}\big(\lambda\theta_{4}^2-\theta_{4}^3\sqrt{\mu^2 + 2 \delta \sigma^2}\big)-2\lambda x_2\theta_{3}\theta_{4}(\theta_{3}+\theta_{4})=0.
\end{aligned}
\end{equation}
}
Let
\begin{equation}\nonumber
\begin{aligned}
G(x)&:=\mathrm{e}^{(\theta_{3}+\theta_{4})x}\big(-\lambda\theta_{3}^2+\theta_{3}^3\sqrt{\mu^2 + 2 \delta \sigma^2}\big)+\big(\lambda\theta_{3}^2-\lambda\theta_{4}^2+\theta_{3}\theta_{4}(\theta_{3}-\theta_{4})\sqrt{\mu^2 + 2 \delta \sigma^2}\big)\\
&+\mathrm{e}^{-(\theta_{3}+\theta_{4})x}\big(\lambda\theta_{4}^2-\theta_{4}^3\sqrt{\mu^2 + 2 \delta \sigma^2}\big)-2\lambda x\theta_{3}\theta_{4}(\theta_{3}+\theta_{4}).
\end{aligned}
\end{equation}
\vskip 4pt \noindent
It is easy to see $G(0)=(\theta_{3}- \theta_{4})(\theta_{3}+\theta_{4})^2\sqrt{\mu^2 + 2 \delta \sigma^2}<0$. Meanwhile, thanks to $\lambda<\delta$ and $\mu^2 +  \delta \sigma^2-\mu\sqrt{\mu^2 + 2 \delta \sigma^2}>0$, we know $\theta_3\sqrt{\mu^2 + 2 \delta \sigma^2}-\lambda=\frac{\mu^2 +  2\delta \sigma^2-\mu\sqrt{\mu^2 + 2 \delta \sigma^2}-\lambda \sigma^2}{\sigma^2}>\frac{\mu^2 +  \delta \sigma^2-\mu\sqrt{\mu^2 + 2 \delta \sigma^2}}{\sigma^2}>0$. Then, it is easy to show that $G(+\infty)=\lim\limits_{x\uparrow +\infty}\big(\theta_3\sqrt{\mu^2 + 2 \delta \sigma^2}-\lambda\big)\theta_3^2\mathrm{e}^{(\theta_{3}+\theta_{4})x}=+\infty>0$. Thus, Equation \eqref{eq2:x1x23} has a solution $x_2^{\lambda}>0$. Moreover,
\begin{equation}\nonumber\begin{aligned}G^{\prime}(x)&=(\theta_{3}+\theta_{4})\big[\theta_{3}^2(\theta_{3}\sqrt{\mu^2 + 2 \delta \sigma^2}-\lambda)\mathrm{e}^{(\theta_{3}+\theta_{4})x}+\theta_{4}^2(\theta_{4}\sqrt{\mu^2 + 2 \delta \sigma^2}-\lambda)\mathrm{e}^{-(\theta_{3}+\theta_{4})x}-2\lambda \theta_{3}\theta_{4} \big]\\
&\ge 2(\theta_{3}+\theta_{4})\theta_{3}\theta_{4}\Big[\sqrt{(\theta_{3}\sqrt{\mu^2 + 2 \delta \sigma^2}-\lambda)(\theta_{4}\sqrt{\mu^2 + 2 \delta \sigma^2}-\lambda)}-\lambda \Big]>0, \forall x \in [0,+\infty),
\end{aligned}\end{equation}
where the strict inequality follows from $(\theta_{3}\sqrt{\mu^2 + 2 \delta \sigma^2}-\lambda)(\theta_{4}\sqrt{\mu^2 + 2 \delta \sigma^2}-\lambda)-\lambda^2=(\theta_{3}+\theta_{4})\sqrt{\mu^2 + 2 \delta \sigma^2}(\delta-\lambda)>0$. Thus, the solution $x_2^{\lambda}>0$ to Equation \eqref{eq2:x1x23} is unique, and $x_2^{\lambda}$ is continuous w.r.t. $\lambda \in [0,\delta)$ by the implicit function theorem.
\vskip 10pt \noindent
To derive the condition for $(\hat\Pi,\hat\Xi)$ defined in \eqref{hatpideg2} and \eqref{hatxideg2} being an equilibrium in this case, we introduce the function $H$ as follows:
\begin{equation}\nonumber
\begin{aligned}
H(x_2):=&\lambda(\theta_3+\theta_4)\big(-\mathrm{e}^{\theta_{3}x_2}-\theta_{3}x_2\mathrm{e}^{\theta_{3}x_2}-\mathrm{e}^{-\theta_{4}x_2}+\theta_{4}x_2\mathrm{e}^{-\theta_{4}x_2}\big)\\&+2\lambda\big(\theta_{3} \mathrm{e}^{\theta_{3}x_2}+\theta_{4}\mathrm{e}^{-\theta_{4}x_2}\big)-\sqrt{\mu^2 + 2 \delta \sigma^2}(\theta_3+\theta_4)\big(\theta_{3} \mathrm{e}^{\theta_{3}x_2}+\theta_{4}\mathrm{e}^{-\theta_{4}x_2}\big)\\
=&-\lambda(\theta_3+\theta_4)x_2\big(\theta_{3} \mathrm{e}^{\theta_{3}x_2}-\theta_{4}\mathrm{e}^{-\theta_{4}x_2}\big)+\lambda(\theta_3-\theta_4)\big(\mathrm{e}^{\theta_{3}x_2}-\mathrm{e}^{-\theta_{4}x_2}\big)\\&-\sqrt{\mu^2 + 2 \delta \sigma^2}(\theta_3+\theta_4)\big(\theta_{3} \mathrm{e}^{\theta_{3}x_2}+\theta_{4}\mathrm{e}^{-\theta_{4}x_2}\big).\end{aligned}
\end{equation}

If  $\theta_{3} \mathrm{e}^{\theta_{3}x_2}-\theta_{4}\mathrm{e}^{-\theta_{4}x_2}\ge 0$, then  it is easy to see that $H(x_2)<0$;  If $\theta_{3} \mathrm{e}^{\theta_{3}x_2}-\theta_{4}\mathrm{e}^{-\theta_{4}x_2}< 0$, then   $x_2<\frac{1}{\theta_3+\theta_4}\ln{\frac{\theta_{4}}{\theta_{3}}}$,  thus
\begin{equation}\nonumber
\begin{aligned}
H(x_2)\le&-\lambda\ln{\frac{\theta_{4}}{\theta_{3}}}\big(\theta_{3} \mathrm{e}^{\theta_{3}x_2}-\theta_{4}\mathrm{e}^{-\theta_{4}x_2}\big)+\lambda(\theta_3-\theta_4)\big(\mathrm{e}^{\theta_{3}x_2}-\mathrm{e}^{-\theta_{4}x_2}\big)\\&-\sqrt{\mu^2 + 2 \delta \sigma^2}(\theta_3+\theta_4)\big(\theta_{3} \mathrm{e}^{\theta_{3}x_2}+\theta_{4}\mathrm{e}^{-\theta_{4}x_2}\big)\\
<& \lambda\theta_4\ln{\frac{\theta_{4}}{\theta_{3}}}\mathrm{e}^{-\theta_{4}x_2}+\lambda(\theta_4-\theta_3)\mathrm{e}^{-\theta_{4}x_2}-\sqrt{\mu^2 + 2 \delta \sigma^2}(\theta_3+\theta_4)\theta_{4}\mathrm{e}^{-\theta_{4}x_2}\\
<&\delta\theta_4\ln{\frac{\theta_{4}}{\theta_{3}}}\mathrm{e}^{-\theta_{4}x_2}+\delta(\theta_4-\theta_3)\mathrm{e}^{-\theta_{4}x_2}-\sqrt{\mu^2 + 2 \delta \sigma^2}(\theta_3+\theta_4)\theta_{4}\mathrm{e}^{-\theta_{4}x_2}\\
=&\delta\theta_4\ln{\frac{\theta_{4}}{\theta_{3}}}\mathrm{e}^{-\theta_{4}x_2}+\delta(\theta_4-\theta_3)\mathrm{e}^{-\theta_{4}x_2}-\delta\frac{\theta_3+\theta_4}{\theta_3\theta_4}(\theta_3+\theta_4)\theta_{4}\mathrm{e}^{-\theta_{4}x_2}\\
=&\delta\theta_3\mathrm{e}^{-\theta_{4}x_2}\big(\frac{\theta_{4}}{\theta_{3}}\ln{\frac{\theta_{4}}{\theta_{3}}}+\frac{\theta_{4}}{\theta_{3}}-1-\frac{\theta_{4}^2}{\theta_{3}^2}-2\frac{\theta_{4}}{\theta_{3}}-1\big)\\
<&\delta\theta_3\mathrm{e}^{-\theta_{4}x_2}\big(\frac{\theta_{4}}{\theta_{3}}\big(\frac{\theta_{4}}{\theta_{3}}-1\big)+\frac{\theta_{4}}{\theta_{3}}-1-\frac{\theta_{4}^2}{\theta_{3}^2}-2\frac{\theta_{4}}{\theta_{3}}-1\big)<0  .\end{aligned}
\end{equation}
Denote 
\begin{equation}\nonumber c_*(\lambda)=-\frac{\sqrt{\mu^2 + 2 \delta \sigma^2}}{H(x_2^{\lambda}) }\big(\theta_{3} \mathrm{e}^{\theta_{3}x_2^{\lambda}}+\theta_{4}\mathrm{e}^{-\theta_{4}x_2^{\lambda}}\big)^2,\end{equation}
 where $x_2^{\lambda}>0$ is the unique solution to Equation \eqref{eq2:x1x23}. We know that $c_*(\lambda)>0$ is continuous w.r.t. $\lambda \in [0,\delta)$. It is easy to show that $x_2^{0}=\frac{2}{\theta_3+\theta_4}\ln\frac{\theta_4}{\theta_3}$, $c_*(0)=\big(\frac{\theta_4}{\theta_3}\big)^{\frac{\theta_3-\theta_4}{\theta_3+\theta_4}}$. When $x_2=x_2^{\lambda}$, $V_4$ given by \eqref{V4deg} satisfies
\begin{equation}\nonumber\begin{aligned}V_4^{\prime}(0)\!=\!&-\!\frac{\lambda \tilde{A}_3}{\mu\!+\!\sigma^2\theta_3}\!+\!\frac{\lambda \tilde{A}_4}{-\mu\!+\!\sigma^2\theta_4}\!+\!\tilde{B}_{3} \theta_{3}\!-\!\tilde{B}_{4} \theta_{4}\!=\!\frac{c\lambda }{\sqrt{\mu^2 \!+\! 2 \delta \sigma^2}}\frac{2}{\theta_{3} \mathrm{e}^{\theta_{3}x_2}\!+\!\theta_{4}\mathrm{e}^{-\theta_{4}x_2}}\\&-\!\frac{c(\theta_3\!+\!\theta_4)}{\theta_{3} \mathrm{e}^{\theta_{3}x_2}\!+\!\theta_{4}\mathrm{e}^{-\theta_{4}x_2}}\Big[\!1\!+\!\frac{\lambda }{\sqrt{\mu^2 \!+\! 2 \delta \sigma^2}}\frac{\mathrm{e}^{\theta_{3}x_2}\!+\!\theta_{3}x_2\mathrm{e}^{\theta_{3}x_2}\!+\!\mathrm{e}^{-\theta_{4}x_2}\!-\!\theta_{4}x_2\mathrm{e}^{-\theta_{4}x_2}}{\theta_{3} \mathrm{e}^{\theta_{3}x_2}\!+\!\theta_{4}\mathrm{e}^{-\theta_{4}x_2}}\!\Big]\\
=\!&\frac{c}{\left(\theta_{3} \mathrm{e}^{\theta_{3}x_2}\!+\!\theta_{4}\mathrm{e}^{-\theta_{4}x_2}\right)^2}\frac{1 }{\sqrt{\mu^2 \!+\! 2 \delta \sigma^2}}\Big[\!\lambda(\theta_3\!+\!\theta_4)\big(\!\theta_{4}x_2\mathrm{e}^{-\theta_{4}x_2}\!-\!\mathrm{e}^{\theta_{3}x_2}\!-\!\theta_{3}x_2\mathrm{e}^{\theta_{3}x_2}\!-\!\mathrm{e}^{-\theta_{4}x_2}\!\big)\\&+\!2\lambda\big(\!\theta_{3} \mathrm{e}^{\theta_{3}x_2}\!+\!\theta_{4}\mathrm{e}^{-\theta_{4}x_2}\!\big)\!-\!\sqrt{\mu^2 \!+\! 2 \delta \sigma^2}(\theta_3\!+\!\theta_4)\big(\!\theta_{3} \mathrm{e}^{\theta_{3}x_2}\!+\!\theta_{4}\mathrm{e}^{-\theta_{4}x_2}\!\big)\!\Big]=-\frac{c}{c_*(\lambda)}.\end{aligned}\end{equation}
By the the same argument as Proposition \ref{convex2}, we know that $V^{\prime\prime}=V_4^{\prime\prime}>0, \forall x\in(0,x_2)$. Therefore, we know that $V_4^{\prime}(0)<V_4^{\prime}(x_2)=-c$, i.e. $c_*(\lambda)\in (0,1), \forall \lambda\in[0,\delta)$. Then we have the following.
\begin{proposition}\label{cdeg2}
Suppose $0<c\le c_*(\lambda)$, $\lambda\in(0,\delta)$, and $x_2=x_2^{\lambda}$, where $x_2^{\lambda}>0$ is the unique solution to \eqref{eq2:x1x23}. Then, the functions $V$ and $f$ specified in \eqref{fv2} with $V_3, V_4$ given by \eqref{V3deg} and \eqref{V4deg} constitute a solution to the extended HJB system \eqref{extendedhjb} under the pseudo-exponential discount function \eqref{eq:pseudo_exp_discount}. In addition, the function $V$ is decreasing and convex on $[0,+\infty)$. Furthermore, the candidate regular-singular control law $(\hat\Pi,\hat\Xi)$ defined in \eqref{hatpideg2} and \eqref{hatxideg2} is an equilibrium regular-singular control law, and $V$ in \eqref{fv2} with \eqref{V4deg} is the corresponding equilibrium value function for the objective \eqref{objexmple} under the pseudo-exponential discount function \eqref{eq:pseudo_exp_discount}. Besides, $(\hat\Pi,\hat\Xi)$ is a mild equilibrium.
\end{proposition}
\begin{proof}
Using the same argument as Proposition \ref{convex2}, we know that $V^{\prime\prime}=V_4^{\prime\prime}>0, \forall x\in(0,x_2)$. Therefore, \eqref{fv2}, \eqref{V3deg} and \eqref{V4deg} are a solution to the extended HJB system \eqref{extendedhjb} under the pseudo-exponential discount function \eqref{eq:pseudo_exp_discount} if $V_4^{\prime}(0)\ge -1$, that is $V_4^{\prime}(0)=-\frac{c}{c_*(\lambda)}\ge -1$, i.e., $0<c\le c_*(\lambda)$. The remaining proofs are similar to Remark \ref{differ} and Proposition \ref{equilver}.
\end{proof}

Suppose $0<c\le c_*(\lambda)$, Proposition \ref{cdeg2} shows that the equilibrium regular-singular control law is to pay dividends once the surplus exceeds the dividend threshold $x_2^{\lambda}$ and apply no effort even if the firm approaches ruin, because $c$ is relatively small and the cost for dividend is relatively large.

\begin{remark}
For the case $x_1=x_2=0$, it holds  that $V_3(x)=V_4(x)=-cx, \forall x\ge0$. From which it is easy to show that \eqref{fv2} in this case is not a solution to the extended HJB system \eqref{extendedhjb} under the discount function \eqref{eq:pseudo_exp_discount}. 
\end{remark}

In the following, we show that  for any $\lambda\in(0,\delta)$, when $c_*(\lambda)<c<1$, there exists a solution $0 < x_1 < x_2$ to \eqref{eq2:x1x2}. First, we need the following lemmas \ref{lem:comparision}-\ref{lem:exponential}.

\begin{lemma}[Piecewise comparison principle with $C^1$ matching]
\label{lem:comparision}
Let $0<a<b$, $\sigma,\mu,M,\delta>0$. Set
$u_a(x)=M1_{\{x<a\}}$ and
$\mathcal{L}_a z(x)=\frac{1}{2}\sigma^2z^{\prime\prime}(x)+(\mu+u_a(x))z^{\prime}(x)-\delta z(x)$.
If
\[
z\in C^1\big([0,b]\big)\bigcap C^2\big((0,a)\cup(a,b)\big),
\qquad \mathcal{L}_a z(x)\le0,\ \forall x \in (0,a)\cup(a,b),
\]
and $z(0)\ge0$, $z^{\prime}(b)=0$, then $z(x)\ge0, \forall x\in[0,b]$.
\end{lemma}

\begin{proof}
Suppose that $m=\min\limits_{x\in[0,b]}z(x)<0$, attained at $x_0\in[0,b]$.
Then $x_0\in(0,b]$ and $z^{\prime}(x_0)=0$ by $C^1$ regularity
and the boundary conditions. 
Since $z,z^{\prime}$ are continuous, we may choose
$\varepsilon\in(0,x_0)$ such that $(x_0-\varepsilon,x_0)$
does not contain $a$, $z(x)<\frac{m}{2}$ and $|z^{\prime}(x)|<-\frac{\delta m}{4(\mu+M)}$, $\forall x \in (x_0-\varepsilon,x_0)$. Therefore,
\[
\frac{1}{2}\sigma^2z^{\prime\prime}(x)\le\delta z(x)-(\mu+u_a(x))z^{\prime}(x)
\le\frac{\delta m}{4}<0,
\quad \forall x \in (x_0-\varepsilon,x_0),
\]
and for any $x \in (x_0-\varepsilon,x_0)$, we have $z^{\prime}(x)\ge -\frac{\delta m}{2\sigma^2}(x_0-x)$. Therefore,
\[
z(x_0)-z(x_0-\varepsilon)
=\int_{x_0-\varepsilon}^{x_0}z^{\prime}(z)\mathrm{d}z
\ge-\frac{\delta m}{4\sigma^2}\varepsilon^2>0,
\]
contradicting with the minimality of $x_0$.
\end{proof}

Using Lemma \ref{lem:comparision}, we have the following estimation.
\begin{lemma}[Lower bounds and a threshold bound]
\label{lem:UW-bounds}
For fixed $0 < x_1 < x_2$, $0\le \lambda <\delta$, $V_{3},V_{4}$ given by \eqref{V3} and \eqref{V4} satisfy 
\[
V_3(x)\ge-\frac{c}{\theta_1}\mathrm{e}^{\theta_1(x-x_2)}\ge-\frac{c}{\theta_1},
\quad
V_4(x)\ge-\frac{c}{\theta_1}\mathrm{e}^{\theta_1(x-x_2)}-\frac{\lambda c}{\delta \theta_1}, \quad \forall x \in[0,x_2].
\]
Moreover, if $0 < x_1 < x_2$ solve \eqref{eq2:x1x2}, we have \[
x_2\le\frac{1+\lambda/\delta}{\theta_1}.
\]
\end{lemma}

\begin{proof}
Set $\phi(x)=-\frac{c}{\theta_1}\mathrm{e}^{\theta_1(x-x_2)}$. We know that $\theta_1>0$ satisfies $\frac{1}{2}\sigma^2\theta_1^2+(\mu+M)\theta_1-\delta=0$. Then by Lemma \ref{lem:comparision}, we have $\mathcal{L}_{x_1}\phi(x)=(u_{x_1}(x)-M)\phi^{\prime}(x)\ge0$.
Thus $\mathcal{L}_{x_1}(V_3(x)-\phi(x))=-u_{x_1}(x)-\mathcal{L}_{x_1}\phi(x)\le0, \forall x\in (0,x_1)\cup(x_1,x_2)$, while
$V_3(0)-\phi(0)=\frac{c}{\theta_1}\mathrm{e}^{-\theta_1x_2}>0$ and $V_3^{\prime}(x_2)-\phi^{\prime}(x_2)=0$.
The piecewise comparison principle Lemma \ref{lem:comparision} yields $V_3(x)\ge\phi(x)\ge-\frac{c}{\theta_1}, \forall x \in[0,x_2]$.

Next, set $\psi(x)=\phi(x)-\frac{\lambda c}{\delta \theta_1}$.
Since $\mathcal{L}_{x_1}\psi(x)=\mathcal{L}_{x_1}\phi(x)+\frac{\lambda c}{\theta_1}$, we have
\[
\mathcal{L}_{x_1}(V_4(x)-\psi(x))
=-u_{x_1}(x)-\mathcal{L}_{x_1}\phi(x)-\lambda\left(V_3(x)+\frac{c}{\theta_1}\right)\le0, \quad \forall x\in (0,x_1)\cup(x_1,x_2).
\]
Moreover, $V_4(0)-\psi(0)=\frac{c}{\theta_1}\mathrm{e}^{-\theta_1x_2}+\frac{\lambda c}{\delta\theta_1}>0$
and $V_4^{\prime}(x_2)-\psi^{\prime}(x_2)=0$.
Another application of the comparison principle Lemma \ref{lem:comparision} gives $V_4(x)\ge\psi(x), \forall x \in[0,x_2]$.

Finally, if $0 < x_1 < x_2$ solve \eqref{eq2:x1x2}, by Proposition \ref{convex2}, we know that $V_4$ is convex on $[0,x_2]$. Therefore, $V_4^{\prime}(x)\le -c, \forall x \in[0,x_2]$. Hence
\[
-\frac{c}{\theta_1}\left(1+\frac{\lambda}{\delta}\right)
\le V_4(x_2)=\int_0^{x_2} V_4^{\prime}(x)\mathrm{d}x \le-cx_2,
\]
which proves the bound for $x_2$.
\end{proof}

For fixed $0< x_1 < x_2$, $0<\lambda<\delta$, $0<c<1$, we define 
\[F_1(x_1,x_2;\lambda,c)=V_4^{\prime}(x_1)+1,\quad F_2(x_1,x_2;\lambda,c)=\frac{2}{\sigma^2}(\delta V_4(x_2)-\lambda V_3(x_2)+\mu c),\]
where $V_{3},V_{4}$ given by \eqref{V3} and \eqref{V4}. It is easy to see that $(F_1,F_2)$ is continuous on $0\le x_1 \le x_2$, $0\le \lambda<\delta$, $0<c<1$. For fixed $0<\lambda<\delta$, $0<c<1$, a solution $0 < x_1 < x_2$ to \eqref{eq2:x1x2} is equivalent to a solution $0 < x_1 < x_2$ to $(F_1(x_1,x_2;\lambda,c),F_2(x_1,x_2;\lambda,c))=(0,0)$.

To facilitate establishing the existence of a solution $0 < x_1 < x_2$ to $(F_1(x_1,x_2;\lambda,c),F_2(x_1,x_2;\lambda,c))=(0,0)$, we next investigate another simpler case with $\lambda=0$, which is the time-consistent case.

\begin{lemma}[The exponential starting point]
\label{lem:exponential}
If $c_*(0)<c<1$, then $(F_1(x_1,x_2;0,c),F_2(x_1,x_2;0,c))=(0,0)$ has exactly one solution with $0<x_1<x_2$. At that solution,
\begin{equation}
\mathcal{D}_{(x_1,x_2)}(F_1,F_2)(x_1,x_2;0,c)
=
\begin{pmatrix}
V_3^{\prime\prime}(x_1)&0\\
0&-\frac{2\delta c}{\sigma^2}
\end{pmatrix},
\
\det \mathcal{D}_{(x_1,x_2)}(F_1,F_2)(x_1,x_2;0,c)
=-\frac{2\delta c}{\sigma^2}V_3''(x_1)<0.
\label{eq:jacobian-claim}
\end{equation}
Thus the solution is a simple zero. Moreover, $V_3=V_4$ given by \eqref{V3} and \eqref{V4} is convex
on $[0,+\infty)$ and $V_3^{\prime\prime}(x)>0,\forall 0<x<x_2$.
\end{lemma}

\begin{proof}
When $\lambda=0$, for fixed $0<x_1<x_2$ that solves $(F_1(x_1,x_2;0,c),F_2(x_1,x_2;0,c))=(0,0)$, we know that $V_3=V_4$ is convex on $[0,+\infty)$ using the same argument as Proposition \ref{convex2} and $V_3^{\prime\prime}(x)=V_4^{\prime\prime}(x)>0, \forall x \in(0,x_2)$.

We only need to show the existence and uniqueness of the solution $0<x_1<x_2$ and \eqref{eq:jacobian-claim}.
 
At a zero $0<x_1<x_2$ of $(F_1(x_1,x_2;0,c),F_2(x_1,x_2;0,c))$, we have $V_3^{\prime}(x_2)=-c$ and $V_3^{\prime\prime}(x_2-)=0$. Thus, $V_3(x)=A \mathrm{e}^{\theta_3(x-x_2)}+B \mathrm{e}^{-\theta_4(x-x_2)}$ with $\theta_3A-\theta_4B=-c$ and $\theta_3^2A+\theta_4^2B=0$. Therefore, 
\begin{equation}
V_3(x)
=-\frac{c\theta_4}{\theta_3(\theta_3+\theta_4)}\mathrm{e}^{\theta_3(x-x_2)}
 +\frac{c\theta_3}{\theta_4(\theta_3+\theta_4)}\mathrm{e}^{-\theta_4(x-x_2)},
\qquad \forall x \in [x_1,x_2].
\label{eq:no-effort}
\end{equation}
Let $z=x_2-x_1$, by the condition $V_3^{\prime}(x_1)=-1$, we have 
\begin{equation}\frac{c\theta_4}{(\theta_3+\theta_4)}\mathrm{e}^{-\theta_3z}
 +\frac{c\theta_3}{(\theta_3+\theta_4)}\mathrm{e}^{\theta_4z}=1.\label{eq:gap}
\end{equation}
 Let $g(z)=\theta_4\mathrm{e}^{-\theta_3z}+\theta_3\mathrm{e}^{\theta_4z}$, then $g^{\prime}(z)>0, \forall z>0$, $g(0)=\theta_3+\theta_4$, and $g(+\infty)=+\infty$. It is easy to show that for any $0<c<1$, \eqref{eq:gap} has only one positive solution $z_c>0$ and $g(z_c)=\frac{\theta_3+\theta_4}{c}$. 
Define $z_0=\frac{2}{\theta_3+\theta_4}\ln\frac{\theta_4}{\theta_3}=x_2^{0}$, we have $g(z_0)=(\theta_3+\theta_4)\big(\frac{\theta_4}{\theta_3}\big)^{\frac{\theta_4-\theta_3}{\theta_3+\theta_4}}=\frac{\theta_3+\theta_4}{c_*(0)}$. Thus, when $c_*(0)<c<1$, we have $0<z_c<z_0$. Besides, since $0<z_c<z_0$, we have
\[\begin{aligned}
V_3(x_1)&=\frac{c}{\theta_3+\theta_4}\big(\frac{\theta_3}{\theta_4}\mathrm{e}^{\theta_4z_c}-\frac{\theta_4}{\theta_3}\mathrm{e}^{-\theta_3z_c}\big)<0,\\
V_3^{\prime\prime}(x_1)&=	\frac{c}{\theta_3+\theta_4}\big(\theta_3\theta_4\mathrm{e}^{\theta_4z_c}-\theta_3\theta_4\mathrm{e}^{-\theta_3z_c}\big)>0.
\end{aligned}
\] 

Use the backward coordinate $y=x_1-x$ in the effort region.
The function $Q(y)=V_3(x_1-y)$ must solve
\begin{equation}
\frac{\sigma^2}{2}Q^{\prime\prime}(y)-(\mu+M)Q^{\prime}(y)-\delta Q(y)+M=0,\quad \forall y \in[0,x_1], 
\quad Q(0)=V_3(x_1),\quad Q^{\prime}(0)=1.
\nonumber
\end{equation}
Solving these two initial conditions gives
\begin{equation}
Q(y)=\frac{M}{\delta}+C\mathrm{e}^{-\theta_1y}+D\mathrm{e}^{\theta_2y},
\nonumber
\end{equation}
where
\begin{equation}
C=\frac{\theta_2(V_3(x_1)-M/\delta)-1}{\theta_1+\theta_2}<0,
\qquad
D=\frac{\theta_1(V_3(x_1)-M/\delta)+1}{\theta_1+\theta_2}.
\nonumber
\end{equation}
Besides, since $Q^{\prime\prime}(0)=C\theta_1^2+D\theta_2^2=V_3^{\prime\prime}(x_1)>0$, we know that $D>0$. Thus, $Q^{\prime}(y)=-\theta_1C\mathrm{e}^{-\theta_1y}+\theta_2D\mathrm{e}^{\theta_2y}>0, \forall y\ge0$, $Q(0)<0$, and $Q(+\infty)=+\infty$. Hence, $Q$ has exactly one positive zero $x_1$. Set $x_2=x_1+z_c$, we have proved the uniqueness of the solution $0<x_1<x_2$. For $0<x_1<x_2$ we obtain above, on $[0,x_1]$, take $V_3(x)=Q(x_1-x)$, and on $[x_1,x_2]$, take $V_3$ by \eqref{eq:no-effort}, we have $V_3=V_4$ is indeed $V_3$ given by \eqref{V3} and $0<x_1<x_2$ is a zero of $(F_1(x_1,x_2;0,c),F_2(x_1,x_2;0,c))$, and the existence has been proven.

Finally, we prove \eqref{eq:jacobian-claim}. Fixed the solution $0<x_1<x_2$ just constructed. All parameter derivatives are first taken in the open set $\{(x_1,x_2)|0<x_1<x_2\}$ and then evaluated at this solution. Define $V_{3,+}(x;x_1,x_2)$ as \eqref{V3} on $[x_1,x_2)$ for any $x\in \mathbb{R}$, and $V_{3,-}(x;x_1,x_2)$ as \eqref{V3} on $[0,x_1)$ for any $x\in \mathbb{R}$, and
\[
Z_j^\pm(x)=\frac{\partial V_{3,\pm}(x;x_1,x_2)}{\partial x_j},
\qquad j=1,2.
\]
Thus for $j=1,2$, both $Z_j^-$ and $Z_j^+$ solve the corresponding
homogeneous second order linear differential equations and $Z_j^-(0)=0$. Since $V_{3,+}(x_1;x_1,x_2)=V_{3,-}(x_1;x_1,x_2)$, we have 
\[
Z_1^-(x_1)+V_{3}^{\prime}(x_1-)
=Z_1^+(x_1)+V_{3}^{\prime}(x_1+).
\]
Thus, $Z_1^-(x_1)=Z_1^+(x_1)$. Besides, we have
\[(Z_1^+)^{\prime}(x_1)-(Z_1^-)^{\prime}(x_1)=V_{3}^{\prime\prime}(x_1-)-V_{3}^{\prime\prime}(x_1+)=0.\]
Thus $Z_1^+$ and $Z_1^-$ has $C^1$ matching. Since $x_2$ is fixed when taking $\frac{\partial}{\partial{x_1}}$ and $V_{3,+}^{\prime}(x_2;x_1,x_2)=-c$, we have $(Z_1^+)^{\prime}(x_2)=0$. Define $Z_1(x)=Z_1^+(x), \forall x \in[x_1,x_2]$ and $Z_1(x)=Z_1^-(x), \forall x \in[0,x_1]$. Applying Lemma \ref{lem:comparision} to $Z_1$ and $-Z_1$, we know $Z_1(x)=0, \forall x \in[0,x_2]$.

Define $Z_2$ similarly, since $x_1$ is fixed when taking
$\frac{\partial}{\partial{x_2}}$, $V_{3,+}(x_1;x_1,x_2)=V_{3,-}(x_1;x_1,x_2)$, and $\frac{\partial V_{3,+}(x_1;x_1,x_2)}{\partial x}=\frac{\partial V_{3,-}(x_1;x_1,x_2)}{\partial x}$, we have 
\[
Z_2^-(x_1)=Z_2^+(x_1),\qquad
(Z_2^-)^{\prime}(x_1)=(Z_2^+)^{\prime}(x_1).
\]
Besides, by $\frac{\partial V_{3,+}(x_2;x_1,x_2)}{\partial x}=-c$, we have $(Z_2^+)^{\prime}(x_2)+V_{3}^{\prime\prime}(x_2)=0$. Thus $(Z_2^+)^{\prime}(x_2)=0$.  Again Lemma \ref{lem:comparision} yields $Z_2(x)=0, \forall x \in[0,x_2]$.
When $\lambda=0$, we have $F_1=V_{3,+}^{\prime}(x_1;x_1,x_2)+1, F_2=\frac{2}{\sigma^2}
 \big(\delta V_{3,+}(x_2;x_1,x_2)+\mu c\big)$.
The chain rule gives
\begin{equation}
\begin{aligned}
\nonumber
\frac{\partial F_1}{\partial x_1}
 &=(Z_1^+)^{\prime}(x_1)+V_{3}''(x_1)=V_{3}''(x_1)>0,\quad
\frac{\partial F_1}{\partial x_2}
 =(Z_2^+)^{\prime}(x_1)=0,\\
\frac{\partial F_2}{\partial x_1}
 &=\frac{2\delta}{\sigma^2}Z_1^+(x_2)=0,\quad
\frac{\partial F_2}{\partial x_2}
 =\frac{2\delta}{\sigma^2}
   \bigl(Z_2^+(x_2)+V_{3}^{\prime}(x_2)\bigr)
   =-\frac{2\delta c}{\sigma^2}.
\end{aligned}
\end{equation}
This completes the proof.
\end{proof}

Based on the above lemmas \ref{lem:comparision}-\ref{lem:exponential}, we have the following proposition for the existence of a solution $0<x_1<x_2$ to \eqref{eq2:x1x2}.

\begin{proposition}
\label{prop:existence}
For $0<\lambda<\delta$, if $c_*(\lambda)<c<1$, then there exist a solution $0<x_1<x_2$ to \eqref{eq2:x1x2}.
\end{proposition}

\begin{proof}
Fix $c$ with $c_*(\lambda)<c<1$, and define
\begin{equation}
\rho=\frac{1-c}{1-c_*(\lambda)}\in(0,1),
\qquad
c_\ell=1-\rho\big(1-c_*(\ell)\big),
\qquad \forall 0\le\ell\le\lambda.
\nonumber
\end{equation}
Then, for every $\ell\in[0,\lambda]$, we have 
\begin{equation}
c_*(\ell)<c_\ell<1,\qquad c_\lambda=c.
\label{eq:admissibility}
\end{equation}
Continuity of $c_*$ implies continuity of $\ell\mapsto c_\ell$. For each $\ell\in[0,\lambda]$, define
\[
\mathbb F_\ell(x_1,x_2)=(F_1,F_2)(x_1,x_2;\ell,c_\ell).
\]
Suppose that $\mathbb F_\ell(x_1,x_2)=(0,0)$ at a strict pair
$0<x_1<x_2$. Then by Lemma \ref{lem:UW-bounds}, we have 
\begin{equation}
x_2\le\frac{1+\ell/\delta}{\theta_1}
\le\frac{1+\lambda/\delta}{\theta_1}.
\label{eq:uniform-bound}
\end{equation}
Choose a fixed number $B$ and a bounded open set $O$ by
\begin{equation}
B>\frac{1+\lambda/\delta}{\theta_1},
\qquad
O=\{(x_1,x_2)\in\mathbb R^2|0<x_1<x_2<B\}.
\nonumber
\end{equation}
Every strict zero $0<x_1<x_2$ of every $\mathbb F_\ell$ therefore lies in $O$. And $((x_1,x_2),\ell)\longmapsto\mathbb F_\ell(x_1,x_2)$ is continuous on $\overline{O}\times[0,\lambda]$. At the collapsed corner $(0,0)$, $\mathbb F_{\ell}(0,0)=(1-c_\ell,\frac{2\mu c_\ell}{\sigma^2})\ne(0,0)$. We now exclude zeros on each part of $\partial O$.\\
(i).\smallskip\noindent
\emph{The zero-effort edge: $x_1=0$ and $0<x_2\le B$.} If $\mathbb F_{\ell,2}(0,x_2)=0$, we have $x_2=x_2^{\ell}>0$, and then $\mathbb F_{\ell,1}(0,x_2^{\ell})=1-\frac{c_\ell}{c_*(\ell)}<0$ by \eqref{eq:admissibility}. Thus the two components cannot vanish simultaneously.\\
(ii).\smallskip\noindent
\emph{The diagonal edge: $0\le x_1=x_2\le B$.} we have $\mathbb F_{\ell,1}(x_2,x_2)=1-c_\ell>0$.\\
(iii).\smallskip\noindent
\emph{The remaining top edge: $x_2=B$ and $0<x_1<B$.} A zero here would be a strict threshold solution, so \eqref{eq:uniform-bound} would imply $B\le\frac{1+\ell/\delta}{\theta_1}\le\frac{1+\lambda/\delta}{\theta_1}<B$, which is a contradiction. 

Therefore
\begin{equation}
0\notin\mathbb F_\ell(\partial O)
\qquad \forall \ell\in[0,\lambda].
\nonumber
\end{equation}
It follows that the Brouwer degree $\deg(\mathbb F_\ell,O,0)$ is well defined and independent of $\ell$ by homotopy invariance. At $\ell=0$, Lemma~\ref{lem:exponential} gives exactly one zero in $O$, and~\eqref{eq:jacobian-claim} shows that its local degree is $-1$.
Therefore 
\[
 \deg\big((F_1,F_2)(\cdot,\cdot;\lambda,c),O,0\big)
 =\deg(\mathbb F_{\lambda},O,0)=\deg(\mathbb F_0,O,0)=-1.
\]
The existence property of the Brouwer degree now gives a zero
$(x_1,x_2)\in O$ of $(F_1,F_2)(\cdot,\cdot;\lambda,c)$, which is a solution $0<x_1<x_2$ to \eqref{eq2:x1x2}.

\end{proof}

\begin{remark}
The proof in Proposition \ref{prop:existence} establishes existence under $c_*(\lambda)<c<1$. It does not require nondegeneracy of intermediate zeros or a continuous choice of a zero along the parameter path.	
\end{remark}

\begin{remark}
Using the same method as in the proof of Proposition \ref{prop:existence}, we can weaken the sufficient conditions in Proposition \ref{suff} as: Suppose that $\big[\sum\limits_{i=1}^{2}\omega_i\frac{\theta_{i,3}+\theta_{i,4}}{\theta_{i,3} \mathrm{e}^{\theta_{i,3}x_{2,0}}+\theta_{i,4}\mathrm{e}^{-\theta_{i,4}x_{2,0}}}\big]^{-1}<c<1$, where $x_{2,0}$ is the unique positive solution to \eqref{eq:x1x23}. Then, there exists a solution $0<x_1<x_2$ to \eqref{eq:x1x2}. The critical point lies in: keeping the weights fixed, for $\ell\in[0,1]$, construct $\delta_1(\ell)=\delta_1$ and $\delta_2(\ell)=\delta_1+\ell(\delta_2-\delta_1)$ to establish a connection between non-exponential discount and exponential discount. 	
\end{remark}

\begin{remark}
Under both a mixture of exponential discount functions and a pseudo-exponential discount function, the equilibrium policy prescribes exerting the maximal level of effort when the surplus falls below the effort threshold and paying dividends once the surplus exceeds the dividend threshold when the cost for dividend is relatively small; otherwise, when the cost for dividend is relatively large, the equilibrium policy is to pay dividends once the surplus exceeds the dividend threshold and to apply no effort even if the firm approaches ruin.
\end{remark}

\begin{remark}
Let 
\begin{eqnarray*}
&&\hat \Pi_a(x,t) := \begin{cases}
M,&0 \le x < x_1,\\
0,&x_1\leq x \leq x_2,\\
a,&x > x_2,
\end{cases}\quad\forall (x,t)\in\mathcal{\tilde{Q}},\\
&&W^{\hat{\Xi}}:=\big\{(x,t)\in\mathcal{\tilde{Q}}\bigm|0\leq x< x_2\big\},\ \ P^{\hat{\Xi}}:=\big\{(x,t)\in\mathcal{\tilde{Q}}\bigm|x\ge x_2\big\}.
\end{eqnarray*}
Suppose that there exist $0<x_1<x_2$ such that $V$ specified in \eqref{fv2} with \eqref{V3} and \eqref{V4} satisfies $V^\prime(x_1) = -1$
 (or equivalently, $V^{\prime\prime}(x_1-)=V^{\prime\prime}(x_1+)$) and $V^{\prime\prime}(x_2-) = b\in[0,\frac{2a(1-c)}{\sigma^2}]$.
 Although we have
 \begin{equation}
\begin{aligned}V\in C^{1,1}\big(\mathcal{\tilde{Q}}\big)\bigcap C^{2,1}\big(\big\{(x,t)\in\mathcal{\tilde{Q}}\bigm| x\in[0,x_2)\cup(x_2,+\infty)\big\}\big),\end{aligned}\nonumber
\end{equation}
the Mild Verification Theorem \ref{mildverificationthm} still holds if we use $V^{\prime\prime}(x_2-)$ when $x=x_2$. Analogous to the above discussion, we can show that, for any $a\in[0,M]$, $(\hat\Pi_a,\hat\Xi)$ is a mild equilibrium regular-singular control law and $V$ in \eqref{fv2}, \eqref{V3} and \eqref{V4} is the corresponding mild equilibrium value function for the objective \eqref{objexmple} under the pseudo-exponential discount function \eqref{eq:pseudo_exp_discount}. All the above-mentioned discussions can be directly transferred here. Due to the limitation of space, we will not elaborate on them any further.
\end{remark}
\vskip 8pt 
\section{Numerical analysis}\label{Numerical}
In this section, we conduct a numerical analysis for the effort and dividend problem \eqref{objexmple} in the case of the pseudo-exponential discount function \eqref{eq:pseudo_exp_discount}. For the case of the discount function \eqref{eq:mix_exp_discount}, readers can refer to the numerical section of \cite{hu2025equilibrium} regarding the case of two regular controls. Our results are similar,  thus we will not elaborate on them here.

\begin{figure}
	\centering
	{\includegraphics[width=0.8\linewidth]{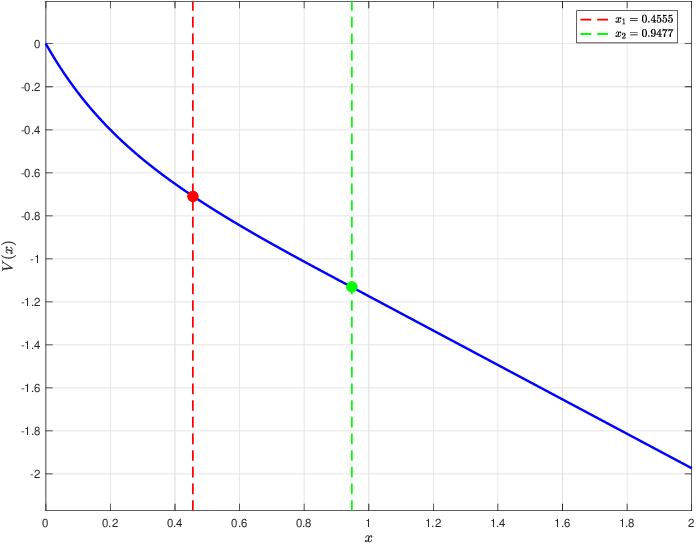}}
	\caption{The equilibrium value function $V$ for the baseline parameters.}\label{f1}
\end{figure}

Referring to \cite{hu2025equilibrium} and \cite{ZHAO20141}, we set $\mu=1$, $\sigma=1$, $M=1$, $c=0.8$, $\delta=0.8$, and $\lambda=0.1$ as the baseline parameters in our numerical analysis. For the baseline parameters, we can obtain that $x_1=0.4555$ and $x_2=0.9477$, which implies that the equilibrium regular-singular control law is to apply the maximal level of effort when the surplus is below the effort threshold $x_1=0.4555$ and to pay dividends once the surplus exceeds the dividend threshold $x_2=0.9477$. We show the corresponding equilibrium value function in Fig.~\ref{f1}. Note that the equilibrium value function is decreasing and convex in the surplus level $x$. Besides, the equilibrium value function is piecewise-defined over three regions and linear when $x\ge x_2$.

In the following, we study the impacts of exogenous parameters on the effort threshold $x_1$, dividend thresholds $x_2$, and the corresponding equilibrium value function.

\begin{figure}
	\centering
	{\includegraphics[width=1\linewidth]{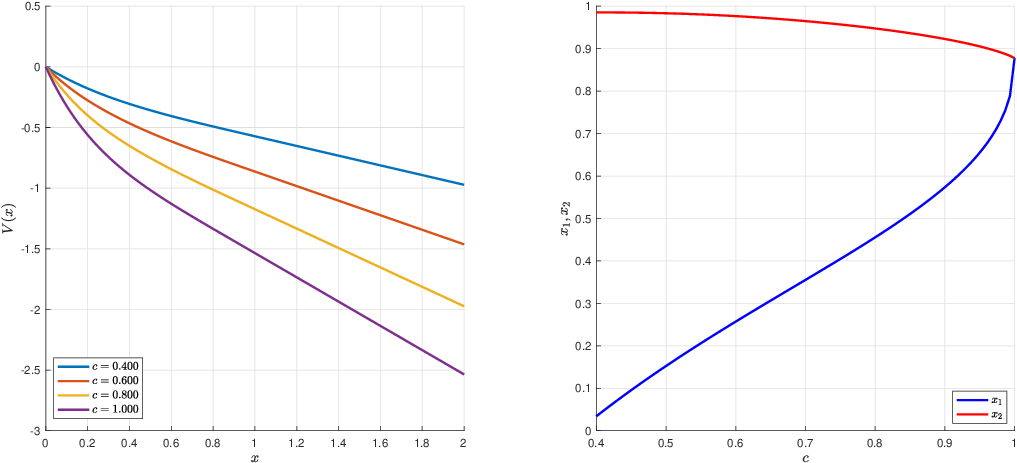}}
	\caption{The impacts of $c$ on the equilibrium value function $V$ (left panel) and thresholds $x_1, x_2$ (right panel).}\label{f2}
\end{figure}

In Fig.~\ref{f2}, we analyze the effect of the proportional parameter $c$ on the equilibrium value function $V$ and thresholds $x_1, x_2$ given that the other parameters are set to their baseline values. The left panel of Fig.~\ref{f2} illustrates the equilibrium value function with four representative curves corresponding to different values of the parameter $c$. The right panel of Fig.~\ref{f2} shows that how thresholds $x_1, x_2$ are varying with respect to $c$. As $c$ increases, the cost for dividend decreases, as a result, the effort threshold $x_1$ increases while the dividend thresholds $x_2$ decreases. Accordingly, the equilibrium value function decreases. Besides, when $c=1$, which means that there is no cost for dividend, the effort threshold $x_1$ is equal to the dividend thresholds $x_2$, and hence, the equilibrium regular-singular control law is either applying the maximal level of effort when the surplus is less than the threshold or paying dividends once the surplus is greater than the threshold.  

\begin{figure}
	\centering
	{\includegraphics[width=1\linewidth]{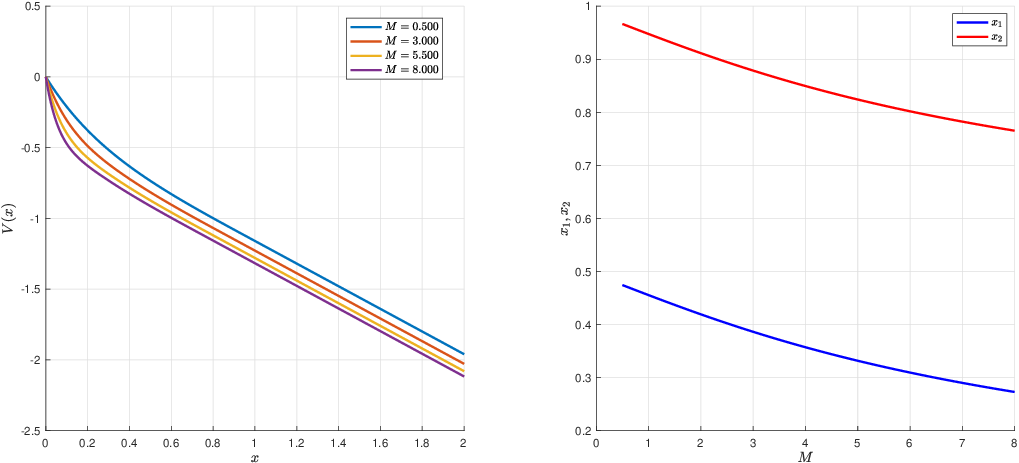}}
	\caption{The impacts of $M$ on the equilibrium value function $V$ (left panel) and thresholds $x_1, x_2$ (right panel).}\label{f3}
\end{figure}

In Fig.~\ref{f3}, we examine the influence of the maximal effort level $M$ on the equilibrium value function $V$ and thresholds $x_1, x_2$ given that the other parameters are set to their baseline values. The left panel of Fig.~\ref{f3} shows the equilibrium value function with four representative curves corresponding to different values of the parameter $M$. The right panel of Fig.~\ref{f3} analyzes that how thresholds $x_1, x_2$ are varying with respect to $M$. In particular, the effort threshold $x_1$ and the dividend thresholds $x_2$ both decrease in the maximal effort level $M$ and the equilibrium value function decreases in $M$. 

\begin{figure}
	\centering
	{\includegraphics[width=1\linewidth]{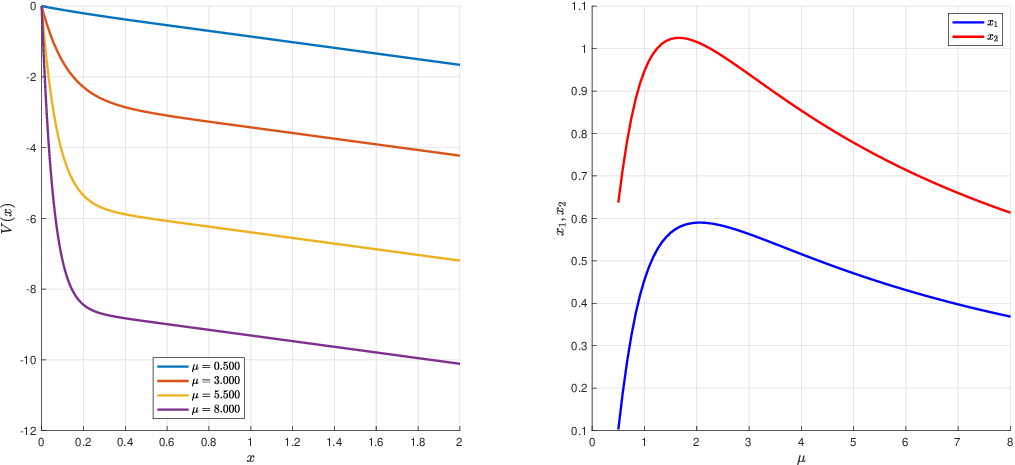}}
	\caption{The impacts of $\mu$ on the equilibrium value function $V$ (left panel) and thresholds $x_1, x_2$ (right panel).}\label{f4}
\end{figure}
 
In Fig.~\ref{f4}, we study the impact of the drift rate $\mu$ on the equilibrium value function $V$ and thresholds $x_1, x_2$ given that the other parameters are set to their baseline values. The left panel of Fig.~\ref{f4} studies the equilibrium value function with four representative curves corresponding to different values of the parameter $\mu$. The right panel of Fig.~\ref{f4} examines that how thresholds $x_1, x_2$ are varying with respect to $\mu$. As $\mu$ increases, the baseline expected growth rate of the surplus increases, resulting in a smaller equilibrium value function. The effort and dividend thresholds $x_1, x_2$ are first increasing then decreasing in $\mu$.

\begin{figure}
	\centering
	{\includegraphics[width=1\linewidth]{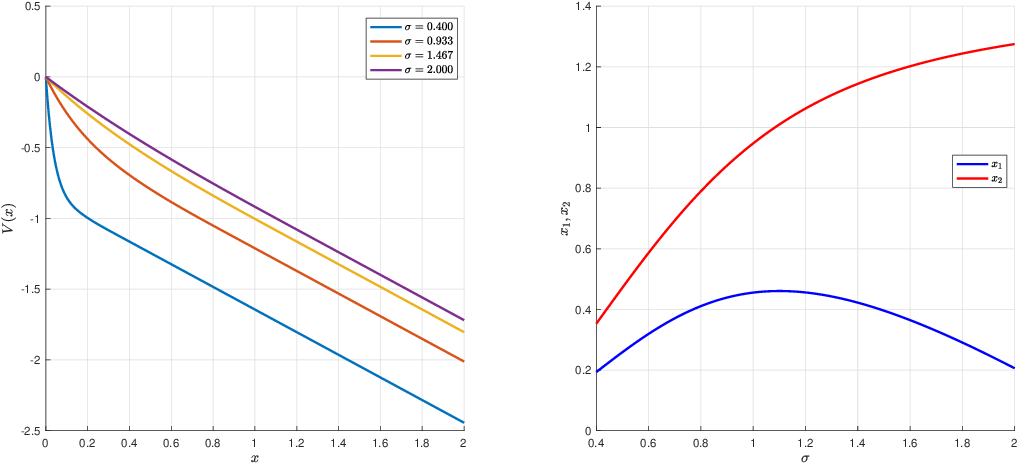}}
	\caption{The impacts of $\sigma$ on the equilibrium value function $V$ (left panel) and thresholds $x_1, x_2$ (right panel).}\label{f5}
\end{figure}

 In Fig.~\ref{f5}, we analyze the effect of the diffusion $\sigma$ on the equilibrium value function $V$ and thresholds $x_1, x_2$ given that the other parameters are set to their baseline values. The left panel of Fig.~\ref{f5} illustrates the equilibrium value function with four representative curves corresponding to different values of the parameter $\sigma$. The right panel of Fig.~\ref{f5} shows that how thresholds $x_1, x_2$ are varying with respect to $\sigma$. As $\sigma$ increases, the volatility of the surplus increases, leading to a larger equilibrium value function. The effort threshold $x_1$ is first increasing then decreasing in $\sigma$ and the dividend thresholds $x_2$ is increasing in $\sigma$. Besides, the difference between two thresholds getting larger in $\sigma$ as well.

\begin{figure}
	\centering
	{\includegraphics[width=1\linewidth]{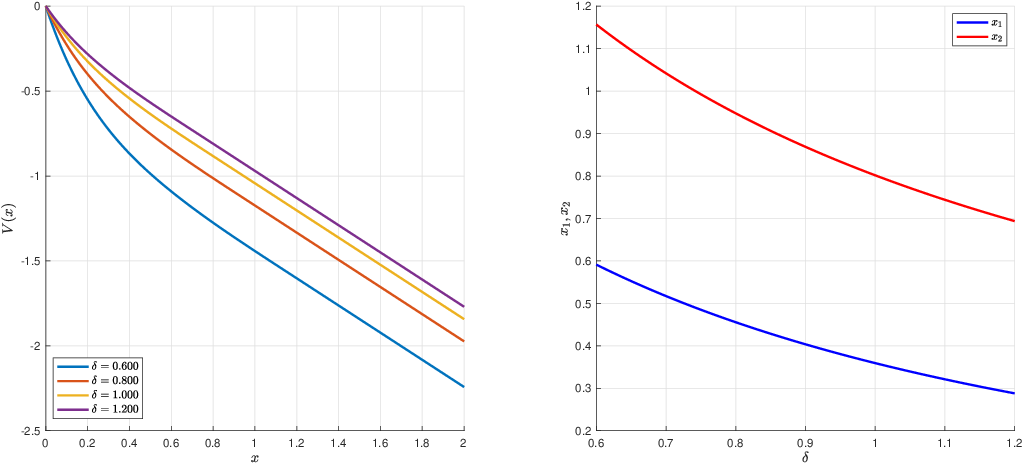}}
	\caption{The impacts of $\delta$ on the equilibrium value function $V$ (left panel) and thresholds $x_1, x_2$ (right panel).}\label{f6}
\end{figure}

In Fig.~\ref{f6}, we examine the influence of the discount rate $\delta$ on the equilibrium value function $V$ and thresholds $x_1, x_2$ given that the other parameters are set to their baseline values. The left panel of Fig.~\ref{f6} shows the equilibrium value function with four representative curves corresponding to different values of the parameter $\delta$. The right panel of Fig.~\ref{f6} analyzes that how thresholds $x_1, x_2$ are varying with respect to $\delta$. We find that the effort threshold $x_1$ and the dividend thresholds $x_2$ both decrease with the discount rate $\delta$ and the equilibrium value function increases with $\delta$.  This is because when $\delta$ increases, the discount function decreases, the equilibrium value function increases consequently. Meanwhile, the larger $\delta$, the more the discount is applied to future cashflows, and the thresholds for both effort and dividend are reduced to raise the ruin probability to seek for more dividends at present.

\begin{figure}
	\centering
	{\includegraphics[width=1\linewidth]{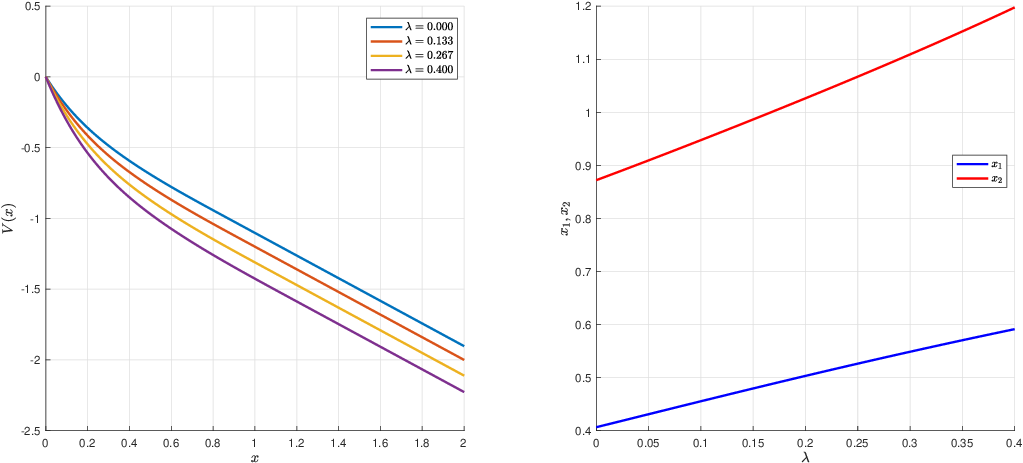}}
	\caption{The impacts of $\lambda$ on the equilibrium value function $V$ (left panel) and thresholds $x_1, x_2$ (right panel).}\label{f7}
\end{figure}

In Fig.~\ref{f7}, we study the impact of the parameter $\lambda$ on the equilibrium value function $V$ and thresholds $x_1, x_2$ given that the other parameters are set to their baseline values. The left panel of Fig.~\ref{f7} studies the equilibrium value function with four representative curves corresponding to different values of the parameter $\lambda$. The right panel of Fig.~\ref{f7} examines that how thresholds $x_1, x_2$ are varying with respect to $\lambda$. As $\lambda$ increases, the discount function increases, resulting in a smaller equilibrium value function. The effort and dividend thresholds $x_1, x_2$ are increasing in $\lambda$. The reason is that the larger $\lambda$, the less the discount is applied to future cashflows, and hence, the thresholds for both effort and dividend are raised to reduce the ruin probability to seek for more potential dividends in the future. Specially, the case with $\lambda = 0$ is time consistent counterpart.

\section{Conclusions}
In this paper, we investigate a time-inconsistent mixed regular-singular control problem in continuous time under non-exponential discount and propose a comprehensive equilibrium framework to address the inherent dynamic inconsistency. At the theoretical level, we seek time-consistent equilibria in an intrapersonal game setting and introduce a novel definition \ref{equil} of equilibrium regular-singular control law tailored to this class of problems for which the perturbations of regular and singular control are simultaneous, and establish a verification theorem \ref{verificationthm} that provides a sufficient condition for the equilibrium. An extended HJB system is derived with verification theorem. In addition, for the case where perturbations are applied separately to the regular and singular control, we propose the concept of mild equilibrium regular-singular control law (Definition \ref{mildequil}), together with its corresponding mild verification theorem \ref{mildverificationthm}. These theoretical results extend the existing literature on time-inconsistent control by accommodating both regular and singular components within a unified framework. At the applied level, we demonstrate the practical relevance of our theory through the effort and dividend problem. By solving the extended HJB system, we derive a semi-explicit equilibrium characterized by the effort and dividend thresholds under two representative discount structures: a mixture of exponential discount functions and a pseudo-exponential discount function. The convexity of the value function is rigorously established, the equilibrium conditions are confirmed, and the existence of the effort and dividend thresholds is proved. The equilibrium policy prescribes exerting the maximal level of effort when the surplus falls below the effort threshold and paying dividends once the surplus exceeds the dividend threshold when the cost for dividend is relatively small; otherwise, when the cost for dividend is relatively large, the equilibrium policy is to pay dividends once the surplus exceeds the dividend threshold and not to apply effort even if the firm is at the ruin time. The closed-form solutions validate the feasibility of our theoretical framework and provides economically intuitive insights for dynamic corporate finance decisions involving effort allocation and shareholder payouts. Numerical analysis further quantifies the effects of key exogenous parameters on the thresholds and the equilibrium value function under the pseudo-exponential discount function, thereby shedding light on their economic implications.

\paragraph{Acknowledgements.} The work is supported by the National Natural Science Foundation of China (12271290, 12671577). The authors thank the members of the group of Mathematical Finance and Actuarial Science at the Department of Mathematical Sciences, Tsinghua University, for their helpful feedback and conversations.

\paragraph{\bf Declaration of generative AI and AI-assisted technologies in the writing process.} During the preparation of this work, the authors used ChatGPT (OpenAI) in order to enhance the clarity and readability of the manuscript. After using this tool, the authors reviewed and edited the content as needed and take full responsibility for the content of the published article. 

\bibliographystyle{abbrvnat}
\bibliography{reference}

\end{document}